%% file: main-AM.tex
\documentclass[12pt]{amsart}
\usepackage{amsmath}
\usepackage{amssymb}
\usepackage{amsthm}
\usepackage{epsfig,graphicx}

\newcommand{\ignore}[1]{\relax}

\newcommand{\C}{\mathbb C}
\newcommand{\mo}{\mu}

\newcommand{\R}{\mathbb R}
\newcommand{\Z}{\mathbb Z}

\newcommand{\T}{\mathbb T}

\newcommand{\PP}{\mathcal P}

\newcommand{\ind}{\operatorname{ind}}
\newcommand{\lk}{\operatorname{lk}}

\newcommand{\Arg}{\operatorname{Arg}}

\newcommand{\ConvexHull}{\operatorname{Convex Hull}}
\newcommand{\lrot}{\operatorname{Rot}_{\operatorname{Log}}}
\newcommand{\larea}{\operatorname{Area}_{\operatorname{Log}}}

\newtheorem{lem}{Lemma}
\newtheorem{claim}{Claim}[section]

\newtheorem{theorem}[lem]{Theorem}
\newtheorem{corollary}[lem]{Corollary}
\newtheorem{thm}{Theorem}

\newtheorem{coro}[lem]{Corollary}
\newtheorem{prop}[lem]{Proposition}

\theoremstyle{definition}
\newtheorem{condition}[lem]{Condition}
\newtheorem{defn}[lem]{Definition}

\newtheorem{problem}[claim]{Problem}
\newtheorem{exa}[lem]{Example}

\theoremstyle{remark}
\newtheorem{rmk}[lem]{Remark}
\newtheorem{rem}[lem]{Remark}

\newcommand{\RR}{\mathcal R}
\newcommand{\QQ}{\mathcal Q}
\newtheorem{exam}[lem]{Example}

\newtheorem{ack}{Acknowledgment}

\renewcommand{\setminus}{\smallsetminus}

\newcommand{\ctor}{(\C^\times)^{2}}

\newcommand{\rtor}{(\R^\times)^{2}}

\newcommand{\ev}{\operatorname{ev}}

\newcommand{\conj}{\operatorname{conj}}

\newcommand{\dd}{\partial}
\newcommand{\am}{\mathcal{A}}

\newcommand{\cp}{{\mathbb C}{\mathbb P}}
\newcommand{\rp}{{\mathbb R}{\mathbb P}}
\newcommand{\tp}{{\mathbb T}{\mathbb P}}

\newcommand{\Log}{\operatorname{Log}}

\newcommand{\Vol}{\operatorname{Vol}}

\renewcommand{\setminus}{\smallsetminus}

\newcommand{\MM}{{\mathcal M}}

\newcommand{\Area}{\operatorname{Area}}

\newcommand{\OO}{\mathcal O}

\renewcommand{\Im}{\operatorname{Im}}

\newcommand{\Fr}{\operatorname{Sq}}

\newcommand{\ntropm}{N^{\partial,\operatorname{trop}}}

\ignore{
\newtheorem{theorem}{Theorem}

\newtheorem{condition}[theorem]{Condition}

\newtheorem{condition}{Balancing condition}
\newenvironment{proof}[1][Proof]{\noindent\textbf{#1.} }{\ \rule{0.5em}{0.5em}}
}

\newcommand{\ntrop}{N^{\operatorname{trop}}}

\begin{document}
\title[Quantum indices of real plane curves]{Quantum indices and refined enumeration
of real plane 
curves}
\author{Grigory Mikhalkin}
\address{Universit\'e de Gen\`eve,  Math\'ematiques, Battelle Villa, 1227 Carouge, Suisse}
\begin{abstract}
We associate a half-integer number, called {\em the quantum index}, to algebraic curves in the real plane
satisfying to certain conditions.
The area encompassed by the logarithmic image of such curves is equal to $\pi^2$ times the quantum index
of the curve and thus has a discrete spectrum of values.
We use the quantum index to refine
enumeration of real rational
curves in a way consistent with
the Block-G\"ottsche invariants from tropical enumerative geometry.
\end{abstract}
\maketitle
\input{s2a}
\input{s3a}
\begin{ack}
The results of this paper are inspired by a discussion
with Ilia Itenberg, Maxim Kontsevich and Ilia Zharkov 
on the Spring Equinox of 2015 before a traditional
IHES volleyball game as well as some discussions with Ivan Cherednik, Sergey Galkin and Yan Soibelman
on some other occasions. In particular, an example of
area computation for
holomorphic disks with boundary in a union of two Lagrangian subvarieties
explained by Maxim Kontsevich was especially helpful.
The author would like to thank everybody for a fruitful exchange
of ideas, as well as the IHES for its hospitality.
The paper has tremendously benefited
from many helpful remarks of the referee
to whom the author is sincerely thankful.

Research is supported in part by the grants 141329, 159240, 159581 and the NCCR SwissMAP
project of the Swiss National Science Foundation as well as by the Chaire d'Excellence
program of the Fondation Sciences Math\'ematiques de Paris.
\end{ack}

\bibliography{b}
\bibliographystyle{plain}

\end{document}

%% file: s2a.tex
\section{Introduction.}
\subsection{Quantum index}
Geometry of real algebraic curves in the plane is one of the most classical
subjects in Algebraic Geometry.

It is easy to see that
the logarithmic image $\Log(\R C^\circ)\subset\R^2$
of any real algebraic curve $\R C^\circ\subset\rtor\subset\R^2$
under the map \eqref{Log} bounds
a region of finite 
area in $\R^2$ (see Figures \ref{p}, \ref{rconics},
\ref{rcircles} for some examples of $\Log(\R C^\circ)$
in degrees 1 and 2). Furthermore, this area
is universally bounded from above for all curves
of a given degree by the Passare-Rullg\aa rd
inequality \cite{PaRu} for the area of amoebas.

E.g. if $\R C^\circ$ is a circle contained in the positive
quadrant $(\R_{>0})^2$ then
it bounds a disk $D\subset (\R_{>0})^2$, $\dd D=\R C^\circ$.
The area of the disk $D$ is $$\int\limits_Ddxdy=\pi r^2,$$
where $r$ is its radius. 
Clearly,
$\Area(D)$ may be arbitrarily large. 
In the same time it can be proved that
the area of $\Log D$ is 
$$\int\limits_D\frac{dx}x\frac{dy}y<\pi^2.$$
The inequality can be established either
through direct computation or as a corollary
of the Passare-Rullg\aa rd upper bound
on the area of amoeba \cite{PaRu}.
Thus this {\em logarithmic} area of $D$
stays bounded no matter how large
is the radius $r$.
In the same time it is clear that $\Area(\Log(D))$
can assume any value between 0 and $\pi^2$.

In this paper we impose the following conditions
on an algebraic curve $\R C\subset\rp^2$
(in the main body of the paper it is also formulated
for other toric surfaces in place of $\rp^2$)
so that such continuous behavior of the logarithmic
area is no longer possible.

Namely, we assume that $\R C$ is an irreducible
curve of type I (see subsection \ref{oldintro11}).
Then according to \cite{Rokhlin} $\R C$ 
comes with a canonical orientation (defined
up to simultaneous reversal in all components
of $\R C$). 
This enables us to consider
the signed area (with multiplicities) $\larea(\R C)$
bounded by $\Log(\R C^\circ)\subset\R^2$.
Unless one of the two possible complex orientations
of $\R C$ is chosen, $\larea(\R C)$ is only 
well-defined up to sign.

The curve $\R C\subset\rp^2$ is the zero set of an irreducible homogeneous
polynomial $f(x_0,x_1,x_2)$. 
For simplicity in the introduction we assume that $\R C$ is disjoint
from the points $(0:0:1)$, $(0:1:0)$, $(1:0:0)$.
The restriction of $f$ to $x_j=0$, $j=0,1,2$, is a homogeneous polynomial $f_j$
in two variables responsible for the intersection of $\R C$ with
the three coordinate axes of $\rp^2$. 
We say that $\R C$ has real or purely imaginary coordinate intersections
if for any (complex) zero $(z_a:z_b)$ of $f_j$ we have 
$(\frac{z_b}{z_a})^2\in\R$. 
%
Theorem 1 asserts that in this case $\larea(\R C)$
must be divisible by $\frac{\pi^2}2$ and therefore cannot
vary continuously.
The number $k=\frac{\larea(\R C)}{\pi^2}$ is thus 
a half-integer number naturally associated to 
the curve. We call it {\em the quantum index}
of $\R C$.
\setcounter{thm}{0}
\begin{thm}[special case for $\rp^2$]
Let $\R C\subset\rp^2$ be a real curve of degree $d$
and type I enhanced with a complex orientation.
If $\R C$ has real or purely imaginary coordinate intersection
then $$\larea(\R C)=k{\pi^2}$$
with $k\in\frac12\Z$ and
$-\frac{d^2}2\le k\le \frac{d^2}2.$
\end{thm}
To our knowledge this classical-looking result is new even in
the case $d=2$.
Meanwhile the special case of $d=1$ is well-known.
The identity $|\larea(\R C)|=\frac{\pi^2}2$ in the case of
lines was used by Mikael Passare \cite{Pa-zeta}
in his elegant new proof of Euler's formula
$\zeta(2)=\frac{\pi^2}6$. 
Another known special case of Theorem \ref{thm-larea} is
the case of the so-called {\em simple Harnack curves}
introduced in \cite{Mi00}. As it was shown in
\cite{MiRu} these curves have the maximal possible
value of $|\larea(\R C)|$ for their degree (equal
to $\frac{d^2}2\pi^2$).
Simple Harnack curves have many
geometric properties \cite{Mi00}.
By now these curves have appeared
in a number of situation outside of real algebraic geometry,
in particular in random perfect matchings of 
bipartite doubly periodic planar graphs of
Richard Kenyon, Andrei Okounkov
and Scott Sheffield \cite{KeOkSh}.
The quantum index of Theorem 1 can be interpreted
as a measure of proximity of a real curve to 
a simple Harnack curve.

Half-integrality of the quantum index $k$ may be
explained through appearance of $2k$ as the degree
of some map as exhibited in Proposition \ref{2kdeg}.
In accordance with this interpretation
Theorem \ref{thm-lrot} computes the quantum index 
through the degree of the real logarithmic Gau{\ss}
map of $\R C$.

Theorem 3 studies the quantum index in the special
case when $\R C$ is not only of type I, but
also of {\em toric type I}
(Definition \ref{def-ttI}).
This condition implies that all coordinate intersections
of $\R C$ are real. In this case the quantum index
may be refined to the {\em index diagram}
(Definition \ref{idiagram}), 
a closed broken lattice curve $\Sigma\subset\R^2$
well-defined up to a translation by $2\Z^2$.

The broken curve $\Sigma$ is an immersed
multicomponent curve with each component
corresponding to a component of the compactification
$\R\bar C$ of $\R C^\circ$
defined by its Newton polygon $\Delta$.
The complex orientation of $\R C$ induces an 
orientation of the closed broken curve $\Sigma$
so that we may compute the signed area
$\Area\Sigma$ inside $\Sigma$ which is 
a half-integer number as the vertices of $\Sigma$
are integer.
\setcounter{thm}{2}
\begin{thm}[simplified version]
If $\R C\subset\rp^2$ is a real algebraic curve of toric type I
enhanced with a choice of its complex orientation
then
its quantum index $k$ coincides with $\Area\Sigma.$
\end{thm}

Each edge of $\Sigma$ corresponds to an intersection
of $\R \bar C$ with a toric divisor
of the toric variety $\R\Delta$ corresponding to
the Newton polygon $\Delta$ and thus to a side
$E\subset\dd\Delta$. If this intersection is transversal
then the corresponding oriented edge of $\Sigma$
is given by the primitive integer outer normal
vector ${\vec n}(E)$. More generally it is given
by  ${\vec n}(E)$ times the multiplicity of the intersection.
This makes finding the index diagram $\Sigma$
and thus the quantum index $k$ especially easy
at least in the case of rational curves with 
real coordinate intersections
(cf. e.g. Figures \ref{rconics} and \ref{c-diagrams}).

The index diagram $\Sigma$ can be viewed as
a non-commutative version of the Newton polygon $\Delta$:
it is made from the same elements (the vectors ${\vec n}(E)$
taken $\#(E\cap\Z^2)-1$ times) as $\dd\Delta$, but
the real structure on $\R\bar C$ gives those pieces
a cyclic order (in the case of connected $\R \bar C$)
or divides these elements into
several cyclically ordered subsets.

Recall that Mikael Forsberg, Mikael Passare and August Tsikh
in \cite{FoPaTs} have defined the amoeba index map
which is a locally constant map on
$\R^2\setminus\am$, the complement
of the {\em amoeba} $\am=\Log(\C C^\circ)$
of the complexification $\C C^\circ$ of $\R C^\circ$.
Each connected component of $\R^2\setminus\am$
is associated a lattice point of the Newton polygon $\Delta$.

For toric type I curves the formula \eqref{indR} defines
the real index map so that each connected
component of the normalization $\R\tilde C^\circ$
or a solitary real singularity of $\R C^\circ$
acquires a real index which is
a lattice point of the convex hull
of the index diagram $\Sigma$.
Theorem \ref{thm-indAR} 
computes the amoeba index map $\ind$
in terms of the linking number with
the curve $\R C^\circ$ enhanced with the real indices.

\subsection{Refined real enumerative geometry in the plane}
The second part of the paper is devoted to applications
of the quantum index of real curves introduced 
in this paper to enumerative geometry over complex
and real numbers.
The space of planar projective rational curves of degree $d$
is $(3d-1)$-dimensional. Thus given a {\em generic}
configuration $\PP$ of $3d-1$ 
points in the projective plane we expect a finite set
${\mathcal S}_d$ of such curves.
What we can do next with this set
depends on our choice of ground field.

Our two main choices are the fields $\C$ and $\R$
of complex and real numbers.
For both of these cases we choose
$\PP\subset\rp^2$ 
generically
and denote with ${\mathcal S}^{\C}_d$ 
(resp., ${\mathcal S}^{\R}_d$)
the finite set of all planar projective rational
curves of degree $d$
defined over $\C$ (resp., over $\R$) passing through $\PP$.
It is easy to see that the cardinality 
$N_d^{\C}=\#({\mathcal S}^{\C}_d)$ does not
depend on the choice of $\PP$ (even if $\PP$ is
chosen generically in $\cp^2$ rather than $\rp^2$).
In the same time the cardinality
$\#({\mathcal S}^{\R}_d)$ depends on the
choice of generic configuration $\PP$ and
{\em a priori} only the parity of this set 
remains invariant.
 
According to the seminal result of Welschinger \cite{We}
the curves $\R C\in{\mathcal S}^{\R}_d$ come with natural
{\em signs} $w(\R C)=\pm 1$ so that the integer number
$$N^{\R}_d=\sum\limits_{\R C\in {\mathcal S}^{\R}_d}w(\R C)$$
is invariant of the choice of $\PP$.
The number $N^{\R}_d$ is thus known as the {\em Welschinger
invariant} and is the fundamental notion
of real enumerative geometry. 
Itenberg, Kharlamov and Shustin in \cite{IKS1} 
have established non-trivial lower bounds on 
$\#({\mathcal S}^{\R}_d)$ with the help of $N^{\R}_d$. 

Both integer numbers $N^{\C}_d$ and $N^{\R}_d$
were simultaneously computed with the help of 
passing to the tropical limit in \cite{Mi05}.
Namely, $N^{\C}_d$ and $N^{\R}_d$ can be presented
as sums of multiplicities of corresponding tropical curves
passing through a generic configuration of points
in the tropical plane. The tropical curves
are the same in both cases,
however the rules for
defining their $\C$ and $\R$ multiplicities
are different,
so the sums  $N^{\C}_d$ and $N^{\R}_d$ are 
different as well.

With the help of this presentation
Block and G\"ottsche in \cite{BlGo} have proposed
combining the numbers $N^{\C}_d$ and $N^{\R}_d$
into a single number $\ntrop_d$,
which is no longer an integer number, but
an integer $q$-number
(a Laurent polynomial in $q$ with positive integer coefficients
invariant under the substitution $q\mapsto\frac 1q$).
The value at $q=1$ is capable
to recover the number of complex curves while
the value at $q=-1$ should be capable to recover
the number of real curves in the same enumerative problem.
E.g. there are $q+10+q^{-1}$ many of rational
cubic curves passing through 8 generic points in $\rp^2$.
In the same time there are 12 curves over $\C$ 
and 8 curves over $\R$ (if we count real curves
with the Welschinger sign \cite{We}).

Conjecturally (see \cite{GoSh}) the $q$-refinement
$\ntrop_d$ of the integer number $N^{\C}_d$ agrees
with the $\chi_y$-genus refinement of Severi degrees
proposed by G\"ottsche and Shende in \cite{GoSh}.
Also this refinement looks to be at least vaguely resemblant
of the refinements of Donaldson-Thomas invariants
considered by Kontsevich and Soibelman
\cite{KoSo} and Nekrasov and Okounkov \cite{NeOk}
in some other frameworks (in particular, for 3-folds).

The quantum index allows us to obtain a refined enumeration
of planar curves entirely within classical
real algebraic geometry of the plane
with the help of Theorem \ref{Rinv}.
Once again for simplicity
we discuss only the special case of the projective
plane here in the introduction while in the main body of
the paper the theorem is formulated for other toric surfaces
as well.

Recall that the space of rational curves of degree $d$
in $\rp^2$ is $(3d-1)$-dimensional.
Thus we expect a finite number of such curves
if we impose on them $3d-1$ conditions.
Let us choose a generic configuration of $3d-1$ points
on the three coordinate axes of $\rp^2$ (the $x$-axis,
the $y$-axis and the $\infty$-axis) so that each axis
contains no more than $d$ points: e.g. there are $d$
generic points on the $x$- and $y$-axis and $3d-1$
generic points on the $\infty$-axis.
The elementary generalization
of the classical Menelaus theorem
(see Figure \ref{FigMenelaus}) found already by Carnot
\cite{Carnot} (later further generalized as the
{\em Weil reciprocity law})
ensures that there is a unique $3d$th point on the
$\infty$-axis such that any irreducible curve of degree $d$
passing through our $3d-1$ points also passes through
the $3d$th point. The resulting
configuration of $3d$ points
on the union 
of three coordinate axes
varies in a $(3d-1)$-dimensional family
of {\em Menelaus configurations}.

We define the square map $\Fr:\cp^2\to\cp^2$
by $\Fr(z_0:z_1:z_2)=(z_0^2:z_1^2:z_2^2)$.
An irreducible rational curve $\R C\subset\rp^2$
such that $\Fr(\C C)$ passes through 
a Menelaus configuration $\PP$
is of type I and has real or purely imaginary
coordinate intersection. Thus the quantum index
$k$ is well-defined.

In \eqref{Rk} we define
$R_{d,k}(\PP)$ (here we write $d$ instead
of $\Delta$ since we restrict ourselves
to the special case of $\rp^2$
in the introduction)
as one quarter of the number of
irreducible oriented rational curves $\R C\subset\rp^2$
of degree $d$ and quantum index $k$ such that
$\Fr(\C C)$ passes through $\PP$.
Each curve $\R C$ here is taken with the sign
\eqref{sigma} which is a modification of
the Welschinger sign \cite{We}.
Note that such curves come in quadruples
thanks to the action of the deck transformations
of the 4-1 covering $\Fr|_{\rtor}:\rtor\to\rtor$.
This is the reason for including $\frac 14$ 
in the definition of $R_{d,k}(\PP)$.
The points of $\PP$ contained in the closure
of the positive quadrant $(\R_{>0})^2$ (positive points)
correspond to real coordinate axes intersections
of $\R C$, other
(negative) points correspond to
purely imaginary coordinate axes intersections.

The image of each component of $\C C\setminus\R C$
under $\Fr$ may
be viewed as an open holomorphic disk $F$ in $\cp^2$ 
with the boundary contained in the closure
$L=\overline{(\R_{>0})^2}$ of the positive quadrant.
The subspace $L\subset\cp^2$
is a {\em Lagrangian submanifold
with boundary}. The positive points
of $\PP$ correspond to tangencies of $\dd F$ and $\dd L$
while the negative points of $\PP$
correspond to intersections of the open disk $F$
with the coordinate axes of $\rp^2$ away from $\dd L$.
From this viewpoint $R_{d,k}(\PP)$ is the number
of certain holomorphic disks whose boundary is
contained in $L$, a framework widely used in
symplectic geometry. An unconventional feature 
is presence of boundary in the contractible 
(and thus orientable) Lagrangian surface $L$.
Positive points of $\PP$ are contained
in the boundary $\dd L$. The holomorphic disks
are tangent to $\dd L$ at these points.
Negative points of $\PP$ are disjoint from $L$
and and thus from the boundaries of the holomorphic disks. 

The number of negative points on three
coordinate axes is given by
$\lambda=(\lambda_1,\lambda_2,\lambda_3)$
with $\lambda_j\le d$.
\setcounter{thm}{4}
\begin{thm}
The number $R_{d,k,\lambda}=R_{d,k}(\PP)$
is invariant of the choice of $\PP$ and depends only on $d$,
$k$ and $\lambda$.
\end{thm}
In particular, $R_{d,k}(\PP)$
depends only on $d$ and $k$
when all points of $\PP$ are positive.

For a positive point $p\in\PP$ the inverse image
$\Fr^{-1}(p)$ consists of two points:
a positive point $p_+\in\dd L$
and a negative point $p_-\notin\dd L$.
The condition $\Fr(\C C)\ni p$ 
is equivalent to the condition that $\R C$ passes through
$p_+$ or $p_-$.
Note that the invariance claimed in Theorem 5 relies
on including
into $R_{d,k}(\PP)$ both of these possibilities.
If we leave out only the curves passing through 
$p_+$ (or $p_-$) as in \eqref{Rtk}
then the resulting sum $\tilde R_{d,k}(\PP)$ is no longer
invariant under deformations of $\PP$.
Nevertheless, a partial invariance result
for $\tilde R_{d,k}(\PP)$ is provided
by Theorem \ref{tildeR}. 

The generating function
$R_d(\lambda)=\sum\limits_k R_{d,k,\lambda}q^k$
defined in \eqref{Rd}
is a Laurent polynomial in $q^{\frac 12}$.
As such it can be compared with the modification
$\ntropm_d$ of the Block-G\"ottsche
refined tropical invariants $\ntrop_d$
where we take for $\PP$ a generic Menelaus configuration
of points in the boundary $\dd\tp^2=\tp^2\setminus\R^2$
rather than a generic configuration of points
in $\R^2$. Namely, the number $\ntropm_d$ is given by
\eqref{BGDeltadef}, where $\Delta$ is a triangle with
vertices $(0,0)$, $(d,0)$ and $(0,d)$.
The last theorem of the paper is an identity
between $R_d=R_d(0,0,0)$ 
and $\ntropm_d$.
\setcounter{thm}{6}
\begin{thm}[special case of $\rp^2$]
$$R_d= (q^{\frac12}-q^{-\frac12})^{3d-2}\ntropm_d.$$
\end{thm}

This theorem has a surprising corollary. 
As the number $N^{\dd,\C}_d$ of irreducible
rational complex curves $\C C\subset\cp^2$
of degree $d$ passing through $\PP$
coincides with $\ntropm_d(1)$,
this number is completely determined by $R_d$,
the number accounting only for curves defined over $\R$.
Note that for this purpose it is crucial
to use the quantum refinement by $q^k$
since for $q=1$ we would have to divide by $0$
(the value of $(q^{\frac12}-q^{-\frac12})^{3d-2}$
at $q^{\frac 12}=1$) to recover $\ntropm_d(1)$.
   

\section{Conventions and notations}
\setcounter{thm}{0}
\subsection{Real curves of type I and their complex orientation}
\label{oldintro11}
A real curve $\R C\subset\rp^2$ is given by a single homogeneous polynomial equation
$F(z_0,z_1,z_2)=\sum\limits_{j,k,l}a_{k,l}z_0^jz_1^kz_2^l=0$,
$j+k+l=d$, $a_{k,l}\in\R$.
The locus $\C C\subset\cp^2$ of complex solutions of $F=0$
is called the {\em complexification} of $\R C$.
We assume $F$ to be irreducible
over $\C$ and such that $\C C$ does not coincide
with a coordinate axis $\{z_j=0\}$, $j=0,1,2$.
The normalization
\begin{equation}\label{nu}
\nu:\C\tilde C\to\C C
\end{equation}
defines a parameterization of $\C C$ by
a Riemann surface $\C \tilde C$. The antiholomorphic
involution of complex conjugation
$\conj$ acts on $\C C$ in an orientation-reversing way
so that its fixed point locus is $\R C$. The restriction
of $\conj$ to the smooth locus of $\C C$
lifts to an antiholomorphic involution $\tilde\conj:\C\tilde C\to\C \tilde C$
on the normalization. We denote with $\R \tilde C$ the fixed point
locus of $\tilde\conj$. Clearly, $\nu (\R \tilde C)\subset\R C$.
Irreducibility of $\C C$ is equivalent
to connectedness of $\C \tilde C$.

Following Felix Klein we say that $\R C$ is {\em of type I} if
$\C \tilde C\setminus \R \tilde C$ is disconnected. In such case
it consists of two connected components $S$ and $S'=\tilde\conj(S)$
which are naturally oriented by the complex orientation of
the Riemann surface $\C \tilde C$.
We have $\R\tilde C=\dd S=\dd S'$, so
a choice of one of these components, say $S$, induces the boundary orientation on
$\R \tilde C$. The resulting orientation is called {\em a complex orientation of $\R \tilde C$}
and is subject to Rokhlin's complex orientation formula \cite{Rokhlin}.
If we choose $S'$ instead of $S$
then the orientations of all components of $\R C$
will reverse simultaneously.
Thus any orientation of a component of $\R \tilde C$ determines a component
of $\C \tilde C\setminus\R \tilde C$.

\subsection{Toric viewpoint and reality of coordinate intersections}
The projective plane $\cp^2$ can be thought of
as the toric compactification
of the torus $\ctor$.
The curve $\C C\subset\cp^2$ is the closure of its toric
part $\C C^\circ=\C C\cap\ctor$.
The complement $\dd\cp^2=\cp^2\setminus\ctor$
is the union of three axes: the $x$-axis, the $y$-axis and the $\infty$-axis.
These axes intersect pairwise at the points
$(1:0:0), (0:1:0), (0:0:1)\in\rp^2$.
If the coefficients $a_{0,0}, a_{d,0}, a_{0,d}$ are non-zero
then $\C C$ is disjoint from the intersection points
of the axes.
In the general case it is reasonable to consider
other toric surfaces compactifying $\ctor$,
so that the closure of $\C C^\circ$ is disjoint
from pairwise intersections of toric divisors.

Let us
consider the (non-homogeneous)
polynomial $f(x,y)=F(1,x,y)$ and its Newton polygon
$$\Delta=\operatorname{Convex Hull}\{(j,k)\in\R^2\ |\ a_{j,k}\neq 0\}.$$
If $\Delta$ has non-empty interior then the
dual fan to $\Delta$ defines a toric compactification $\C\Delta\supset\ctor$.
The toric divisors of $\C \Delta$ correspond to the sides of $\Delta$.
Their pairwise
intersections correspond to the vertices of $\Delta$
and are disjoint from the compactification $\C \bar C$ of the curve $\C C^\circ$.
We denote with $\dd\C \Delta\subset\C \Delta$ the union of toric divisors.
Accordingly, we denote with $\R\Delta$ (resp. $\dd\R \Delta$, $\R C^\circ$, $\R \bar C$)
the real part of $\C \Delta$  (resp. $\dd\C \Delta$, $\C C^\circ$, $\C \bar C$).
E.g. we have $\rp^2=\R \Delta$ for the triangle
$\Delta=\operatorname{Convex Hull}\{(0,0),(1,0),(0,1)\}$
or a positive integer multiple of this triangle.

Let $\Fr:\ctor\to\ctor$ be the map defined
by $\Fr(x,y)=(x^2,y^2)$.
This map extends to a map $\Fr^\Delta:\C \Delta\to\C \Delta$.

We call a point $p\in\C \Delta$ {\em real or purely imaginary} if
$\Fr^\Delta(p)\in\R\Delta$.
We say that a curve $\R C\subset\rp^2$ has {\em real or purely
imaginary coordinate intersection}
if
every point of $\C C\cap\dd\C \Delta$ is real or purely imaginary.

\subsection{Logarithmic area and other numbers associated to a real curve of type I}
Let $\R C$ be a real curve of type I enhanced with a choice of a complex orientation.
Consider the image $\Log(\R C^\circ)\subset\R^2$, where
$\Log:\ctor\to\R^2$ the map defined by
\begin{equation}\label{Log}
\Log(x,y)=(\log|x|,\log|y|).
\end{equation}
For a point $p\in\R^2\setminus\Log(\R C^\circ)$ we define
$\ind(p)\in\Z$ as the intersection number of an oriented ray $R\subset\R^2$
emanating from $x$ in a generic direction and the oriented curve
$\Log(\R C^\circ)$
(this number can be considered as the linking number
of $p$ and $\Log(\R C^\circ)$).

\begin{defn}
The integral
$$\larea(\R C)=\int\limits_{\R^2}\ind_{\R C}(x)dx$$
is called the {\em logarithmic area} of $\R C$.
\end{defn}
This is the signed area encompassed by $\Log(\R C^\circ)$,
where the area of each region of $\R^2\setminus\Log(\R C^\circ)$
is taken with the multiplicity
equal to the linking number of $\Log(\R C^\circ)$.

Let $S\subset\C\tilde C\setminus\R \tilde C$
be the component corresponding
to the chosen complex orientation of $\R \tilde C$.
The intersection points $\nu(S)\cap\rtor$ are the so-called
{\em solitary real singularities} of $\R C^\circ$.
The {\em multiplicity} of a solitary real singularity
$p\in \nu(S)\cap\rtor$ is the local
intersection number of $S$ and $\rtor$ at $p$.
Here the orientation of $S$ is induced by
the inclusion $S\subset\C \tilde C$,
while the orientation of $\rtor$ is induced by the covering $\Log|_{\rtor}:\rtor\to\R^2$.
In other words, the quadrants $\R_{>0}^2$ and $\R_{<0}^2$ are counterclockwise-oriented
while the quadrants $\R_{>0}\times\R_{<0}$ and $\R_{<0}\times\R_{>0}$
are clockwise-oriented.
{\em The toric solitary singularities number}
$E(\R C)\in\Z$ is the sum of multiplicities over all
solitary real singularities of $\R C^\circ$,
i.e. the total intersection number of $S$ and $\rtor$
(enhanced with our choice of orientation).

The logarithmic Gau{\ss} map 
sends a smooth point of $\R C^\circ$ to the tangent direction of the corresponding point
on $\Log(\R C)\subset\R^2$. This map uniquely extends to a
map $$\gamma:\R\tilde{C}\to\rp^1,$$ cf. \cite{Ka-Gauss}, \cite{Mi00}.
{\em The logarithmic rotation number} $\lrot(\R C)\in\Z$
is the degree of $\gamma$.

\section{Quantum indices of real curves.}
\setcounter{thm}{0}
\begin{thm}\label{thm-larea}
Let $\R C\subset\rp^2$ be a real curve of type I enhanced with a complex orientation.
If $\R C$ has real or purely imaginary coordinate intersection
then $$\larea(\R C)=k{\pi^2}$$
where $k\in\frac12\Z$, $$-\Area(\Delta)\le k\le \Area(\Delta)$$
and $k\equiv\Area(\Delta)\pmod 1$.
\end{thm}

Note that as $\Delta\subset\R^2$ is a lattice polynomial,
its area is a half-integer number.

\begin{defn}
We say that $k(\R C)=\frac{1}{\pi^2}\larea(\R C)$ is
{\em the quantum index} of $\R C$.
\end{defn}
If $\R C$ is an irreducible real curve of type I with real
or purely imaginary coordinate intersection,
but the complex orientation of $\R C$ is not fixed then its quantum index is
well-defined up to sign.

The quantum index $k(\R C)$
can also be expressed without computing
the logarithmic area.
\begin{prop}\label{2kdeg}
The integer number $2k(\R C)$ coincides with the
degree of the map
$$
2\Arg:\C C^\circ\setminus\R C^\circ\to(\R/\pi\Z)^2,
$$
i.e. with the number of inverse images at a generic point of the
torus $(\R/\pi\Z)^2$ counted with the sign according to the orientation.
(In particular, this number does not depend on the choice
of a point in $(\R/\pi\Z)^2$ as long as this choice
is generic.)
Here the orientation of
$\C C^\circ\setminus\R C^\circ$ is defined by the condition that it
coincides with the complex orientation of $\C C$
on the component $S\subset\C C^\circ\setminus\R C^\circ$ determined by the orientation of $\R C$
and is opposite to the complex orientation of $\C C$ on the component
$\conj(S)\subset\C C^\circ\setminus\R C^\circ$.
The map $2\Arg$ is defined by $2\Arg(x,y)=(2\arg(x),2\arg(y)).$
\end{prop}
We say that $\R \bar C$ is {\em transversal to $\dd \R\Delta$}
if for any $p\in\R C\cap\dd\R \Delta$ we have $\nu^{-1}(p)\subset \R \tilde C$
and
the composition $\R \tilde C{\to}\R\bar C\subset\R\Delta$ is an immersion
near $\nu^{-1}(p)\subset \R \tilde C$,
and this immersion is transversal to $\dd \R\Delta$.
\begin{thm}\label{thm-lrot}
Let $\R C$ be a curve of type I with real
or purely imaginary coordinate intersections
such that $\R \bar C$ is transversal to $\dd\R \Delta$.
Then
$$k(\R C)=-\frac12\lrot(\R C)+E(\R C).$$
\end{thm}
If $\R \bar C$ is not transversal
to $\dd\R \Delta$ then an adjustment of the
right-hand side according to the order of tangency and the orientation of $\R C$
should be added to the formula of Theorem \ref{thm-lrot}.
\begin{exam}[Simple Harnack curves]
If $\R C^\circ\subset\rtor$ is a simple Harnack curve (see \cite{Mi00})
then $k(\R C)=\pm\Area(\Delta)$.
Vice versa, if $k(\R C)=\pm\Area(\Delta)$ then $\R C^\circ$ is a simple
Harnack curve, see \cite{MiRu}.
This characterizes real curves of the highest and lowest quantum index.
\end{exam}
\begin{exam}[Quantum indices of real lines]
Any real line is a curve of type I and has real coordinate intersections.
The quantum index of a real line in $\rp^2$ disjoint from the points
$(1:0:0),(0:1:0),(0:0:1)$ is $\pm\frac12$ (depending on its orientation), see Figure \ref{p}.
The quantum index of a line passing through exactly one of these points is 0.
\begin{figure}[h]
\includegraphics[height=20mm]{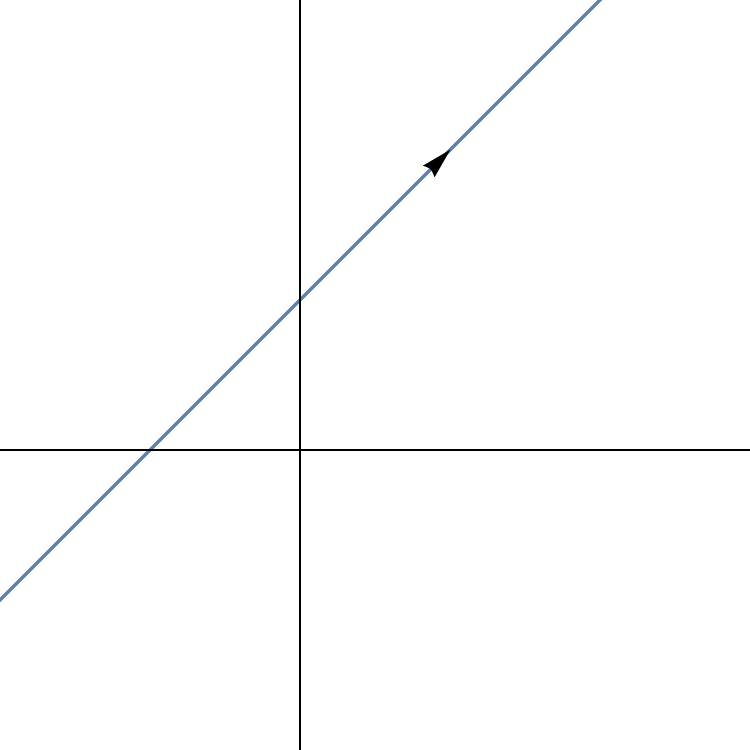}\hspace{10mm}
\includegraphics[height=20mm]{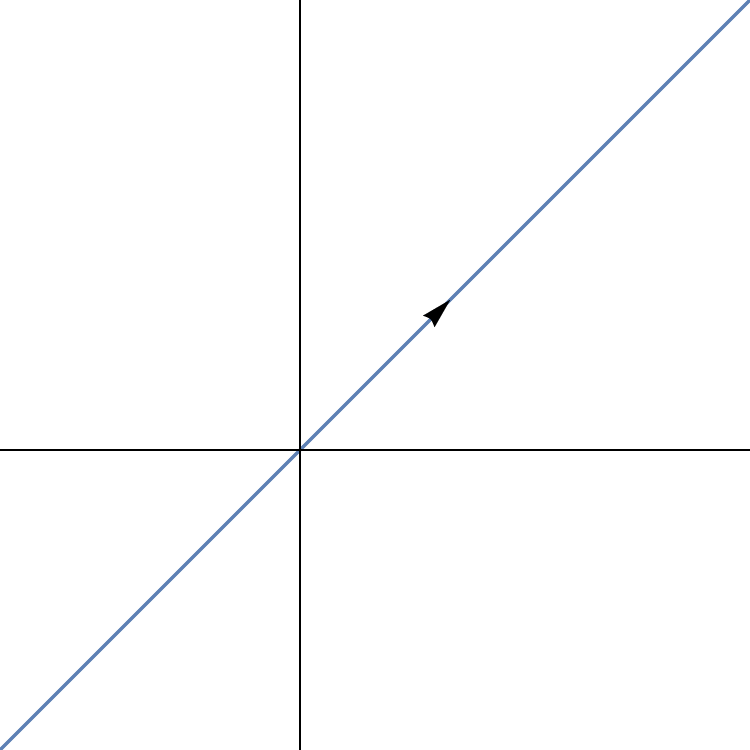}\hspace{10mm}
\includegraphics[height=20mm]{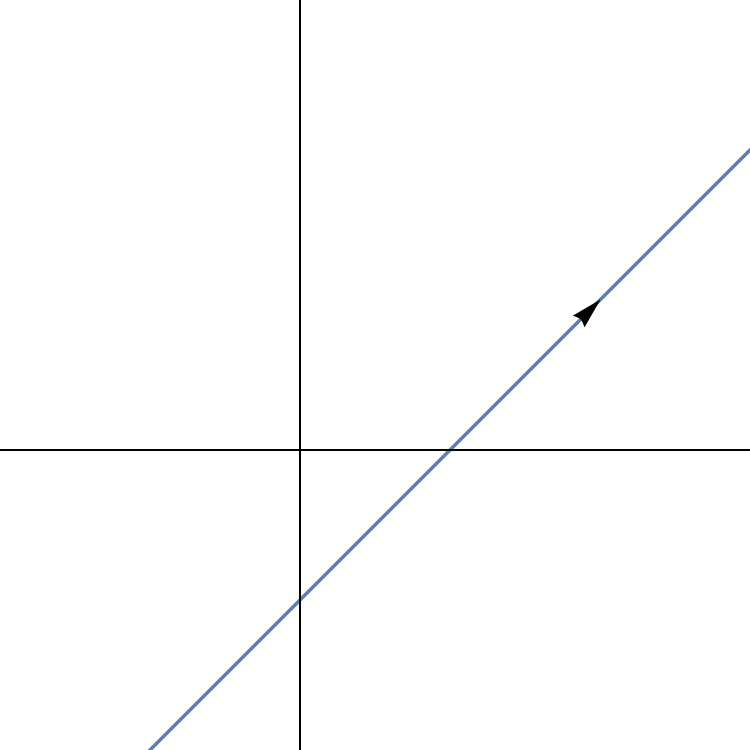}\vspace{2mm}\\
\includegraphics[height=28mm]{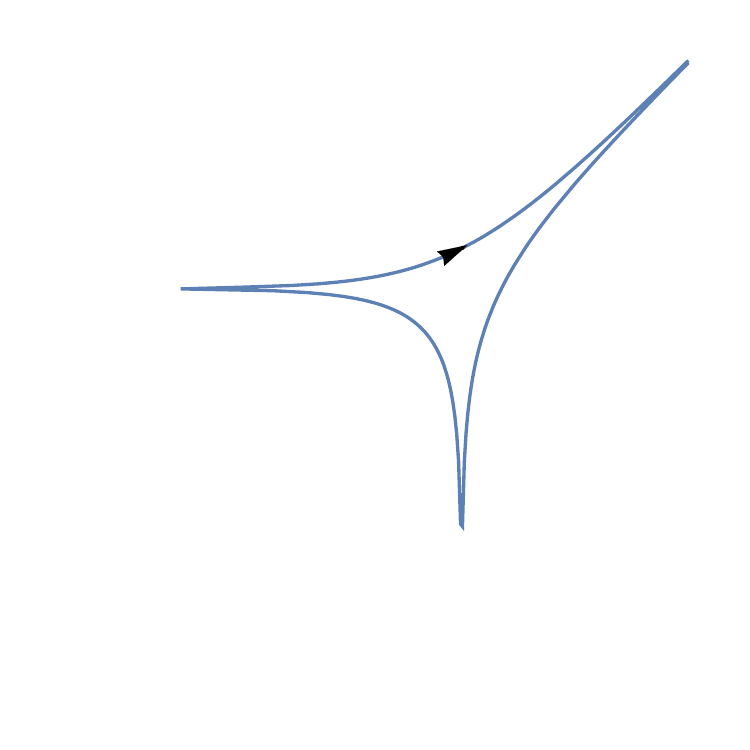}\hspace{5mm}
\includegraphics[height=28mm]{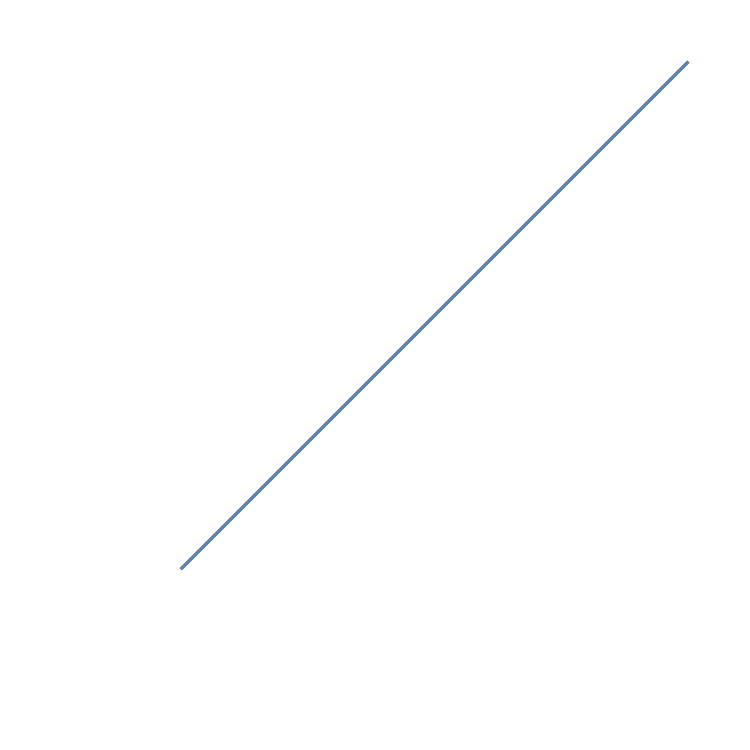}\hspace{-1mm}
\includegraphics[height=28mm]{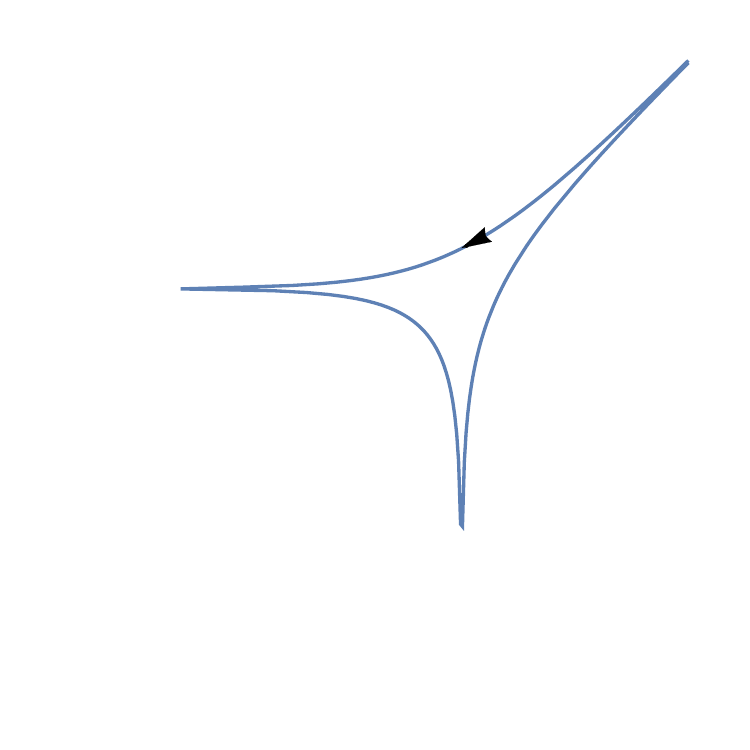}
\vspace{-5mm}\\
\hspace{1mm}\text{$k=-1$}\hspace{22mm}\text{\ \ \ $k=0$}\hspace{22mm}\text{$k=+1$}
\caption{
Oriented lines, their logarithmic images and quantum indices.
\label{p}}
\end{figure}
\end{exam}
\begin{exam}[Quantum indices of real conics]
A smooth nonempty real conic is a curve of type I.
Figure \ref{rconics} depicts real conics in $\rp^2$ that intersect the coordinates axes
in 6 real points. 
\begin{figure}[h]
\includegraphics[height=25mm]{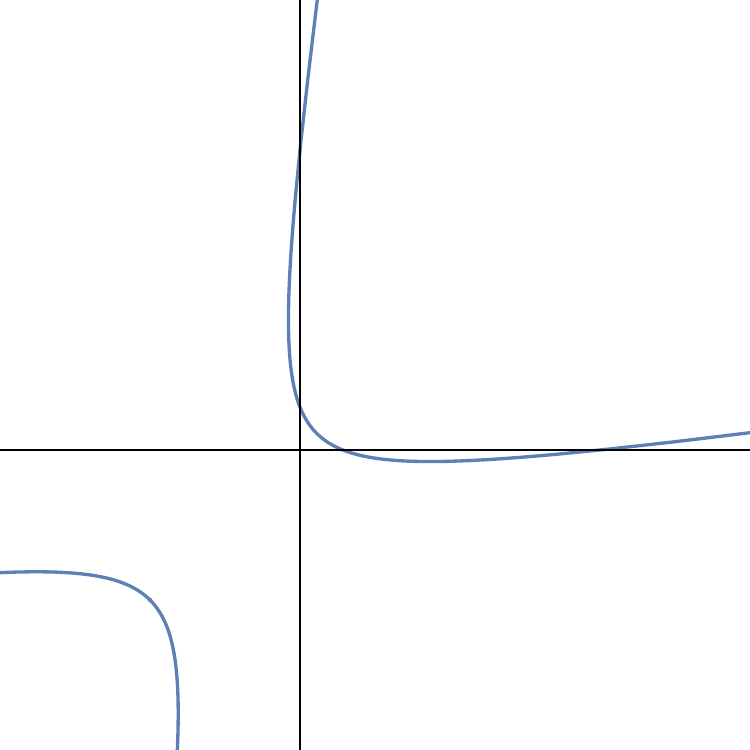}\hspace{5mm}
\includegraphics[height=25mm]{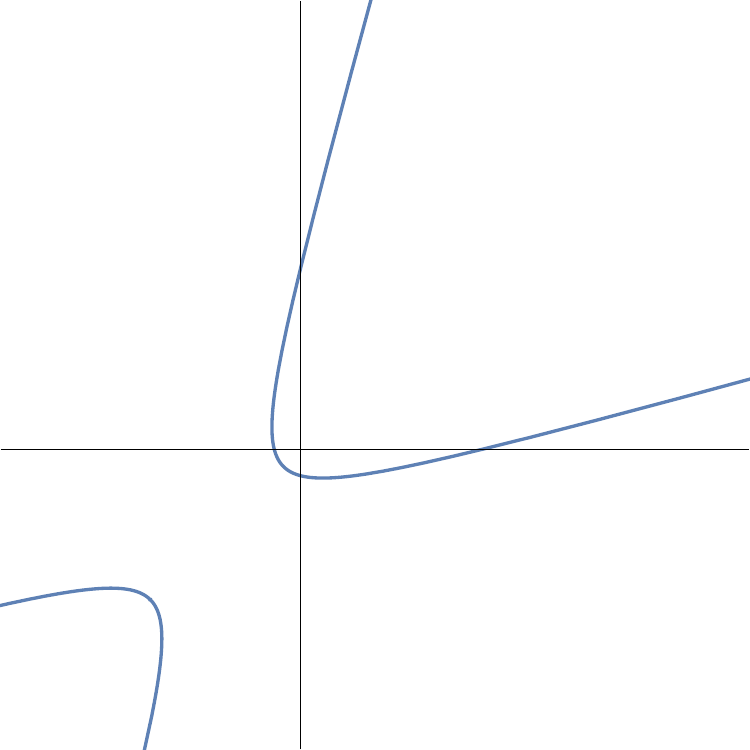}\hspace{5mm}
\includegraphics[height=25mm]{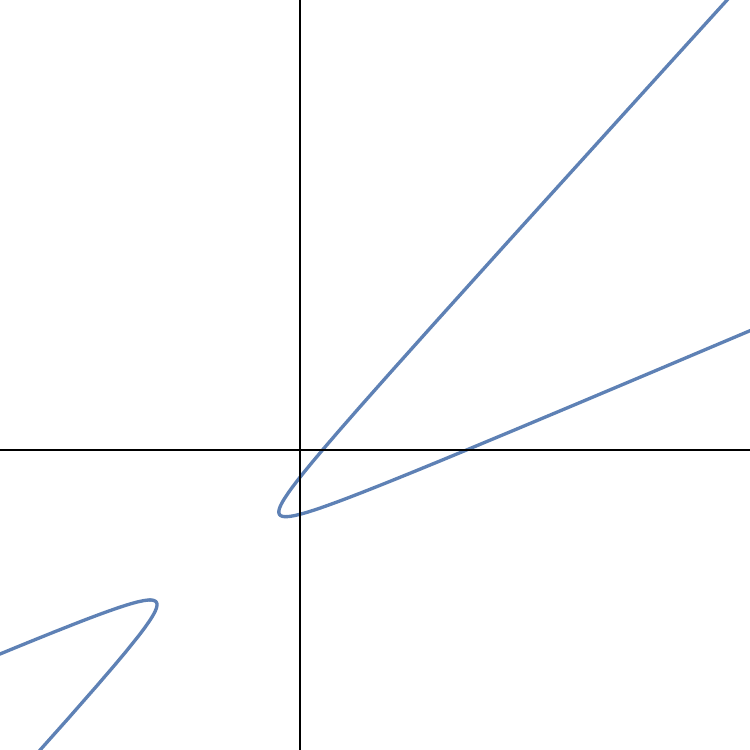}\hspace{5mm}
\includegraphics[height=25mm]{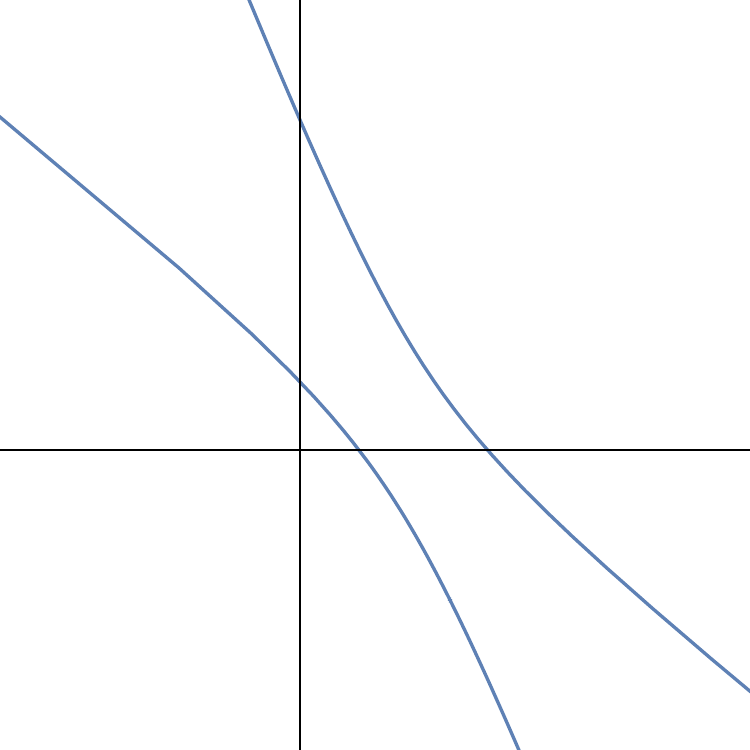}
\vspace{5mm}\\
\includegraphics[height=25mm]{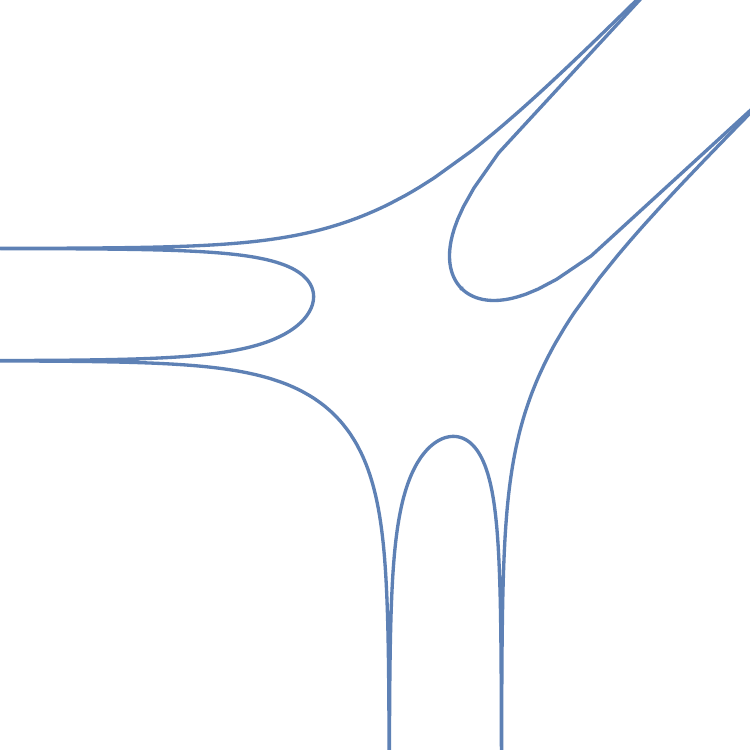}\hspace{5mm}
\includegraphics[height=25mm]{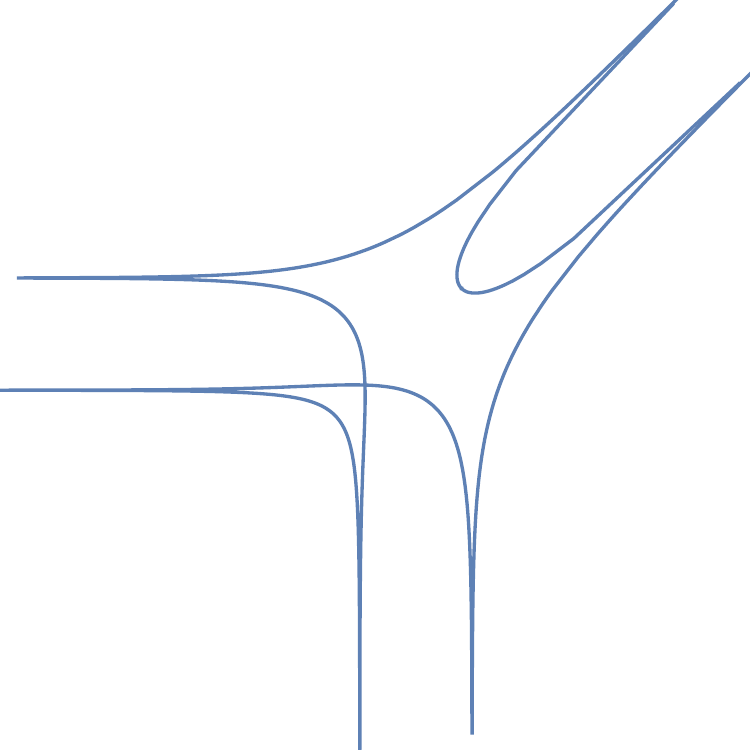}\hspace{5mm}
\includegraphics[height=25mm]{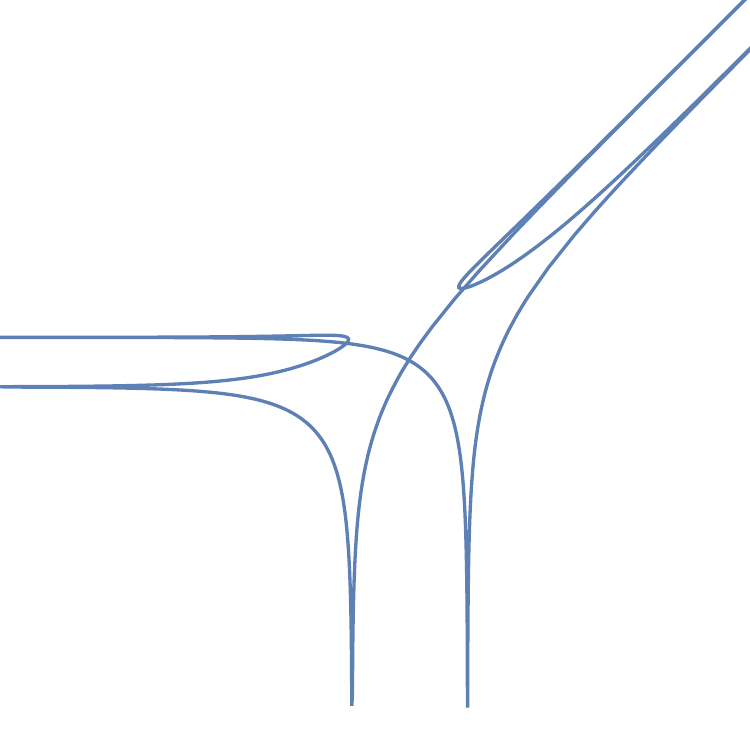}\hspace{5mm}
\includegraphics[height=25mm]{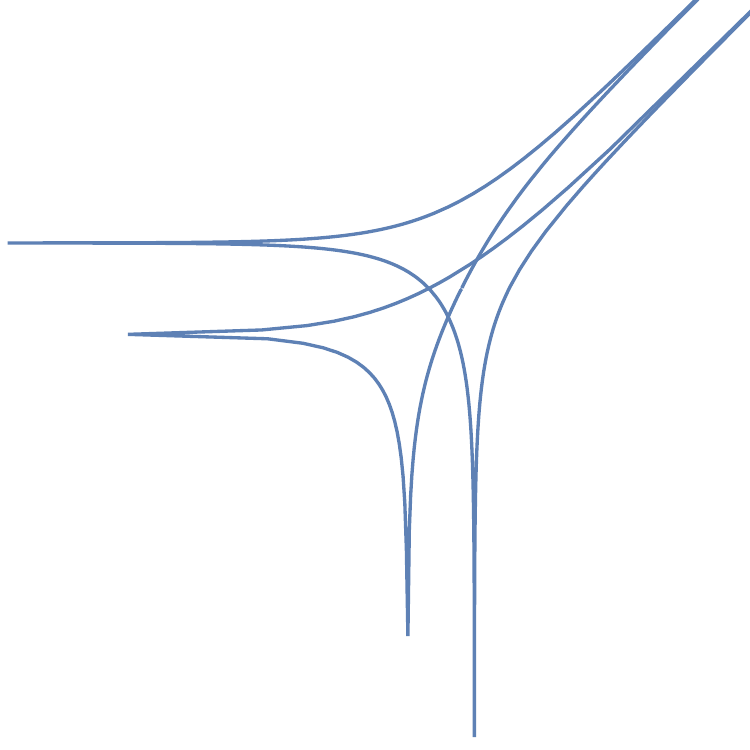}
\vspace{2mm}\\
\text{$k=\pm 2$}\hspace{20mm}\text{$k=\pm 1$}\hspace{20mm}\text{$k=0$}\hspace{20mm}\text{$k=\mp 1$}
\caption{
Projective hyperbolas, their logarithmic images and quantum indices.
\label{rconics}}
\end{figure}

Note that a circle in $\R^2$ intersects the infinite axis of $\rp^2$ at
the points $(0:1:\pm i)$. Thus a circle intersecting the coordinate axes of $\R^2$
in 4 real points has real or purely imaginary coordinate intersection, see Figure \ref{rcircles}.
A circle passing through the origin in $\R^2$
has quantum index $\pm\frac12$. Otherwise, the quantum
index of a circle is $0$ or $\pm 1$.
\begin{figure}[h]
\includegraphics[height=25mm]{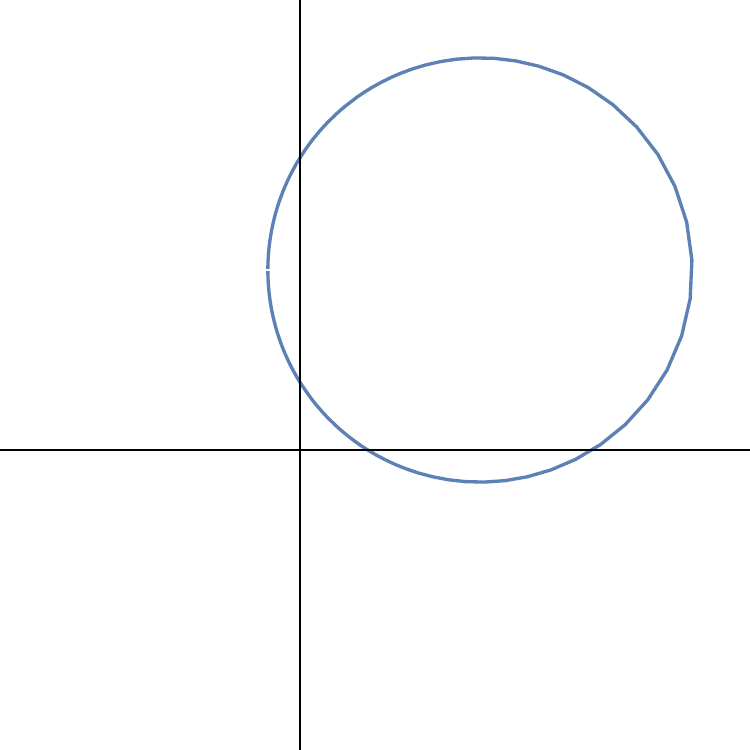}\hspace{10mm}
\includegraphics[height=25mm]{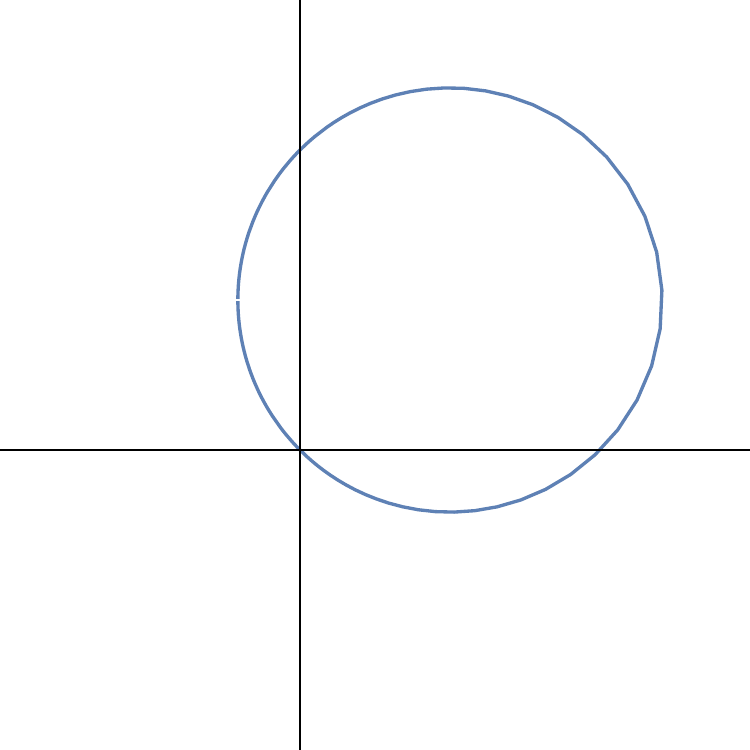}\hspace{10mm}
\includegraphics[height=25mm]{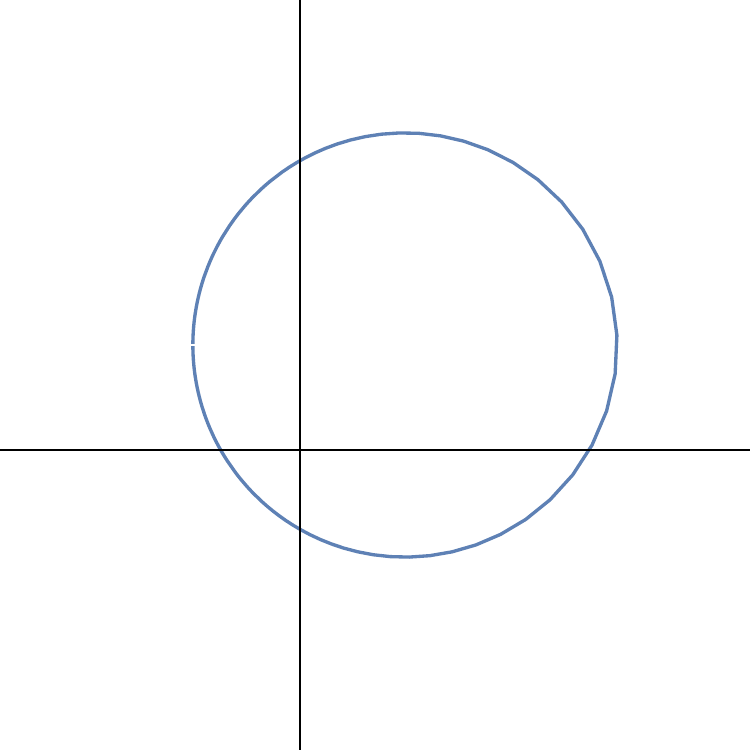}
\vspace{-5mm}\\
\includegraphics[height=28mm]{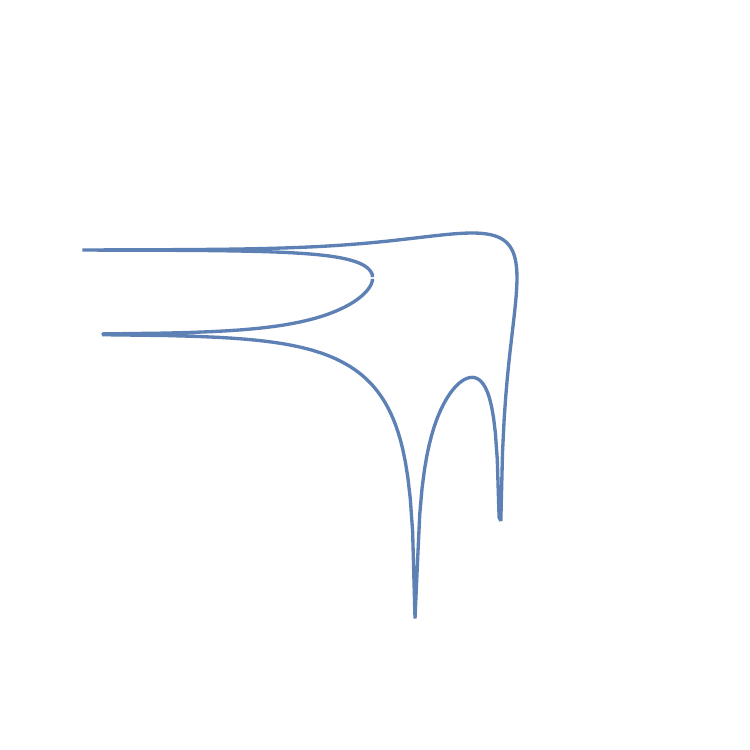}\hspace{5mm}
\includegraphics[height=28mm]{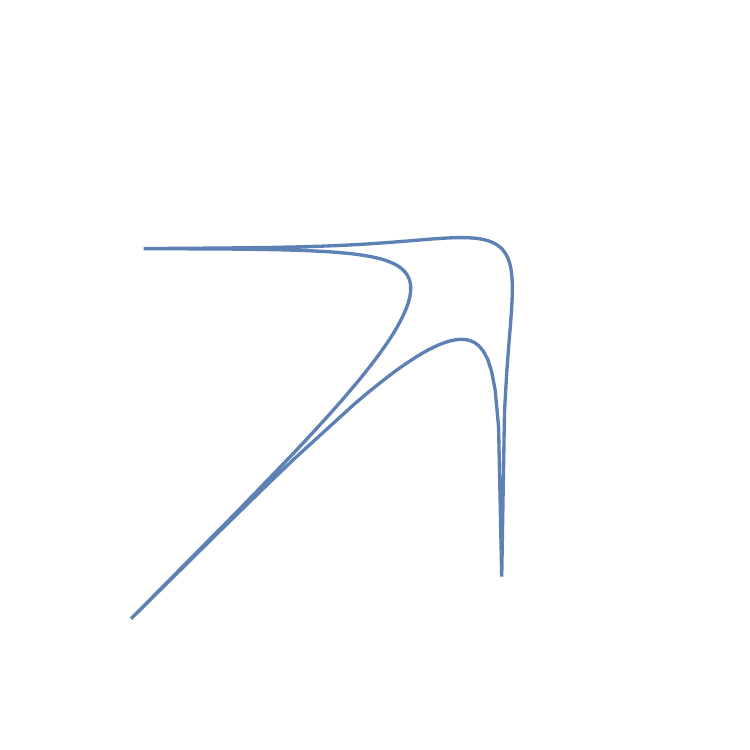}\hspace{5mm}
\includegraphics[height=28mm]{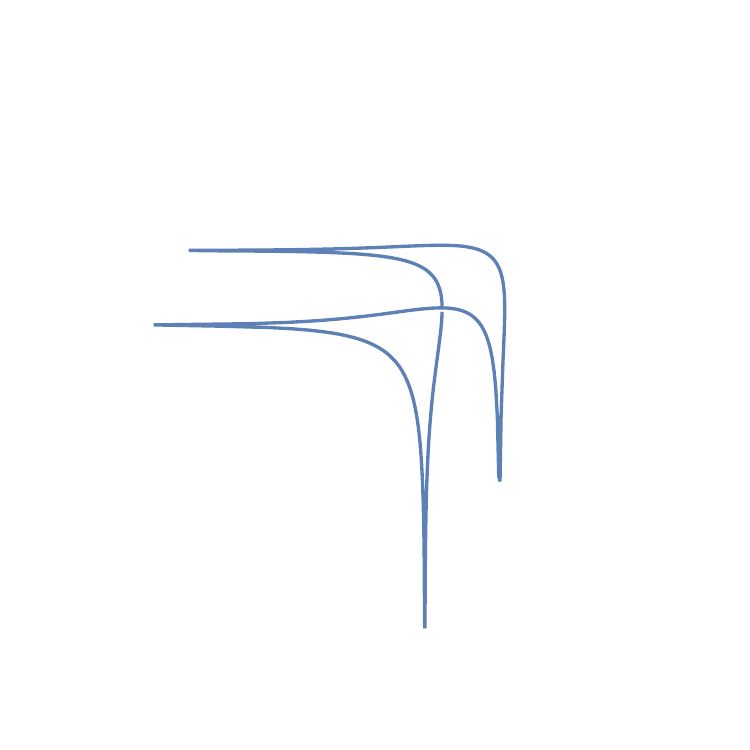}
\vspace{-2mm}\\
\text{$k=\pm 1$}\hspace{20mm}\text{$k=\pm \frac12$}\hspace{20mm}\text{$k=0$}
\caption{
Circles, their logarithmic images and quantum indices.
\label{rcircles}}
\end{figure}
\end{exam}

\section{Toric type I curves: quantum indices and diagrams}
\subsection{Definition of toric type I curves
and their index diagrams}
Denote with $\C\tilde C^\circ\subset\C \tilde C$ the normalization of an algebraic curve $\C C^\circ\subset\ctor$
and with $\R\tilde C^\circ$ its real part. The composition of the normalization and
the inclusion map induces a map $\C\tilde C^\circ\setminus\R\tilde C^\circ\to\ctor$.
\begin{defn}\label{def-ttI}
We say that an irreducible real algebraic curve
$\R C^\circ\subset\rtor$ 
has {\em toric type I} if $\R C$ is of type I
(see Section 2.1) and
the induced homomorphism
\begin{equation}\label{homS}
H_1(\C\tilde C^\circ\setminus\R\tilde C^\circ)\to H_1(\ctor)=\Z^2
\end{equation}
is trivial.
\end{defn}


\newcommand{\nuu}{\vec n}
Each side $E\subset\Delta$ is dual to a unique primitive integer vector $\nuu(E)\subset\Z^2$
(which sits in the space dual to the vector space containing the Newton polygon $\Delta$) oriented
away from $\Delta$. We refer to $\nuu(E)$ as {\em the normal vector to $E\subset\dd\Delta$}.

\begin{prop}\label{prop-realintersections}
If $\R C^\circ\subset\rtor$ is of toric type I then
$$\C \bar C\cap\dd\C \Delta\subset\R\bar C\subset\R\Delta.$$
In other words $\R C$ has real coordinate intersection.
Thus it has a well-defined quantum index for any of its
two complex orientation.
\end{prop}
\begin{proof}
The homology class in $H_1(\ctor)=\Z^2$ of a small loop in $\C C^\circ$
around a point of $\C \bar C\cap\dd\C \Delta$
is a positive multiple of $\nuu(E)$ 
for a side $E\subset\Delta$.
Therefore this class is non-zero.
Thus such loop must intersect $\R C^\circ$ if $\R C^\circ$ is of toric type I.
\end{proof}


\begin{defn}\label{idiagram}
A continuous map $a:\Sigma\to\R^2$ from
a graph $\Sigma$
is called {\em the index diagram} of the curve
$\R C^\circ$ of toric type I enhanced with a choice
of the complex orientation corresponding to
$S\subset\C \tilde C\setminus \R \tilde C$
if the following conditions hold.
\begin{itemize}
\item The vertices of the graph $\Sigma$ are 
parameterized by the connected components 
$K^\circ\subset\R\tilde C^\circ$.
We denote the corresponding vertex with  $v(K^\circ)\in\Sigma$.
\item The image $a(v(K^\circ))=(a,b)\in\Z^2$
is a lattice point in $\R^2$ such that $K^\circ$
is contained
in the $((-1)^{a},(-1)^{b})$-quadrant of $\rtor$.
\item 
Vertices $v(K^\circ_1),v(K^\circ_2)\in\Sigma$
are connected 
with an oriented edge $e$ (which we identify with
the straight oriented interval $[0,1]$)
if and only if
$K^\circ_1$ and $K^\circ_2$ are adjacent 
at a point $p_e\in\R\tilde C$
in the order defined by the complex orientation of
$\R \tilde C$. (Clearly both
$K^\circ_1$ and $K^\circ_2$ are non-compact
in such case.)
\item
The restriction $a|_e:e\approx [0,1]\to\R^2$
is an affine map with
\begin{equation}\label{id3}
a(v(K^\circ_2))-
a(v(K^\circ_1))=m_e\nuu(E).
\end{equation}
Here 
$m_e$ is the local intersection
number of $\C\tilde C$ and $\dd\C\Delta$ at $p_e$. 
\item 
There exists a continuous map 
\begin{equation}\label{liftS}
\tilde l: \tilde S= (S\setminus\dd\C\Delta)
\cup\R\tilde C^\circ\to\C^2
\end{equation}
holomorphic on $\tilde S\setminus\R\tilde C^\circ$
such that $e^{\pi \tilde l}$ coincides with the tautological
map $\tilde S\to\ctor$
while for every connected component
$K^\circ\subset\R\tilde C^\circ$
we have $$\Im\tilde l(K^\circ)=a(v(K^\circ)).$$
Here both the exponent $e^{\pi \tilde l}$
and the imaginary part $\Im\tilde l(K^\circ)$
are understood coordinatewise.
\end{itemize}
Topologically the graph $\Sigma$ is the disjoint union
of $n$ circles and $m$ points,
where $n$ is the number of 
components of $\R\tilde C$
intersecting $\dd\R\Delta$,
and $m$ is the number of compact components of
$\R\tilde C^\circ$.

Denote with $\bar\Sigma\subset\R^2$ the
convex hull of $a(\Sigma)$.
The map
\begin{equation}\label{indR}
\alpha:K^\circ\mapsto\Im\tilde l (K^0)
=a(v(K^\circ))\in
\bar\Sigma\cap\Z^2
\end{equation}
defined on the components of $\R\tilde C^\circ$
is called the {\em real index map}. 
Since the map $\tilde l$ is holomorphic
its imaginary part $\Im\tilde l$
is harmonic, and
thus $\Im\tilde l(\tilde S)\subset\bar\Sigma$.
\end{defn}

\begin{prop}\label{IDexists}
Any curve $\R C^\circ\subset\rtor$ of toric type I
admits an index diagram $\Sigma(\R C)\subset\R^2$
which is unique up to a translation by $2\Z^2$ in $\R^2$.
%
%
\end{prop}
\begin{proof}
Since $\R C$ is of toric type $I$ the surface
$\tilde S\subset\ctor$ lifts under the exponent
map $\C^2\to\ctor$. Translating the lift
by integer multiples of $\pi$ if needed
ensures that $(a,b)+2i\Z^2\subset\C^2$ corresponds to
the lift of the $((-1)^a,(-1)^b)$-quadrant in $\ctor$.
Denote this lift with $\tilde l$, and
define the map $a$ on the vertices
of $\Sigma$ by \eqref{indR}.
An edge $e\subset\Sigma$ 
is mapped to the image of the accumulation set
at the end of $\tilde S$ corresponding to $e$.
%
To check the condition \eqref{id3}
we change coordinates
in $\ctor$ multiplicatively so that
that the toric divisor corresponding to $e$
is the $x$-axis.
Then $\tilde l$ maps the accumulation set
at the $e$-end of $\tilde S$
to the vertical interval of length $2m_e$.
Reversing the coordinate change
we recover an interval parallel to $\nuu(E)$.
\end{proof}

Note that for each connected component
$K\subset\R\tilde C$ (which is necessarily closed)
with $K\cap\dd\R\Delta\neq\emptyset$
the formula \eqref{id3}
already determines the part $a(K):\Sigma(K)\to\R^2$
corresponding to $K$ of the index
diagram $a:\Sigma\to\R^2$ up to
a translation in $\R^2$. Indeed it suffices to
choose arbitrarily $\alpha(K^\circ)$ of an arc
$K^\circ\subset K\setminus\dd\R\Delta$ and
proceed inductively.
\begin{prop}\label{closedK}
If $\R C^\circ$ is a curve of toric type I
then
the broken line $a(\Sigma(K))$ resulting
from inductive application of \eqref{id3}
is closed for any
connected component $K$ of $\R\tilde C$
with $K\cap\dd\R\Delta\neq\emptyset$.

Conversely,
suppose that $\R\tilde C$ is an M-curve
(i.e. the number of its components
is one plus the genus of $\C\tilde C$) with
$(\C\tilde C\setminus\R\tilde C)\cap\dd\C\Delta=\emptyset$
such that 
the broken line defined inductively by \eqref{id3}
for every
connected components $K\subset\R\tilde C$
is closed.
Then $\R C^\circ$
has toric type I.
\end{prop}
\begin{proof}
The first part of the statement
is a corollary of Proposition \ref{IDexists}.
Conversely, for an $M$-curve $\R \tilde C$
each component of $\C\tilde C\setminus\R\tilde C$
is a sphere with punctures corresponding to 
the components of $\R \tilde C$.
The homology class of a loop for each component
is determined inductively by \eqref{id3}.
It is zero by our hypothesis.
\end{proof}



\subsection{Quantum index and toric complex orientation formula}
Let $\R C^\circ\subset\rtor$ be a curve
of toric type I enhanced with the complex
orientation corresponding to a half
$S\subset\C\tilde C\setminus\R\tilde C$.
Denote with
\begin{equation}\label{AreaSigma}
\Area \Sigma \in\frac 12\Z
\end{equation}
the signed
area (with multiplicities) enclosed by $a(\Sigma)$
in $\R^2$.
%

Let $p\in\R^2$ be a point,
let $R_\epsilon\subset\R^2$ be the oriented ray
emanating from $p$
in a generic direction $\epsilon$ in $\R^2$.
Define $\lk_\epsilon(p,\Sigma)$ 
as the intersection number of 
the image $a(\Sigma)$ and $R_\epsilon$
in points other than $p$.
If $p\notin a(\Sigma)$ then this number is the
index
of $p$ with respect to $a(\Sigma)$
(considered in Section 2.3),
and does not depend on 
the choice of $\epsilon$.
Otherwise,  $\lk_\epsilon(p,\Sigma)$
depends on $\epsilon$.

For each quadrant $Q=((-1)^a,(-1)^b)\R_{>0}^2$ we define
\begin{equation}\label{lQ}
\lk_\epsilon(Q,\Sigma)=
\sum\limits_{k_a,k_b\in\Z}
\lk_\epsilon((a+2k_a,b+2k_b),\Sigma)\in\Z.
\end{equation}

Any connected component $K\subset \R\tilde C$
disjoint from $\dd\R\Delta$ is contained
in a single quadrant $Q$.
The image $\Log(K)$ is a closed oriented curve in $\R^2$.
Let $\lambda(K)\in\Z$ be the rotation number of
$\Log(K)$, i.e.
the degree of the logarithmic Gau{\ss} map
of $K\subset\R^2$.
(E.g. if $K\subset\R^2$
is a positively oriented embedded circle contained in the $(+,+)$- or
$(-,-)$-quadrant (resp. in the $(+,-)$- or $(-,+)$-quadrant)
then $\lambda(K)=1$ (resp. $\lambda(K)=-1$).)
Any point of $S\cap Q$ is
a real isolated singular point $p\in \R C^\circ$.
We denote with $\lambda(p)\in\Z$ the intersection number of $S$
and $Q$.
Recall that the orientation of $\rtor$ (and thus,
of $Q$)
is defined in Section 2.3 as the pull-back
of the standard orientation of $\R^2$ by $\Log|_Q$).

If $K^\circ\subset\R\tilde C^\circ$ is 
a connected component (not necessarily compact)
then the local degree
of the oriented logarithmic Gau{\ss} map
$\tilde\gamma|_K:K\to\tilde{\rp}^1$ at a point 
$\epsilon\in\tilde{\rp}^1$ may depend
on the choice of $\epsilon$.
For $(a,b)\in\Z^2$ and $\epsilon\in\tilde{\rp}^1
\setminus\tilde{{\mathbb Q}{\mathbb P}}^1$
we set 
\begin{equation}
\label{lambda-ab}
\lambda_\epsilon(a,b)=
-\sum\limits_{K}\lambda(K)+\sum\limits_p\lambda(p),
\end{equation}
where the sums are taken over all components
$K^\circ\subset\R\tilde C^\circ$
with $\alpha(K^\circ)=(a,b)$
and all isolated singular
points $p$ of $\R C^\circ$ with $\Im \tilde l(p)=(a,b)$.
The following statement is straightforward 
(with the help of
the maximum principle for $\Im\tilde l$).
\begin{prop}\label{prop-sigmabar}
The number $\lambda_\epsilon(a,b)$ does not depend on $\epsilon$
if $(a,b)\notin a(\Sigma)$.
If $(a,b)\notin\bar\Sigma$ then $\lambda_\epsilon(a,b)=0$.
\end{prop}

For each quadrant $Q=((-1)^a,(-1)^b)\R_{>0}^2\subset\rtor$
we may take the sum
\begin{equation}
\label{lambdaQ}
\lambda_\epsilon(Q)=
\sum\limits_{k_a,k_b\in\Z}\lambda_\epsilon(a+2k_a,b+2k_b).
\end{equation}
The result is independent on the translation
ambiguity in the definition of
the real index map.

\begin{thm}
\label{thm-index}
If $\R C^\circ\subset\rtor$ is a real algebraic curve
of toric type I
enhanced with a choice of its complex orientation
then
\begin{equation}\label{kSigma}
k(\R C)=\Area\Sigma(\R C).
\end{equation}
For each $(a,b)\in\Z^2$
and $\epsilon\in\tilde{\rp}^1
\setminus\tilde{{\mathbb Q}{\mathbb P}}^1$
we have
\begin{equation}\label{co-ab}
\lambda_\epsilon(a,b)=\lk_\epsilon((a,b),\Sigma).
\end{equation}
\end{thm}

\begin{coro}
For a curve of toric type I with the index
diagram $\Sigma$ we have
\begin{equation}\label{co-Q}
\lambda_\epsilon(Q)=\lk_\epsilon(Q,\Sigma)
\end{equation}
for each quadrant $Q\subset\rtor$.
\end{coro}

The equality \eqref{co-Q} 
may be viewed
as {\em the toric complex orientation formula} for toric type I curves.

\begin{coro}\label{coro-number}
The total number of closed components of
a curve $\R\tilde C^\circ$ of toric type I
and
its solitary real singularities
is not less than the number of lattice
points $(a,b)\in\Z^2\setminus\Sigma$
with $\lk((a,b),\Sigma)\neq 0$.
\end{coro}
\begin{proof}
If $\lk((a,b),\Sigma)\neq 0$ then
by \eqref{co-ab} $\lambda(a,b)\neq 0$
and thus there exists a closed component or
a solitary real singularity of $\R C^\circ$ 
of real index $(a,b)$.
\end{proof}

\ignore{
The following lemma contains a strengthening of \eqref{co} .
Let $M$ be a connected compact oriented surface of genus equal to the genus of $S$
and with the number of boundary components equal to the number of components
of the diagram $\Sigma(\R C)$.

\begin{lem}\label{prop-co}
If $\R C^\circ$ is of toric type I then there exists a continuous finite-to-one map
\begin{equation}\label{mapM}
\alpha:M\to\R^2
\end{equation}
and a 1-1 correspondence between the points $q$ of
$\alpha^{-1}(\Z^2)\setminus\dd M$ and the union of the set of closed
components $K$ of $\R\tilde C^\circ$ with the set of points
$p\in S$ mapped to $\rtor$ by the normalization map $\C\tilde C\to \C \bar C$
satisfying to the following properties.
\begin{itemize}
\item  The map $\alpha$ is smooth on $M^\circ=M\setminus\dd M$.
\item The image $\alpha(\dd M)\subset\R^2$ is the digram $\Sigma(\R C)$ (well-defined up
to translations by $2\Z^2$).
\item A point $q\in M^\circ$ with $\alpha(q)=(a,b)$
corresponds to a component $K$ or an isolated
real singular point $p$ contained in the quadrant $((-1)^a,(-1)^b)\R_{>0}^2$.
\item The local degree of $\alpha$ at $q$ coincides with $\lambda(K)$
or $\lambda(p)$ for the corresponding component $K$
or the point $p$.
\end{itemize}
Furthermore, for any linear map $L:\R^2\to\R$ the map $L\circ\alpha|_{M^\circ}$
does not have local maxima.
\end{lem}
}

\begin{exa}\label{exa-rq}
All real rational curves which intersect $\dd\R\Delta$ in $\#(\dd\Delta\cap\Z^2)$ points
(counted with multiplicity) have toric type I as $\C\tilde C^\circ\setminus\R\tilde C^\circ$
is the disjoint union of two open disks. Therefore we may compute the quantum index of such
curves with the help of Theorem \ref{thm-index}.

\begin{figure}[h]
\includegraphics[height=25mm]{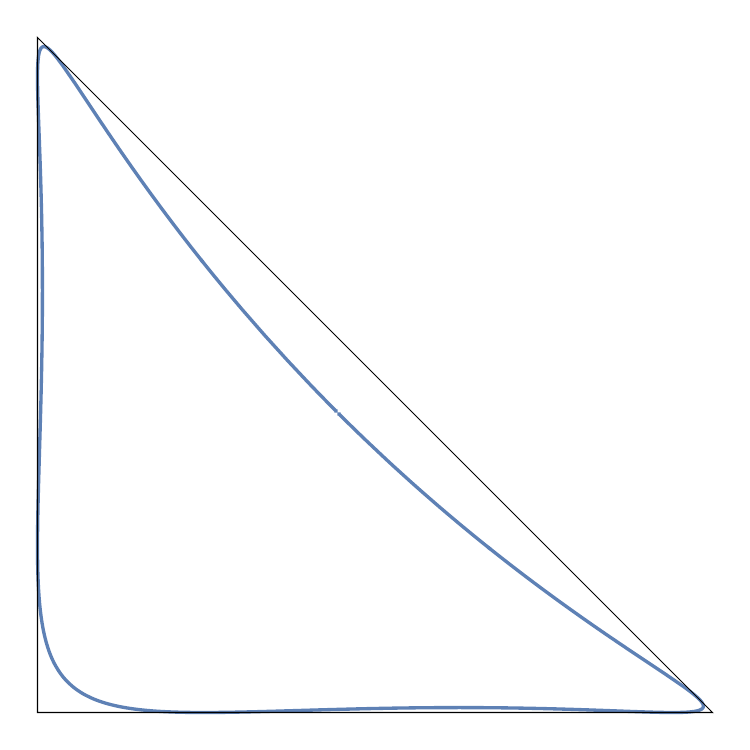}\hspace{5mm}
\includegraphics[height=25mm]{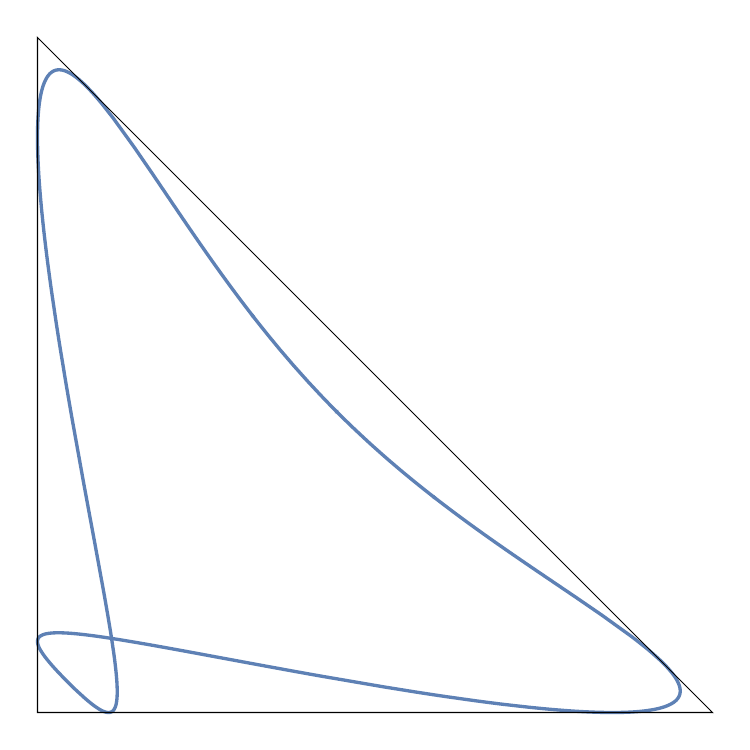}\hspace{5mm}
\includegraphics[height=25mm]{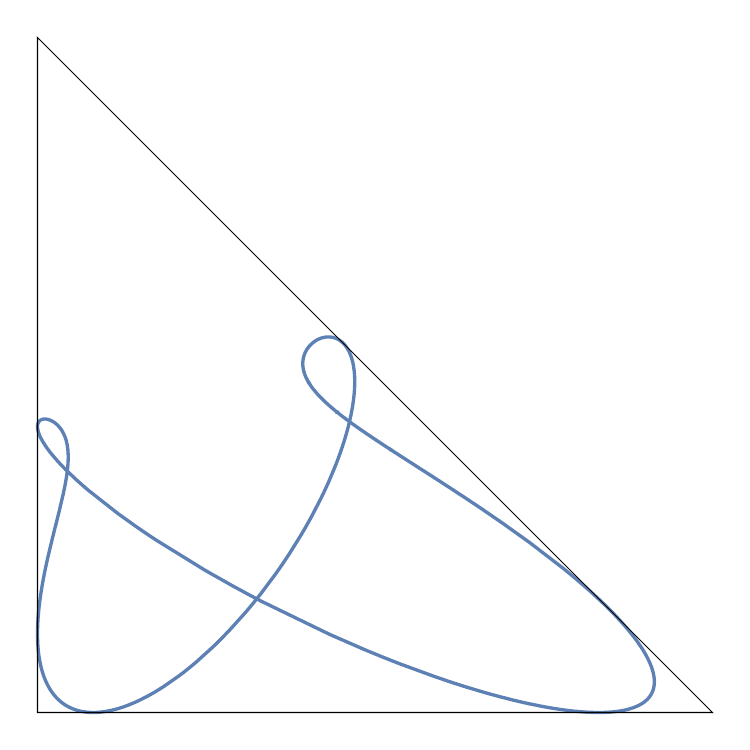}\hspace{5mm}
\includegraphics[height=25mm]{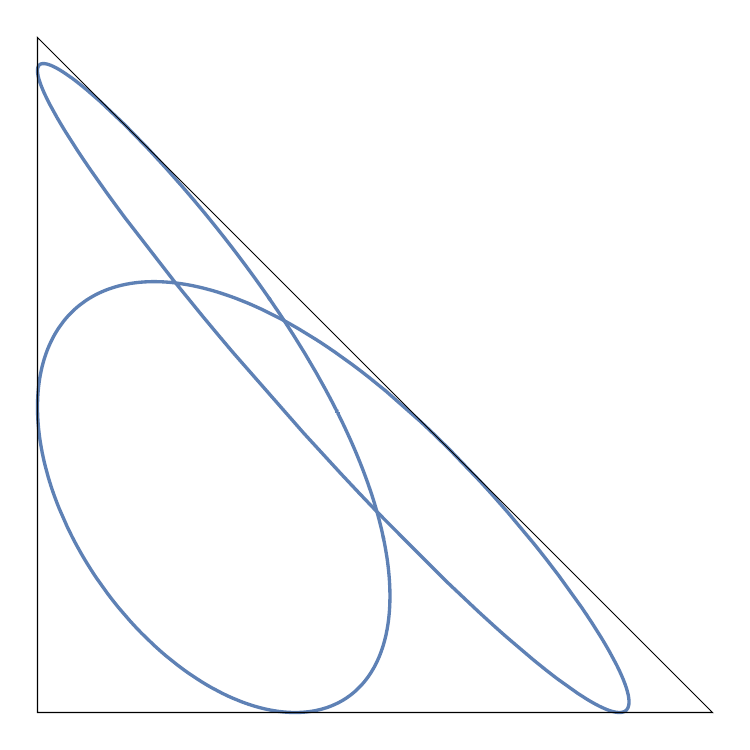}
\vspace{5mm}\\
\includegraphics[height=12.5mm]{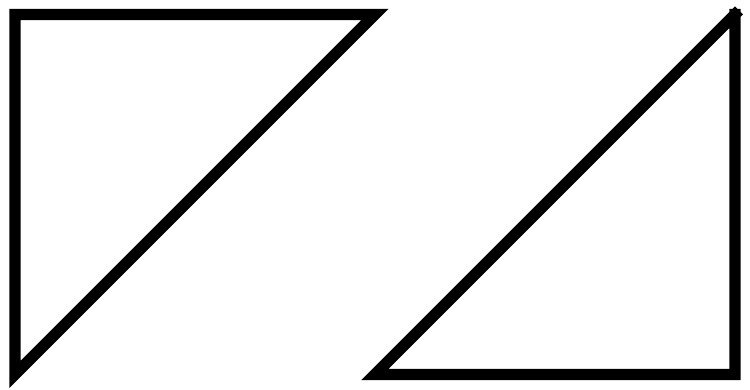}\hspace{6mm}
\includegraphics[height=12.5mm]{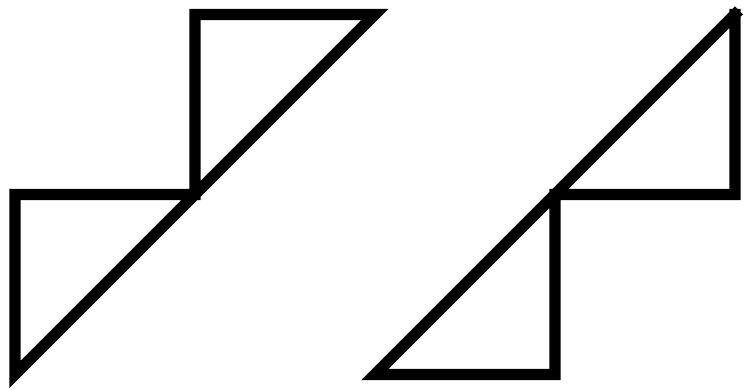}\hspace{6mm}
\includegraphics[height=12.5mm]{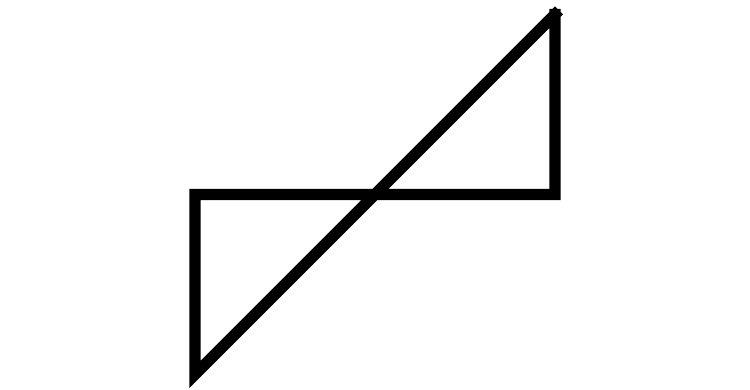}\hspace{6mm}
\includegraphics[height=12.5mm]{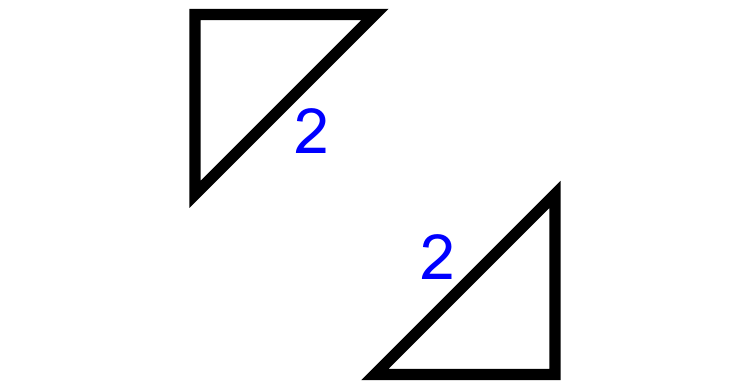}
\caption{Squares of the real conics from Figure \ref{rconics} and their diagrams $\Sigma$
for both possible orientations.
\label{c-diagrams}}
\end{figure}

Figure \ref{c-diagrams} depicts the images of the real conics from Figure \ref{rconics}
under $\Fr^\Delta$ (reparameterized with the help of the moment map). Each of these conics
may be oriented in two different ways producing two different diagrams.
For one of these conics the diagrams for the two opposite orientations coincide.
For the other three conics they are different.

Note that in the case of $\C\Delta=\cp^2$ the orientation can be uniquely recovered
from the diagram as the edges correspond to the normals to $\Delta$. E.g. the vertical
edges are always directed downwards.
\end{exa}

\begin{rem}
The diagram $\Sigma$ may be viewed as a non-commutative version
of the polygon $\Delta$. Here the set of normal vectors is given a cyclic order.

Note that $-\Area\Delta\le \Area\Sigma\le \Area\Delta$ for
any (possibly multicomponent)
closed broken curve $\Sigma$
whose oriented edges are normal vectors to $\Delta$
so that each side $E\subset\Delta$ contribute to $\#(E\cap\Z^2)-1$ normal
vectors (counted with multiplicity).
Furthermore, we have $\Area\Sigma=\pm\Area\Delta$ if and only if
$\Sigma\subset\R^2$ is a single-component
broken curve coinciding with the polygon
$\Delta$ itself rotated by 90 degrees
(as we can represent the primitive
normal vector to a vector $(a,b)\in\Z^2$ by $(-b,a)$ identifying the vector
space $\R^2$ with its dual).

Recall the notion
of cyclically maximal position of $\R \bar C\subset\R\Delta$
in $\R \Delta$ from
Definition 2 of \cite{Mi00}.
It can be rephrased that $\R \bar C$ has a connected component $K$
intersecting $\dd\R\Delta$
in $m=\#(\dd\Delta\cap\Z^2)$ points counted with multiplicity, and
the cyclic order of the intersection points on $K$
agrees with the
cyclic order of the corresponding normal vectors
to $\dd\Delta$.
This condition is equivalent to the equality \eqref{area-eq}.
It was proved in \cite{Mi00} that for each $\Delta$
the topological type of the triad $(\R \Delta;
\R\bar C_\Delta,\dd\R\Delta)$ is unique
if a curve $\R\bar C$ with the Newton polygon
$\Delta$ has cyclically maximal position and 
is transversal to $\dd\R \Delta$.
In this case $(\R \Delta;
\R\bar C_\Delta,\dd\R\Delta)$
is called {\em the simple Harnack $\Delta$-triad}.

The number of points in $(\Delta\setminus\dd\Delta)\cap\Z^2$
is equal to the arithmetic genus $g$ of $\R\bar C$.
On the other hand, Corollary \ref{coro-number} implies that
the total number of closed connected components of $\R C^\circ$ and
its isolated real singular points is at least $g$.
Thus all closed components of $\R C^\circ$ are smooth ovals and all
singular points of $\R C^\circ$ are solitary double points,
and the curve $\C C^\circ$ is a nodal $M$-curve.
We get the following statement.
\end{rem}

\begin{coro}\label{fle}
If $\R C^\circ\subset\rtor$ is a curve of toric type I with
\begin{equation}\label{area-eq}
\Area\Sigma(\R C)=\pm\Area\Delta
\end{equation}
then it is an M-curve whose only singularities are solitary real nodes.

Furthermore,
the topological type of $(\R\Delta;\R\bar C,\dd \R\Delta)$
is obtained from
the simple Harnack $\Delta$-triad $(\R \Delta;
\R\bar C_\Delta,\dd\R\Delta)$
by contracting some of the ovals of $\R C_\Delta^\circ$ to
solitary double points and replacing some $n$-tuples of consecutive transversal intersection points
of $\R \bar C_\Delta$ and $\dd\R\Delta$ (sitting on the same toric divisor)
with points of $n$th order of tangency.
\end{coro}
\begin{proof}
The curve $\Fr^\Delta(\C C^\circ)$ also has toric type I.
Its diagram is obtained from $\Sigma(\R C)$ by scaling both coordinates by 2,
so the quantum index of $\Fr^\Delta(\C C^\circ)$ is equal to $\pm\Area(2\Delta)$.
Corollary \ref{coro-number}
implies that the only singularities of $\Fr^\Delta(\C C^\circ)$ are solitary real nodes,
so that $\Fr^{\Delta}(\R \bar C)$ does not have self-intersections.
Therefore for each toric divisor
$\R E$ the order of intersection points on $\R E$ and that on the component
$K\subset\R \tilde C$ agree.

Let us look at the compact components of $\R C^\circ$. Their number and distribution
among the quadrants of $\rtor$ is determined by the lattice points of $\Sigma(\R C)$
and thus by $\Delta$. Furthermore,
Corollary \ref{coro-number} implies that in
each quadrant of $\rtor$ all ovals
and solitary real nodes of $\R C^\circ$ have the same orientation.
The complex orientation formula \cite{Rokhlin} ensures that these components
cannot be nested among themselves and that they are arranged with respect to $K$
so that their complex orientation is coherent with the complex orientation of $K$.
\end{proof}

\begin{rem}
The proof of Corollary \ref{coro-number}
is applicable also for pseudoholomorphic,
and even the so-called {\em flexible} (see \cite{Vi-fle}) real curves
of toric type I. Thus Corollary \ref{fle} may be considered
as a further generalization of the topological uniqueness theorem for simple Harnack curves
\cite{Mi00} from its version for pseudoholomorphic curves \cite{Br}
recently found by Erwan Brugall\'e.
\end{rem}

\begin{exa}
\begin{figure}[h]
\includegraphics[height=45mm]{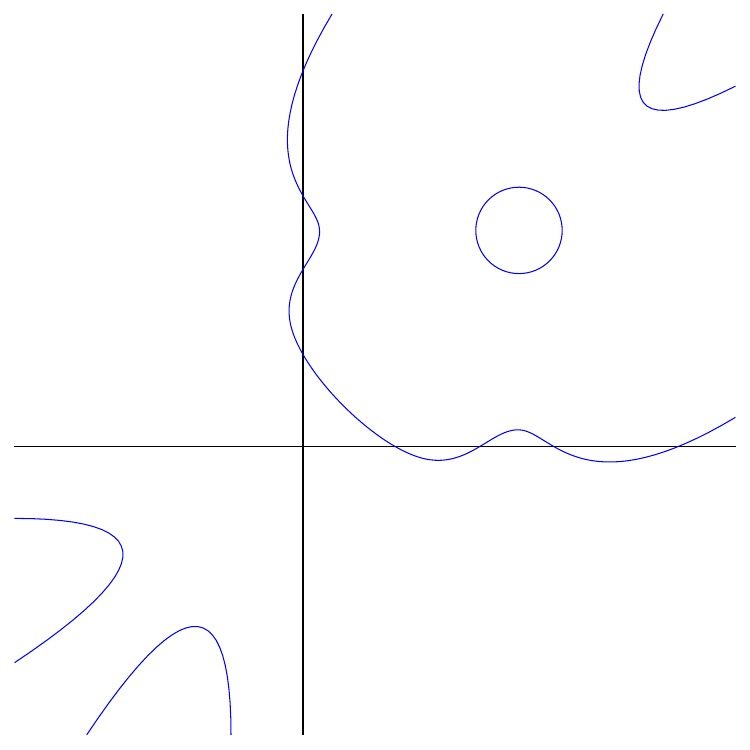}\hspace{5mm}
\caption{A quartic curve of type I, but not of toric type I.
\label{nHa}}
\end{figure}
The curve sketched on Figure \ref{nHa} is isotopic to a smooth real quartic curve of type I.
Namely, it can be obtained as a perturbation of a union of 4 lines.
However, it is not an $M$-curve while its diagram
coincides with the 
diagram of the simple Harnack curve of the same degree (i.e.
the triangle with vertices $(0,0),(0,-4),(-4,-4)$ for one of the orientations).
By Corollary \ref{fle} this curve is not of toric type I.
In other words there is a cycle in $\C C^\circ \setminus\R C^\circ$
that is homologically non-trivial in $H_1(\ctor)$.
Also we can deduce this
from the toric complex orientation formula \eqref{co-Q}.
\end{exa}

\subsection{The real index map vs. the amoeba index map}
To advance the viewpoint of the index diagram $\Sigma$ 
as a non-commutative version of the Newton polygon $\Delta$
it is interesting to compare
the real index map \eqref{indR} 
for toric type I curves and the
{\em amoeba index map} 
\begin{equation}\label{indA}
\ind:\R^2\setminus\Log(\C C^\circ)\to\Delta\cap\Z^2
\end{equation} 
of Forsberg-Passare-Tsikh
\cite{FoPaTs}.
The map \eqref{indA} is locally constant
and thus it indexes the components $K$
of the amoeba complement $\R^2\setminus\Log(\C C^\circ)$
by lattice points of $\Delta$.

One obvious distinction between $\ind$ and $\alpha$ is that
they take values in dual spaces: the Newton polygon
$\Delta$ belongs to the dual vector space to $\R^2=\Log\ctor$.
But thanks to the symplectic form $\omega((a,b),(c,d))=
ad-bc$, $a,b,c,d\in\R$ we have a preferred isomorphism
between these spaces.
Denote with $(a,b)^*=(b,-a)$ the corresponding identification.

As usual, we fix a complex orientation on $\R\tilde C$
and consider the corresponding component  
$S^\circ\subset\C\tilde C^\circ\setminus\R\tilde C^\circ$.
Let $$p,p'\in\R^2\setminus\Log(\C C^\circ)$$
and $\gamma\subset\R^2$ be a smooth path 
between $p$ and $p'$.
We assume $\gamma$ to be in general position 
with respect to $\R C^\circ$. 
The intersection number 
$\#(\gamma,\Log\R C)\in\Z$ is well-defined as
$\Log(\R C)\subset\R^2$ is proper.

\begin{prop}
We have $\#(\gamma,\Log\R C)=0.$ 
\end{prop}
\begin{proof}
The local degree of the map $\Log|_{S^\circ}:S^\circ\to\R^2$
changes along $\gamma$ according to the
intersection with $\R C^\circ$.
Since the local degree at the endpoints of $\gamma$
vanishes we have $\#(\gamma,\Log\R C)=0.$
\end{proof}

Using the real index map \eqref{indR} 
we may refine the intersection number 
$\#(\gamma,\Log\R C)=0$ to
\begin{equation}
\label{rig}
\#_\alpha(\gamma,\Log\R C^\circ)=\sum\limits_
{q\in\Log^{-1}(\gamma)\cap\R C^\circ}
\#_q(\gamma,\Log\R C^\circ)\alpha(q)\in\Z^2
\end{equation}
Here $\#_q(\gamma,\Log\R C^\circ)=\pm 1$
is the local intersection
number between $\gamma$ and $\Log \R C^\circ$
and $\alpha(q)\in\Z^2$
is the real index of the component of $\R\tilde C^\circ$
containing the point $q$. 

\begin{thm}\label{thm-indAR}
Let $\R C^\circ\subset\rtor$ be an algebraic
curve of toric type I. For any two points $p,p'\in
\R^2\setminus S^\circ\cup\R C^\circ$ and a generic
smooth path $\gamma\subset\R^2$ connecting $p$ and $p'$
we have
\begin{equation}\label{M-class}
(\ind(p')-\ind(p))^*=\#_\alpha(\gamma,\Log\R C^\circ).
\end{equation}
\end{thm}
\begin{proof}
Consider the 1-dimensional submanifold 
$$M=(\Log|_{\C C^\circ})^{-1}(\gamma)\subset \C C^\circ.$$
Its orientation is induced by that of $\gamma$
through the pull-back map
with the help of the orientations of the ambient
spaces: the standard orientation $\R^2\supset\gamma$
and the complex orientation of $\C C^\circ\supset M$.
Since $\gamma$ is chosen generically,
the 1-submanifold $M$ is smooth.

Any component of $M$ disjoint from $\R C^\circ$
is null-homologous in $\ctor$ as $\R C^\circ$
is a toric type I curve.
A component $L\subset M$ intersecting $\R C^\circ$ consists
of two arcs interchanged by $\conj$.
Let $q\in\R C^\circ$ be the source and
$q'\in\R C^\circ$ be the target of the 
arc $\delta=L\cap S^\circ$ with the orientation
induced from $M$.
We have $$\Im\tilde l(q')-\Im\tilde l(q)
=\alpha(q')-\alpha(q)$$ by \eqref{indR}
and therefore $[L]=\alpha(q')-\alpha(q)\in\Z^2=H_1(\ctor)$
so that $[M]\in H_1(\ctor)$ is given by the right-hand side
of \eqref{M-class}.

By \cite{Mi00} we may interpret
$\ind(p)$ as the linear functional on
$$H_1(\Log^{-1}(p))=\Z^2$$
that associates to each oriented
loop $N\subset\Log^{-1}(p)$ the linking number of
$N$ and the closure of the surface $\C C^\circ$ in $\C^2$.
Suppose that  $N'\subset\Log^{-1}(p')$
is a loop homologous to $N$
in $\ctor$ so that $N'-N=\dd P$ for a surface $P\subset\ctor$.
Then the difference of the linking
numbers of $N'$ and $N$ coincides with the intersection
number of $P$ and $\C C^\circ$. Choosing the membrane
$P$ to be contained in $\Log^{-1}(\gamma)$ we identify
the difference with the intersection number of $[N]$ and
$[M]$ in
$$\Z^2=H_1(\Log^{-1}(p))=H_1(\ctor)=H_1(S^1\times S^1).$$
\end{proof}

\section{Refined real enumerative geometry}
\subsection{Invariance of real refined enumeration}
\label{me51}
Let $\Delta\subset\R^2$ be a lattice polygon
with non-empty interior.
Let $E_j\subset\dd\Delta$, $j=1,\dots,n$, be its sides
of integer length $m_j=\#(E_j\cap\Z^2)-1$
enumerated counterclockwise. 
We denote with $v_j\in\Delta$ the vertices
of $\Delta$ enumerated so that $E_j$ connects
$v_{j-1}$ and $v_j$ (using the convention
$v_0=v_n$).
Let $\C E_j$
be the corresponding toric divisors.
Let $\PP=\{p_l\}_{l=1}^m$
be a collection
of $m=\#(\dd\Delta\cap\Z^2)$
points
on $\dd\C\Delta$. 
We do not assume the points $p_l$
to be distinct, but assume that 
exactly $m_j$ of these points
are contained in the toric divisor $\C E_j$
(in particular, we have
$\PP\cap\bigcup\limits_{j=1}^n \{v_j\}=\emptyset$).

The primitive vector
$(a_j,b_j)\in\Z^2$ parallel to $E_j$
and coherent with the counter-clockwise
orientation of $\dd\Delta$
defines the multiplicative-linear
(monomial) map
$\pi_j:\ctor\to\C^\times$
by $\pi_j(z,w)=z^{a_j}w^{b_j}$.
This map extends
to a continuous map
$\bar\pi_j:\ctor\cup(\C E_j\setminus\{v_{j-1},v_j\})$).
Define the map 
$$\rho:\bigcup\limits_{j=1}^n\C E_j\setminus\{0_j,\infty_j\}\to\C^\times$$
by $\rho(p)=\bar\pi_j(p)$ for $p\in\C E_j$.
\begin{rmk}
Note that the two coordinates $z$ and
$w$ in $\ctor$
give two meromorphic functions on
the Riemann surface $\C\tilde C$
obtained through normalization
$\bar\nu:\C\tilde C\to\C\bar C$.
{\em Le symbole mod\'er\'e} 
(defined by J. Tate according to
\cite{Deligne-SM})
\begin{equation}\label{symmod}
(z,w)_{\tilde p}=(-1)^{v(z)v(w)}
[w^{v(z)}z^{-v(w)}](\tilde p)
\in\C^\times
\end{equation}
for a point $\tilde p\in\C \tilde C$
is determined up to a sign
by its image
$p=\bar\nu(\tilde p)\in\C\Delta$.
The notation $v(z)$ refers to
the valuation of $z$ at $p$,
i.e. to the order of zero (or pole with
the negative sign) of this meromorphic 
function. The square brackets
$[w^{v(z)}z^{-v(w)}](\tilde p)$ refer
to taking the limit of the function
$w^{v(z)}z^{-v(w)}$ when the argument tends
to $\tilde p$. 
Note that the result 
is always a nonzero complex number.
Equivalently, we can write
$(z,w)_{\tilde p}=1$ if
$p\in\ctor$, and
\begin{equation}\label{sm-rho}
(z,w)_{\tilde p}=(-1)^{v(z)v(w)}\rho^{-a}(p)
\end{equation}
if $p\in\dd\C\Delta$, where $a$ is the order
of tangency of $\C\tilde C$ to $\dd\C\Delta$
at $\tilde p$.
\end{rmk}

\begin{condition}[Menelaus condition on $\PP$]
\begin{equation}\label{meneco}
\prod\limits_{l=1}^m\rho(p_l)=(-1)^m.
\end{equation}
\end{condition}
The following proposition is known as the {\em Menelaus theorem}
in the case of $\R\Delta=\rp^2$ and $m=3$ (lines),
and generalized by Carnot \cite{Ca} to higher
degree curves.
It is also known as the {\em Weil reciprocity law}, see e.g. \cite{GrHa}.
\begin{prop}[cf. (1.2) of \cite{Deligne-SM}]
\label{WRL}
There exists a curve
$\C\bar C\subset\C\Delta$
such that
$\C\bar C\cap\dd\C\Delta=\PP$
(in the sense that each point 
$p\in\C\bar C\cap\dd\C\Delta$ is included in $\PP$
the number of times equal to the local intersection
number of $\C\bar C$ and $\dd\C\Delta$ at $p$)
if and only if \eqref{meneco} holds.
\end{prop}
\begin{proof}
The torus part
$\C C^\circ=\C\bar C\cap\ctor$ is defined by
a polynomial
$f(z,w)=
\sum a_{(\iota_1,\iota_2)} z^{\iota_1}w^{\iota_2}$.
Note that the condition
$\C\bar C\cap \C E_j=\PP\cap\C E_j$
implies that
the Newton polygon
of $f$ coincides with $\Delta$
(up to translation in $\R^2$).
Furthermore, the intersection $\C\bar C\cap \C E_j$
is determined by $a_{(\iota_1,\iota_2)}$ with
$(\iota_1,\iota_2)\in E_j$.

Suppose that \eqref{meneco} holds.
The set $\pi_j^{-1}(\PP\cap\C E_j)$ is
the zero locus a polynomial $f_j$ whose Newton
polygon $\Delta_{f_j}$ is a translate of the side $E_j$.
Multiplying $f_j$ by an appropriate monomial
we ensure that $\Delta_{f_j}=E_j$,
$f_j=\sum a^{(j)}_{(\iota_1,\iota_2)}
z^{\iota_1}w^{\iota_2}$.
We have
$$a^{(j)}_{v_{j-1}}/a^{(j)}_{v_j}=
(-1)^{m_j}\prod\limits_{p_l\in\C E_j}
\rho(p_l)$$
by the Vieta theorem.
Therefore we can choose $f_j$ in such
a way that $a^{(j)}_{v_j}=a^{(j+1)}_{v_j}$
(using the convention
$a^{(n+1)}_{v_n}=a^{(0)}_{v_n}$)
if and only if
\eqref{meneco} holds.

Vice versa, if a curve with
$\C\bar C\cap\dd\C\Delta=\PP$
exists then it is given by a polynomial
with the Newton polygon $\Delta$.
Applying the Vieta theorem to the coefficients
corresponding to $\dd\Delta$ we recover 
the condition \eqref{meneco}.
\end{proof}

In other words, \eqref{meneco} means that
the linear system defined by the divisor $\PP$
on the (singular) elliptic curve $\dd\C\Delta$ is $\OO(\Delta)$,
i.e. any curve with the Newton polygon $\Delta$
passing through the points $\{p_j\}_{j=1}^{m-1}$ also passes through $p_m$.
By Proposition \ref{WRL} for any subset of $m-1$ points on $\PP$ there exists a unique remaining
point with this condition.
We say that $\PP$ is
a {\em a generic Menelaus configuration
of $m$ points on $\dd\R\Delta$}
if the first $m-1$
points of $\PP$ are chosen generically on $\dd\R\Delta$.
We make this genericity assumption to formulate
Theorem \ref{Rinv}.


\ignore{
Assume that the configuration $\PP$ is invariant under the
involution of complex conjugation in $\C \Delta$.
In addition, assume that
we have $\Fr^{\Delta}(\PP)\subset\R \Delta$.
}

An {\em oriented real rational curve}
$\R \bar C\subset\C \bar C$
is a real curve
whose normalization is isomorphic to $\rp^1\subset\cp^1$
as well as a choice of orientation on this
$\rp^1\approx S^1$.
Note that by Jordan's theorem such curve must be
of type I.
The configuration $(\Fr^\Delta)^{-1}(\PP)$ consists of 
real or purely imaginary
points.
Thus an oriented real rational curve $\R \bar C\subset\R \Delta$ with the
Newton polygon $\Delta$
such that $\Fr^{\Delta}(\C C)$ passes through $\PP$
has the quantum index $k(\R C)\in\frac12\Z$.

Define
\begin{equation}\label{sigma}
\sigma(\R C)=(-1)^{\frac{m-\lrot(\R C)}2}.
\end{equation}
Since the parities of $\lrot(\R C)$ and $m$ coincides we have $\sigma(\R C)=\pm 1$.
\begin{rmk}
Note that if $\R C^\circ$ is nodal then its toric
solitary singularities number $E(\R C)$
has the same parity as the number of solitary nodes of $\R C^\circ$.
Thus the Welschinger sign $w(\R C)$ (see \cite{We}) coincides with $(-1)^{E(\R C)}$.
Since the curve $\R \bar C$ intersects the union $\dd\R \Delta$ of toric divisors
in $m$ distinct points and $\frac m2\equiv \Area(\Delta)\pmod 1$ by Pick's formula,
we have
\begin{equation}\label{sigma-w}
\sigma(\R C)=(-1)^{Area(\Delta)-k(\R C)}w(\R C)
\end{equation}
by Theorem \ref{thm-lrot}.
\end{rmk}

\ignore{
Any real rational curve is a curve of type I. To associate a quantum
index to $\R C$ we need to choose its orientation. We may do so
with the help of the point $p_m$. Namely if $p_m$ is
adjacent to the quadrant $\R_{>0}^2$ then we orient $\R C$ from
$\R_{>0}^2$ to another quadrant at $p_m$.
If $p_m$ is not adjacent to the quadrant $\R_{>0}^2$, but is adjacent
to the quadrant $\R_{<0}^2$ then we orient $\R C$ from
$\R_{<0}^2$ to another quadrant at $p_m$.
}

We define
\begin{equation}\label{Rk}
R_{\Delta,k}(\PP)=\frac 14\sum\limits_{\R \bar C}\sigma(\R C),
\end{equation}
where the sum is taken over all oriented real rational
curves $\R \bar C$ (in particular, irreducible over $\C$)
with the Newton polygon $\Delta$ and $k(\R \bar C)=k$
such that $\Fr^\Delta(\C \bar C)\supset\PP$.
We have the coefficient $\frac 14$ in the right-hand side of $\eqref{Rk}$
as the group of the deck transformations of $\Fr^\Delta:\C\Delta\to\C\Delta$
is $\Z_2^2$, so each curve $\R \bar C$ comes in four copies with the same
image $\Fr^\Delta(\C \bar C)$.
(Alternatively, we can take a sum over different oriented images
$\Fr^\Delta(\R \bar C)$
without the coefficient $\frac 14$.)
Each real rational curve gives rise to two oriented
real rational curves $\R \bar C$
(one for each orientation),
and thus enters the sum \eqref{Rk} twice.

Recall that we call a point in $\dd\R\Delta$ {\em positive} if it is adjacent to
the quadrant $(\R_{>0})^2$ and {\em negative} otherwise.
Note that $(\Fr^{\Delta})^{-1}(p)$ consists of real points if $p$ is positive
and of purely imaginary points if it is negative.
Let $\lambda_j$ be the number of negative points in $\PP\cap\R E_j$.
Denote $\lambda=(\lambda_j)_{j=1}^n$.
\begin{thm}\label{Rinv}
The number $R_{\Delta,k}(\lambda)=R_{\Delta,k}(\PP)$
depends only on $\Delta$, $k$ and $\lambda$.
\end{thm}
If all points of $\PP$ are negative (i.e. $m_j=\lambda_j$)
the number $R_{\Delta,k}(\lambda)$ is the number of oriented real rational curves $\R \bar C$ of quantum index $k$ contained in the positive quadrant $\R_{>0}^2\subset\rtor$ and passing through {\em all} points of the purely imaginary configuration $(\Fr^{\Delta})^{-1}(\PP)$.
For a positive point $p\in\PP$ a curve $\R \bar C$ 
should passes through
one of the two real points in $(\Fr^{\Delta})^{-1}(p)$.

Define
\begin{equation}\label{Rtk}
\tilde R_{\Delta,k}(\PP)=
\sum\limits_{\R \bar C}\sigma(\R C)
\end{equation}
where the sum is taken over all oriented real rational curves $\R\bar C$ of quantum index $k$
with the Newton
polygon $\Delta$ and $k(\R C)=k$ passing through $\PP$.
Let
\begin{equation}\label{Delta-d}
\Delta_d=\ConvexHull\{(0,0),(d,0),(0,d)\}.
\end{equation}
We have $\C\Delta_d=\cp^2$.
\begin{exa}
The curves $\R C$ with the Newton polygon $\Delta_2$ are projective conics.
In this case $n=3$, $m_1=m_2=m_3=2$, $m=6$, and
For any generic Menelaus configuration
$\PP\subset\dd\rp^2$
we have a unique conic through $\PP$. This gives us two oriented curves through $\PP$
of opposite quantum index.

We may assume (applying the reflections in $x$ and $y$ axes if needed)
that $\PP$ contains a positive point in the $x$-axis and a positive point
in the $y$-axis. As the positivity of the last point of $\PP$ will be determined
by the Menelaus condition we have 3 possibility for the non-decreasing sequence
$\lambda=(\lambda_1,\lambda_2,\lambda_3)$.
The possible values of $k(\R C)$ are listed in
the following table, cf. Figure \ref{rconics}.
\begin{table}[hh]
\begin{tabular}{|| l | l | l || l ||}
\hline
$\lambda_1$ & $\lambda_2$ & $\lambda_3$ & $k(\R C)$\\
\hline
0 & 0 & 0 & $\pm 2$\\
\hline
0 & 0 & 2 & $\pm 1$ or 0\\
\hline
0 & 1 & 1 & $\pm 1$ or 1\\
\hline
\end{tabular}
\caption{\label{tabconics} Quantum indices for conics}
\end{table}
In particular in this case $\tilde R_{\Delta_2,k}(\PP)$ is different
for different configurations with the same $\lambda$
for two last rows of Table \ref{tabconics}.
Thus the numbers $\tilde R_{\Delta,k}(\PP)$ may vary
when we deform $\PP$.
\end{exa}

Define
\begin{equation}\label{Reven}
\tilde R_{\Delta, even}(\PP)=
\sum\limits_{k\in \frac m2+2\Z}\tilde R_{\Delta, k}(\PP);\ \ \
\tilde R_{\Delta, odd}(\PP)=\sum\limits_{k\in \frac m2+1+2\Z}\tilde R_{\Delta, k}(\PP).
\end{equation}
\begin{thm}\label{tildeR}
The numbers $\tilde R_{\Delta_d,even}(\PP)$ and $\tilde R_{\Delta_d, odd}(\PP)$
do not depend of $\PP$ as long as $d$ is even and all the points of $\PP$
are positive (i.e.  $\lambda=(0,0,0)$).
\end{thm}

\subsection{Refined real and refined tropical enumerative geometry}
We return to the study of the invariant $R_{\Delta,k}$ from Theorem \ref{Rinv}.
\begin{defn}
The sum
\begin{equation}\label{Rd}
R_\Delta(\lambda)=\sum\limits_{k=-\Area(\Delta)}^{\Area(\Delta)}R_{\Delta,k}(\lambda)q^k
\end{equation}
is called {\em the real refined enumerative invariant} of $\rtor$ in degree $\Delta$.
\end{defn}

If $\lambda=0$ (i.e. all $\lambda_j=0$) then all points of
$(\Fr^{\Delta})^{-1}(\PP)$ are real.
In such case we use notations $R_{\Delta,k}=R_{\Delta,k}(0)$
and $R_{\Delta}=R_{\Delta}(0)$.

Recall that the Block-G\"ottsche invariant
\cite{BlGo} 
is a symmetric (with respect to the substitution $q\mapsto q^{-1}$)
Laurent polynomial with positive integer coefficients.
This polynomial is
responsible for the enumeration of the tropical curves
with the Newton polygon $\Delta$ of genus $g$, passing through
a generic collection of points in $\R^2$, see \cite{ItMi}.
The expression $\ntropm_\Delta$ defined by \eqref{BGDeltadef}
may be viewed as the counterpart of
$N^{\Delta,0}_{\operatorname{trop}}$
defined in \cite{BlGo}
for the case when the tropical curves pass
through a collection of $m$ points on the boundary of the toric tropical surface $\T\Delta$
which are generic among those satisfying
to the tropical Menelaus condition \eqref{mo0}.

\begin{thm}\label{BG}
$$R_\Delta= (q^{\frac12}-q^{-\frac12})^{m-2}\ntropm_\Delta.$$
\end{thm}

\begin{coro}
The number $N^{\dd,\C}_\Delta$ of complex rational curves in $\C\Delta$
with the Newton polygon $\Delta$ passing through $\PP$
is determined by $R_\Delta$. 
\end{coro}
\begin{proof}
By \cite{Mi05} the number $N^{\dd,\C}_\Delta$
coincides with the value of $\ntropm_\Delta$
at $q=1$.
\end{proof}
Let us reiterate that $R_\Delta$ accounts only for curves 
in $\ctor$ defined over $\R$.


\subsection{Holomorphic disk interpretation}
Recall that an orientation of a real rational curve $\R C$ defines
a connected component $S\subset\C \tilde C\setminus\R \tilde C$.
Let $D$ be the topological closure of the image of $S$
in $\C \Delta$.
The disk $D\subset\C\Delta$ is
a holomorphic disk whose boundary $\dd D=\R C$ is contained in
the Lagrangian subvariety $\R\Delta\subset\C\Delta$.

Let $L$ be the topological closure of the quadrant $\R_{>0}^2$ in $\C \Delta$.
Note that $L$ is a Lagrangian subvariety of $\C \Delta$ with boundary.
The image $\Fr^\Delta(D)$ is a holomorphic disk whose
boundary is contained in $L$.

Thus the expression \eqref{Rd} may also be interpreted as a refined
enumeration of holomorphic disks with boundary in $L$,
passing through $\PP$,
and tangent to $\dd\R\Delta$ at the points of $\PP$.

These disks 
are images under $\Fr^\Delta$ of disks $D$ with boundary in $\R \Delta$
and
\begin{equation}
\Fr^\Delta(D)\cap\dd\C\Delta=\PP\subset\dd\R\Delta.
\end{equation}
Let $\widehat{\C\Delta}$ be the result of blowup of the toric variety $\C \Delta$ at $\PP$.
Let $\hat L=\overline{\rtor}\setminus\widehat{\dd\R\Delta}$
where $\overline{\rtor}$
is the topological closure of $\rtor$ in
$\widehat{\C\Delta}$ and
$\widehat{\dd\R\Delta}$ is the proper transform of
$\dd\R \Delta$ in $\widehat{\C\Delta}$.
Then a holomorphic disk $D$ lifts to a holomorphic disk $\hat D$
with the boundary in the non-compact Lagrangian subvariety $\hat L\subset\widehat{\C\Delta}$
without boundary. Furthermore, the Maslov index of $\hat D$ is 0.

%% file: s3a.tex
\section{Proofs}
\subsection{Proof of Proposition \ref{2kdeg} and Theorems \ref{thm-larea}, \ref{thm-lrot} and \ref{thm-index}}
Consider the map $\Arg:\ctor\to(\R/2\pi\Z)^2$ defined by
\begin{equation}\label{arg}
\Arg(z,w)=(\arg(z),\arg(w))
\end{equation}
and
the map $2\Arg:\ctor\to(\R/\pi\Z)^2$ obtained by multiplication
of $\Arg$ by 2,
in other words a composition of $\Arg$
with with the quotient map
$(\R/2\pi\Z)^2\to(\R/\pi\Z)^2$.
The involution of complex conjugation in $\ctor$ descends to
$(\R/\pi\Z)^2$ as the involution $\sigma:(a,b)\mapsto (-a,-b)$.
Denote with
\begin{equation}\label{pillow}
P=(\R/\pi\Z)^2/\sigma
\end{equation}
the quotient space.
The orbifold $P$ is the so-called {\em pillowcase}.
The projections of the four points $(0,0),(\frac\pi 2,0),
(0,\frac\pi 2),(\frac\pi 2,\frac\pi 2)$
form the $\Z_2$-orbifold locus of $P$ (the corners of the pillowcase).
All other
points are smooth.
We denote with $0\in P$ the origin of $P$,
i.e. the projection of $(0,0)$.
Note that $(2\Arg)^{-1}(0,0)=\rtor$.
The product volume form
on $(\R/\pi\Z)^2$ defines the volume
form $d\Vol_P$ on the smooth 
points of the orbifold $P$
since the involution $\sigma$
is orientation-preserving.

Let $\R C$ be a real curve of type I with 
real or purely imaginary
coordinate intersection.
Consider the surface $S^\circ=S\setminus\nu^{-1}(\dd\C\Delta)$,
where $S$ is the component of $\C\tilde C\setminus\R \tilde C$
corresponding to the orientation of $\R C$ and $\nu$ is
the normalization map \eqref{nu}.
Denote with
\begin{equation}\label{beta}
\beta:S^\circ\to P
\end{equation}
the composition of the map $2\Arg|_{S^\circ}:S^\circ\to (\R/\pi\Z)^2$
and the projection $(\R/\pi\Z)^2\to P$.

Let $p\in P$ be a regular point of $\beta$.
A point $q\in\beta^{-1}(p)$ is called {\em positive} (resp. {\em negative})
if locally near $q$
the map $\beta$ is an orientation-preserving
(resp. orientation-reversing)
open embedding. The difference between the number of positive and negative
points in $\beta^{-1}(p)$ is called {\em the degree at $p$}.
A priori, since $\beta$ is a non-proper map, the degree at different
points could be different.

\begin{lem}\label{lem-larg}
We have
\begin{equation}\label{larg}
\larea(\R C)=\int\limits_{S^\circ}\beta^*d\Vol_P.
\end{equation}
Furthermore, the degree of $\beta$ at a generic point of $P$ is
$2k(\R C)$.
\end{lem}
\begin{proof}
Consider the form
\begin{multline}
\frac{dx}x\wedge\frac{dy}y=(d\log|x|+id\arg(x))\wedge (d\log|y|+id\arg(y))=\\
d\log|x|\wedge d\log|y| - d\arg(x)\wedge d\arg(y)+\\
i(d\log|x|\wedge d\arg(y)+d\arg(x)\wedge d\log|y|)
\end{multline}
on $\ctor$.
As it is a holomorphic 2-form, it must restrict to the zero form
on any
holomorphic curve in $\ctor$. In particular, the real part of this
form must vanish everywhere on $S^\circ$, so that
$d\log|x|\wedge d\log|y| = d\arg(x)\wedge d\arg(y)$ on $S^\circ$,
and thus \eqref{larg} holds, cf. \cite{Mi14}.

The smooth map $\beta:S^\circ\to P$ extends to a continuous map
\begin{equation}\label{barbeta}
\bar\beta:\bar S\to P
\end{equation}
for a surface with boundary $\bar S\supset S^\circ$
such that $S^\circ=\bar S\setminus\dd\bar S$.
Each
$p\in\C \bar C\cap\C E_j$ corresponds
to a geodesic in $(\R/\pi\Z)^2$ in the direction
parallel to $\nuu(E_j)$
for a side $E_j\subset\dd\Delta$,
cf. \cite{MiOk}.
Since $\Fr^{\Delta}(p)\in\R\Delta$
the corresponding geodesic passes through two of the points
$(0,0),(\frac\pi 2,0), (0,\frac\pi 2),(\frac\pi 2,\frac\pi 2)$.
The image of this circle in $P$ is a geodesic segment connecting
the corresponding $\Z_2$-orbifold points of $P$.

Thus $\bar\beta(\bar S)$ is a 2-cycle in $P$
and it covers a generic point
$l$ times (counted algebraically), where $l$ is a number
independent on the choice of a generic point.
But then $\int\limits_{S^\circ}\beta^*d\Vol_P=l\Area(P)=l\pi^2/2$.
Since we have already proved \eqref{larg} we get
$l=2\frac{\Area(\R C)}{\pi^2}=2k(\R C)$
(the last equality is the definition of $k(\R C)$).
\end{proof}

Note that this lemma implies Proposition \ref{2kdeg}.
\begin{proof}[Proof of Theorem \ref{thm-larea}]
We have $k(\R C)\in\frac12\Z$ since $2k(\R C)$ is the degree of $\beta$ at a generic
point of $P$ by Lemma \ref{lem-larg}.
Let $\tilde a\in (\R/\pi\Z)^2$ be a generic point and
$a\in P$ be the point corresponding to $\tilde a$.
The inverse image $\beta^{-1}(a)$ consists of points of $S^\circ$ mapped
to $\tilde a$ or $\sigma(\tilde a)$.
If $2\Arg(p)=-\tilde a$ for $p\in S^\circ$
then $2\Arg(\conj(p))=\tilde a$,
where $\conj(p)\in\conj(S^\circ)$.
Thus we have a 1-1 correspondence between sets
$\beta^{-1}(a)$
and $R=(2\Arg)^{-1}(\tilde a)\cap{\C C^\circ}$.

Consider the continuous involution
$\conj_{\tilde a}:\C\Delta\to \C\Delta$
extending
the involution of $\ctor$ defined
by $z\mapsto e^{i\tilde a}\conj(e^{\sigma(\tilde a)}(z))$.
Note that the fixed point locus of this involution in
$\ctor$ coincides with $(2\Arg)^{-1}(\tilde a)$, cf. \cite{Mi14}.
Note that
\begin{equation}\label{Q}
R\subset \C C^\circ\cap \conj_{\tilde a}(\C C^\circ)
\end{equation}
while $R\setminus (\C C^\circ\cap\conj_a(\C C^\circ))$
consists of pairs of points interchanged by the involution
$\conj_{\tilde a}$.
For generic $\tilde a$
the curve $\conj_{\tilde a}(\C C^\circ)$ is transverse to $\C C^\circ$,
while $\conj_{\tilde a}(\C C^\circ)\cap\C C^\circ\cap\dd\C\Delta=\emptyset$.

Thus the number of points in $R$ is not greater
than
$\#(\C C^\circ\cap\conj_{\tilde a}(\C C^\circ))$,
while we have $\#(\C C^\circ\cap\conj_{\tilde a}(\C C^\circ))
=2\Area(\Delta)$
by the Kouchnirenko-Bernstein theorem \cite{Kou}.
Thus the degree of $\beta$ takes values between $-2\Area(\Delta)$
and $2\Area(\Delta)$.
Also $\#(R)=2\Area(\Delta)$.
\end{proof}

\begin{proof}[Proof of Theorem \ref{thm-lrot}]
Let us compute the degree of the map \eqref{barbeta} at a generic point $a\in P$
close to the origin $0\in P$.
The set $\bar\beta^{-1}(0)\cap S^\circ$
contributes $2E$ to the degree of $\bar\beta$
as the intersection number gets doubled when we pass from $(\R/\pi\Z)^2$ to $P$.

Note that the set
$S_{\R}=\bar\beta^{-1}(0)\cap \dd \bar S$
can be thought of
as the topological closure of $\R C^\circ$ in $\bar S$
by our assumption of transversality to $\dd\R\Delta$.
\ignore{
Let $V\subset P$ be a small regular neighborhood
of $(0,0)\in P$.
Denote $U=\bar\beta^{-1}(V)$.
Each component $K\subset U$ is a neighborhood
of a component of $S_{\R}$.
The boundary $\dd K$ of the open set $K\subset\bar S$
is homeomorphic to the corresponding component
$K_{\R}\subset S_{\R}$. 
In its turn $K_{\R}$ is
the closure of a component $C_K$ of $\R C^\circ$.

Let $U\subset \bar S$ be a small regular
neighborhood of $S_{\R}$.
We have $U\approx S_{\R}\times[0,\epsilon)$.
with $S_{\R}\times\{0\}$ corresponding to $S_{\R}$.
}
Consider a non-vanishing 
tangent vector field ${\gamma}$
on $\R\tilde C^\circ$ such that it extends
to a tangent vector field on $\R\tilde C$
with simple zeroes at
$\R\tilde C\setminus\R\tilde C^\circ$.
Our condition on zeroes of ${\gamma}$
implies that $\pm i{\gamma}$ is consistent
with a trivialized regular neighborhood
$U\approx S_{\R}\times [0,1)$
(we take $i{\gamma}$ on the components
of $S_{\R}$ where ${\gamma}$ agrees with
the complex orientation of $\R C^\circ$
and $-i{\gamma}$ otherwise). 
The lift $\tilde\beta_\epsilon$ of 
$\bar\beta|_{S_{\R}\times\{\epsilon\}}$
to $\R^2/\{(x,y)\sim (-x,-y)\}$ is approximated
(to the first order by $\epsilon$)
by $\epsilon {\gamma}$ for small $\epsilon>0$.
Thus the linking number of the image of
$\tilde\beta_\epsilon$
and
$(0,0)\in\R^2/\{(x,y)\sim (-x,-y)\}$ is 
$\lrot(\R C)$.
%
%
%
%
%
Thus
$S_\R$ contributes $-\lrot(\R C)$ to the degree of $\bar\beta$.
We have the appearance of the negative sign since
the basis composed of vectors $v_1,v_2,iv_1,iv_2$
is negatively-oriented in $\C^2$ whenever vectors $v_1,v_2$
are linearly independent over $\C$.
Thus a positive rotation in $(i\R)^2$ (and therefore also in $P$)
corresponds to a negative contribution to the degree of $\bar\beta$.
\end{proof}

\ignore{
\begin{proof}[Proof of Proposition \ref{closedK}]
If $\R C$ has toric type I then the map \eqref{beta}
lifts to the universal covering $\R^2$ of the torus $(\R/2\pi\Z)^2$.
The boundary of the push-forward of $S^\circ$ to $\R^2$ multiplied by $\frac 1\pi$ is
a representative of $\Sigma(\R C)$, so the broken line $\Sigma(K)$
must be closed.
Conversely, if $\R \tilde C$ is an $M$-curve and $\C\tilde C\setminus\R\tilde C$
is disjoint from $\dd\C\Delta$ then
$H_1(S^\circ)$ is generated by the ends of $S^\circ$ which in turn
corresponds to $\Sigma(\R C)$ and the compact components of $\R C^\circ$.
Any compact connected component of $\R C^\circ$
is null-homologous in $\ctor$ as it is contained
in a single quadrant of $\rtor$.
\end{proof}

\begin{proof}[Proof of Lemma \ref{prop-co}]
The surface $M$ is obtained from $\bar S$
by contracting all components of $\dd\bar S$
corresponding to compact components
of $\R C^\circ$.
The map \eqref{mapM} is obtained by lifting the map \eqref{beta} to the universal covering
space $\R^2$
of the orbifold $P$ followed by multiplication by $\frac 1\pi$.
We have a correspondence between $\alpha^{-1}(\Z^2)\cap M^\circ$
and the set of compact components of $\R C^\circ$.
The map $L\circ\alpha|_{M^\circ}$ is harmonic and thus cannot have local maxima
by the maximum principle.
\end{proof}
}
\begin{proof}[Proof of Theorem \ref{thm-index}]
The map \eqref{liftS}
gives the lift of $\beta|_{S^\circ}$ to the
universal covering $\C^2$ of $\ctor$ after rescaling
each coordinate by $\pi$.
Thus the signed area of $\beta(S^\circ)$ coincides
with $\pi^2\Area\Sigma(\R C)$.
Lemma \ref{lem-larg} now implies
that $k(\R C)=\Area\Sigma(\R C)$.
For $(a,b)\in\Z^2$ and $\epsilon\in\tilde{\rp}^1\setminus\tilde{{\mathbb Q}{\mathbb P}}^1$
we consider a point $p_\epsilon$ obtained by a small translation of $(a,b)$
in the direction of $\epsilon$. A point of $S^\circ$ mapped to $p_\epsilon$
by the lift of $\beta$ must correspond to a point of $\R C^\circ$ of
real index $(a,b)$ which is either
singular or has $\epsilon$ as the image of its logarithmic Gau{\ss} map.
Summing up the contributions of all such points we get \eqref{co-ab}.
\end{proof}

\subsection{Proof of invariance}
\begin{proof}[Proof of Theorem \ref{Rinv}]
First we compute the dimension of the space of
rational curves $\R C^\circ\subset\rtor$ with the Newton polygon $\Delta$.
Two coordinate functions on $\ctor$ define two meromorphic
functions on the Riemann surface $\C\tilde C$.
The set of zeroes and
poles of these functions is $\dd \C\tilde C=
\C \tilde C\cap\dd\C\Delta$.
The order of each of
these zeroes and poles is determined by the
multiplicity  of the corresponding intersection points of $\C\bar C$
with $\C E_j$ as well as by the slope of $E_j\subset\Delta$.
We may consider $\dd \C\tilde C$ as
an $m$-tuple of points in $\C \tilde C$
(i.e. an element of the $m$th power 
of $\C \tilde C$).
As $\C \tilde C$ is rational, we may freely deform
this $m$-tuple
while each such deformation extends to the deformation
of the coordinate functions.
The group $PSL_2(\R)$ of symmetries
of $\R\tilde C$
is 3-dimensional, so the space of real deformations
of the $m$-tuple $\dd\C \tilde C$ in $\tilde C$
is $(\#(\dd\C\tilde C)-3)$-dimensional.
The resulting curve is well-defined
up to the multiplicative translation in $\rtor$.
Altogether,
the space of rational curves in $\rtor$
with $m$ distinct ends is a smooth manifold 
of dimension $m-1$.
This number coincides with the dimension
of the space of configurations $\PP$.

Note that a deformation of a single point $p$ to
$p'$ in the $m$-tuple $\dd\C\tilde C$
can be achieved through multiplying
coordinate functions
by a rational function with zeroes
and poles only at $p$ and $p'$.
A non-immersed point of $\C C^\circ$ corresponds
to a common zero of the differentials of the coordinate functions.
Clearly the coordinate functions can be deformed
independently of each other. 
Thus a generic rational curve in $\ctor$
with the Newton polygon $\Delta$ is immersed.

By curves in $\C\Delta$ we mean parameterized
curves, i.e. maps from abstract curves to $\C\Delta$.
For our treatment it is convenient
to consider not only
irreducible rational curves as sources for these maps,
but also the so-called {\em stable rational curves}
with marked points (cf. \cite{Mumford} and \cite{KM}).
These are 
(perhaps reducible)
nodal curves composed
from rational irreducible components so that
the dual graph (formed by irreducible
components as vertices and nodes as edges) is a tree.
The marked points are required to be disjoint from
the nodes.
If an irreducible
component of such curve is adjacent to $j$
other components and contains less than $3-j$
punctures then it can be contracted to
a point (a smooth point if $j=1$ and
to a node if $j=2$) in a new nodal curve
so that the images of the marked points are disjoint
from the nodes. These operations generate equivalence
relation on the space of abstract rational curves
with marked points.

By stable rational curves in $\C\Delta$ we
mean continuous maps from stable rational curves
with $m$ marked points
to $\C\Delta$ so that the restriction to each
irreducible component is holomorphic, and
the marked points are mapped to $\dd\C\Delta$.
Two such maps are considered
to be equivalent if the source curves are
equivalent, and each contraction from the equivalence
contracts a component where the restriction of our map
is constant. Clearly the images of equivalent stable
curves in $\C\Delta$ coincide with each other.

Suppose that the image
$\C\bar C$
of a stable curve in $\C \Delta$ is transversal to
the boundary $\dd\C\Delta$ and has $\Delta$ as its
Newton polygon.
In this case $\dd\C\bar C=\C\bar C\cap\dd\C\Delta$
consists of $m$ points which we treat as marked points
of $\C\bar C$, or punctures of
$\C C^\circ=\C\bar C\cap\ctor$.
Resolution of singularities allows us to present
any such $\C \bar C$ as an image of a stable rational
curve with $m$ marked points.
Furthermore, for the source curve in this case 
we may choose a stable rational curve so that
none of its component is mapped to $\dd\C\Delta$.

Conventional rational curves are stable rational curves
with a single irreducible component.
Any sequence of (stable) rational curves in $\C\Delta$
with the Newton polygon $\Delta$
admits a subsequence
converging to a stable rational
curve in $\C\Delta$ while the Newton
polygon of the limit
is a subpolygon $\Delta'\subset\Delta$.

Let $\MM_{\Delta}$ be the space
of real curves $\R C^\circ\subset\rtor$
(with non-empty real normalizations
$\R\tilde C^\circ\to\R C^\circ$)
such that their Newton polygon is $\Delta$,
the intersection
$\C\bar C\cap\dd\C\Delta$
consists of $m$ distinct points,
and all components
of the curves $\C\bar C$ are rational.
Let $\MM_{\dd\Delta}$ be the space of $\conj$-invariant Menelaus
configurations of $m$ distinct
points in $\dd\C\Delta$
with $m_j$ points in $\C E_j$,
$m=\sum m_j=\#(\dd\Delta\cap\Z^2)$.
\ignore{
Consider a curve $\R C^\circ\in\MM_\Delta$.
Its deformations are described by deformations
of its irreducible components which are
parameterized, as we have already seen,
by smooth manifolds (of dimension less by one than
the number of ends of the corresponding component).
\begin{prop}
\end{prop}
}
This space is an open set in the real
part of the product of $m$ copies of $\C^\times$
(for a choice of a real structure on that product).
Thus $\MM_\Delta$ is a smooth manifold.

\begin{prop}\label{propprop}
The map
\begin{equation}\label{Mmap}
\ev:\MM_{\Delta}\to\MM_{\dd\Delta}
\end{equation}
sending a curve $\R C^\circ$ to the configuration
$\Fr^\Delta(\C \bar C)\cap\dd\C\Delta$
is proper.
\end{prop}
\begin{proof}
Each curve 
from $\MM_\Delta$ can be presented
as an image of a stable rational curve in $\C\Delta$,
and stable rational curves form a compact space
(see \cite{Mumford},
\cite{KM}).
Suppose that we have a sequence of curves
in $\MM_\Delta$ such that their image 
is converging to a point in $\MM_{\dd\Delta}$.
Consider an accumulation point for the sequence. 
It is a map $h:A\to\C\Delta$ from a stable
rational curve $A$ holomorphic on all the components
of $A$.
If no components of $A$ are mapped to
$\dd\C\Delta=\C\Delta\setminus\ctor$ then 
$h(A)\cap\ctor$ has $\Delta$ as its Newton polygon,
otherwise the Newton polygon is a subpolygon of $\Delta$.

It is enough to prove that
no irreducible component $K\subset A$
can be mapped to $\dd\C\Delta$.
Note that $h(K)$ must be contained in
a single boundary divisor of $\C\Delta$,
and intersects
two other boundary divisors at two points $v_1,v_2$,
corresponding
to two vertices of $\Delta$.
In particular the points $v_j$, $j=1,2$,
are disjoint from any configuration of $\MM_{\dd\Delta}$.
Thus each point in the inverse image 
$h^{-1}(v_j)$ must be a node of $A$,
otherwise any perturbation of $h$ in the class
of stable rational curve would intersect $\dd\C\Delta$
at a point close to $v_j$.
But the graph $\Gamma$
dual to the arrangement of components
of $A$ is a tree. Thus there must exist a component $K$
adjacent to a single node. If this component is mapped
to $\dd\C\Delta$ we get a contradiction. If not,
we form a new tree $\Gamma'$ by removing the one-valent
vertex of $\Gamma$ corresponding to $K$ and proceed
inductively. Eventually we find a component of $A$
such that at most one of its adjacent components
is mapped to $\dd\C\Delta$. Thus at most one of
the vertex points $v_1,v_2$ corresponding to
this component can have
the preimage consisting entirely of nodes and we
get a contradiction.
\end{proof}
\begin{prop}\label{w-generic}
For a generic configuration
$\PP\subset\MM_{\dd\Delta}$
the set $\ev^{-1}(\PP)$ is finite
and consists of irreducible
rational curves
$\C C^\circ$ immersed to $\ctor$.

The space $\MM_\Delta$ is a smooth $(m-1)$-dimensional
manifold outside
of a subset of codimension 2.
For a generic path
$\{\PP_t\}_{t\in[0,1]}\subset\MM_{\dd\Delta}$
connecting two generic configurations $\PP_0$
and $\PP_1$
the inverse image $\ev^{-1}(\{\PP_t\}_{t\in[0,1]})$
is contained in the smooth part of $\MM_\Delta$.
Furthermore,
it consists of curves $\C C^\circ$ with at most two
irreducible components. 
If $\C C^\circ\in \ev^{-1}(\{\PP_t\}_{t\in[0,1]})$
is not irreducible then
both components of $\C C^\circ$
are immersed real rational curves intersecting
transversely with each other.
\end{prop}
\begin{proof}
In the beginning
of the proof of Theorem \ref{Rinv}
we have already shown that 
a generic irreducible curve from $\MM_\Delta$
is smoothly immersed.
\ignore{
We have already noted that a deformation
of an irreducible curve $\C C^\circ$ is determined
by the deformations of the punctures inside
the Riemann sphere $\C\tilde C$ up to a multiplicative
translation in $\ctor$.
A non-immersed point of $\C C\circ\subset\ctor$
corresponds to a common zero of differentials
of two coordinates of $\ctor$ defined so that
they have zeroes (or poles) of given order
at the puncture points. Since the Newton polygon
$\Delta\subset\R^2$ has non-empty interior, we
may move the puncture points so as to perturb
all zeroes of the differential of one of the 
coordinates without changing the zeroes of
another coordinate thus ensuring that the
result of perturbation is an immersed curve.
}
The same argument with perturbation
of the puncture points of $\C C^\circ$
in the Riemann sphere $\C\tilde C$
shows that the irreducible locus of $\MM_\Delta$ is
a smooth $(m-1)$-manifold.
The dimension of the stratum of $\MM_\Delta$
consisting of
$a$ irreducible components
is equal to the sum of the corresponding dimensions,
and thus to $m-a$.
This implies the first part of the proposition.


For the second part we have to study $\MM_\Delta$
outside of a subset of codimension 2.
The locus of reducible curves with two components
has codimension 1 in $\MM_\Delta$.
The set of curves with
two components such that one of them is not
smoothly immersed or such that their
intersection is not transverse
has codimension 1 in the set of reducible curves,
and thus it has codimension 2 in $\MM_\Delta$.
Thus outside of a set of codimension 2 the curves
in $\MM_\Delta$ are nodal.
Any nodal curve in $\C\Delta$ is uniquely parameterized by
an abstract stable curve once we specify
a subset of nodes coming from the nodes of the source
curve.
Thus smoothness of $\MM_\Delta$ outside of a subset
of codimension 2
follows from smoothness of the space of stable
rational curves with $m$ punctures.
\end{proof}

\ignore{
Any two generic configurations in
$\MM_{\dd\Delta}$ may be connected with a path
$\{\PP_t\}$, $t\in [0,1]$, such that 
the set $\RR_{\Delta,k}(\PP)$ of stable
rational curves
$\R C^\circ\subset\rtor$ with the Newton
polygon $\Delta$, the quantum index $k$,
and satisfying to
$\Fr^{\Delta}(\C C)\supset\PP$ is finite,
and consists of immersed and irreducible
rational curves. 
}


Consider the space $\MM_{\dd\Delta,\lambda}\subset\MM_{\dd\Delta}$
of real configurations $\PP\subset\R\Delta$ with $\lambda_j$ points in $\R E_j\setminus (\overline{(\R_>0)^2})$.
Note that the space $\MM_{\dd\Delta,\lambda}$ 
is connected.
Any curve in $\ev^{-1}(\MM_{\dd\Delta,\lambda})$ satisfies to the hypothesis
of Theorem \ref{thm-larea}, thus its quantum index is well-defined
(once the orientation of each irreducible
component is fixed).
Let $\PP,\PP'\in\MM_{\dd\Delta,\lambda}$
and $\gamma=\{\PP_t\}_{t\in [0,1]}\in\MM_{\dd\Delta,\lambda}$
be a smooth generic path connecting two
generic configurations $\PP=\PP_0$ and $\PP'=\PP_1$
from $\MM_{\dd\Delta,\lambda}$.

Let ${\mathcal C}\subset 
\ev^{-1}(\gamma)$
be a connected component.
By Proposition \ref{w-generic},
since the path $\{\PP_t\}$ is generic,
the component $\mathcal C$ is a smooth manifold and thus
is diffeomorphic to an interval or a circle.
All but finitely many values of $t$ correspond to $\PP_t$
such that $\RR_{\Delta,k}(\PP_t)$ consists
of smoothly immersed irreducible curves.

\begin{prop}\label{wlem}
If ${\mathcal C}$ is disjoint 
from 
the locus of reducible curves in $\MM_{\Delta}$
then 
\begin{equation}\label{locW}
\sum\limits_{\R\bar C\in\RR_{\Delta,k}(\PP_0)
\cap{\mathcal C}}\sigma(\R C)=
\sum\limits_{\R\bar C\in\RR_{\Delta,k}(\PP_1)
\cap{\mathcal C}}\sigma(\R C).
\end{equation}
\end{prop}
Note that our definitions of the sign
$\sigma(\R C)$ depends on the orientation
of $\R C$
(recall that $\MM_\Delta$ is defined
as the space of {\em oriented} real curves).
With the help of 
\eqref{sigma-w} we may reformulate the proposition
in terms of Welschinger's signs
$w(\R C)=(-1)^{E(\R C)}$ which are independent
on the choice of the orientation of $\R C$ as
\begin{equation}\label{locWw}
\sum\limits_{\R\bar C\in\RR_{\Delta,k}(\PP_0)
\cap{\mathcal C}}w(\R C)=
\sum\limits_{\R\bar C\in\RR_{\Delta,k}(\PP_1)
\cap{\mathcal C}}w(\R C)
\end{equation}
since the quantum index $k$ is constant
on $\mathcal C$.
In this form Proposition \ref{wlem} may be viewed
as a version of Welschinger's invariance.
To prove the proposition we need to recall 
orientation constructions for the space
of sections of real vector bundles over
an oriented real rational curve.
We do it in a series of auxiliary lemmas.

Let $\pi_A$ 
be a line bundle
of degree $a\ge 0$
over an oriented real rational (non-empty) curve
$\R K$, $\xi$ be its section with
distinct real zeroes $x_1,\dots,x_a\in \R K$,
and $x_0\in\R K$ be a point such that 
$x_0,x_1,\dots,x_a,\in \R K$
are numbered in the order 
consistent with the orientation
of $\R K$.
For each $j=1,\dots,a$, consider 
a section $\xi_j$ of $\pi_A$ 
whose zeroes are obtained from $\{x_l\}_{l=1}^a$
by moving the point $x_j$ in the positive
direction in $\R K$ while leaving the other
points unchanged.
In addition we assume that
$\xi(x_0)$ and $\xi_j(x_0)$
are positive multiples of each other.

%
\ignore{
Choose an orientation of $\pi_A$
over $\R K\setminus\{x_{\infty}\}$, so that
we may speak of the sign of a section of $\pi_A$
at a point other than $x_{\infty}$.
If $a$ is even then an orientation over
$\R K\setminus\{x_{\infty}\}$ canonically
extends to an orientation of $\pi_A$.

These data determine
(up to positive scalar multiplication)
$a+1$ sections
$\xi_0,\dots,\xi_a\in\Gamma(\pi_A)$ by the conditions
$\xi_j(x_k)=0$ for $j\neq k$ and $\xi_j(x_j)>0$,
$j,k=0,\dots,a$.
Here $\Gamma(\pi_A)$ stand for the space of sections
of the line bundle $\pi_A$.
}

\newcommand{\orr}{\operatorname{or}}
\begin{lem}\label{olb}
For a real algebraic line bundle $\pi_A$
of degree $a\ge 0$ over an oriented
rational real curve $\R K$
the sections $\xi,\xi_1,\dots,\xi_a$ form a basis
of $\Gamma(\pi_A)$. The orientation
$\orr_A$ of the vector
space $\Gamma(\pi_A)$ determined
by $\xi,\xi_1,\dots,\xi_a$ depends only on the orientation
of the curve $\R K$
if $a$ is odd.
If $a$ is even then $\orr_A$ depends
on the orientation
of the curve $\R K$ as well
as the orientation of $\pi_A$ 
determined by the non-zero vector $\xi(x_0)$.
\end{lem}
\begin{proof}
Since the points $x_1,\dots,x_a$ are distinct,
any linear combination of the sections $\xi_j$
vanishing at $x_l$ must have zero as the coefficient 
at $\xi_j$.
The sections $\xi_j$ along with $\xi$ form a basis
since $\dim\Gamma(\pi_A)=a+1$ and $\xi\neq 0$.
All configurations of $a+1$
distinct points in $\R K$ in the same cyclic order
are isotopic to each other.
Reversing the orientation of $\pi_A$ over $x_0$
results in reversing the sign for all sections
of our basis.
The resulting basis is of the same sign
if and only if $a+1$ is even.
\end{proof}

Let now $\pi_E$ 
be a real algebraic 2-dimensional
vector bundle of degree $d$ over $\R K$.
We assume that $\pi_E$ is generated by global
sections. Since $\pi_E$ decomposes into the sum
of two line bundles (as $\R K$ is rational),
this assumption is equivalent to non-negativity
of the degrees of both line bundles. 
In particular, we have $d\ge 0$ and
$\dim \Gamma(\pi_E)=d+2\ge 2$.

Suppose that $\pi_A:\R A\to\R K$
is a (1-dimensional) subbundle
of $\pi_E$
of non-negative
degree $a$. Let $\pi_{E/A}$ be the corresponding
quotient line bundle. 
We get the following short exact sequence
\begin{equation}\label{Gs}
0\to\Gamma(\pi_A)\to\Gamma(\pi_E)\to
\Gamma(\pi_{E/A})\to 0
\end{equation}
which allows us to combine orientations
of $\Gamma(\pi_A)$ and $\Gamma(\pi_{E/A})$
into an orientation of $\Gamma(\pi_E)$.
Note that if $a$ is even then
changing the sign of $\xi(x_0)$
also changes $\orr_A$.
Thus the following statement is thus
a corollary of Lemma \ref{olb}.
\begin{coro}
\label{coroor}
Let $\pi_E$ be a real algebraic 2-dimensional
vector bundle of degree $d$ generated
by global sections.
Let $\pi_A$
be its real algebraic 1-dimensional
subbundle of degree $a\ge 0$.

If $d$ is even then
the orientation of $\Gamma(\pi_E)$
is defined by $\pi_A$ together
with an orientation of the bundle $\pi_E$.

If $d$ and $a$ are odd then
the orientation of $\Gamma(\pi_E)$
is defined by $\pi_A$ together
with an orientation of the bundle $\pi_{E/A}$.

If $d$ is odd while $a$ is even then
the orientation of $\Gamma(\pi_E)$
is defined by $\pi_A$ together
with an orientation of the bundle $\pi_{A}$.
\end{coro}
Denote the resulting orientation of 
$\Gamma(\pi_E)$ with $\orr_{E,A}$.
Consider the decomposition $\pi|_E=\pi_b\oplus
\pi_c$ into a direct sum of two line subbundles
of degrees $b$ and $c$ respectively, $b\ge c\ge 0$.

A section of $\pi_A$ is a sum of a section
$\sigma_b$ 
of $\pi_b$ and a section $\sigma_c$ of $\pi_c$.
Note that the degree $a$ is equal
to the number of common zeroes of 
$\sigma_b$ and $\sigma_c$
(counted with multiplicities).
Thus varying $\sigma_b$ and $\sigma_c$ we
get a deformation of the line subbundle $\pi_A$
within the line subbundles of the same degree
as long as the common zeroes are deformed together,
and no new common zeroes are introduced
(even if at some instant of a deformation a zero
of $\sigma_b$ or $\sigma_c$ passes through 
one of the common zeroes).
The orientation $\orr_{E,A}$ is unchanged
under such a deformation.

\newcommand{\p}{\mathbb P}
\newcommand{\Ss}{\mathcal S}
Consider the circle bundle
$F_+\to\R K$ obtained by positive
projectivization of $\pi_E$.
Here $F_+$ is a (topological) surface
obtained from the $\R E$ by removing
the zero section of $\pi_E$ and identifying
the positive scalar multiples in the same fiber
of $\pi_E$.
Depending on the parity of $d=b+c$,
the surface $F_+$ is a torus or a Klein bottle.
We have $H_1(F_+;\Z_2)\approx\Z_2^2$, in particular
the subset $\Ss\subset H_1(F_+;\Z_2)$ consisting
of the classes with nontrivial image under
the homomorphism induced by $F_+\to\R K$
is a 2-element set.

Given an oriented line subbundle $\pi_A\subset\pi_E$
denote with $s(A)\in\Ss$ the homology class
of the projectivizations of the positively-directed
vectors in $\pi_A$.
Given a non-orientable line subbundle $\pi_A\subset\pi_E$
($a$ is odd)
and a local orientation of $\pi_E$ over $x_0$
we may add (topologically) a left half-twist 
to the subbundle $\pi_A$ to obtain
an orientable bundle and proceed as above
to define $s(A)\in\Ss$.
Note that in the case when both $a$ and $d$ are
odd, and the orientation of $\pi_{E/A}$ is fixed
(in particular, over $x_0$) then $s(A)$ depends only
on this orientation as a different choice of 
the local orientation of $\pi_E$ over $x_0$
changes both the direction of the half-twist
and the orientation of the line bundle after
the half-twist. Thus 
\ignore{
If the quotient
bundle $\pi_{E/A}$ is oriented then
we may choose a topological (not necessarily algebraic)
oriented line subbundle $\pi_A^{\perp}$
of $\pi_E$ such that
the restriction 
of $\pi_E\to\pi_A$ to $\pi_A^{\perp}$
yields an isomorphism of oriented line bundles.
Denote
with $s(A)\in\Ss$ the homology class
of the projectivization of the positively-directed
vectors in such a subbundle. Clearly $s(A)$ does
not depend on such a subbundle. Furthermore,
if both $\pi_A$ and $\pi_{E/A}$ are orientable,
then $F_+$ is a torus and the homology classes
of the positive projectivization of $\pi_A$,
$\pi_A^\perp$ are the same, and also do not depend
on the choice of orientations in
$\pi_A$ and $\pi_A^\perp$. Thus our notation $s(A)$
is consistent.

By contrast, if $F_+$ is a Klein bottle,
then reversing the orientation of the orientable
bundle from the pair $\pi_A$, $\pi_{E/A}$ changes
$s(A)$ to the other element of $\Ss$.
Finally, if both $\pi_A$ and $\pi_{E/A}$
are non-orientable, but $\pi_E$ is oriented
(so that $F_+$ is a torus) then
we use the orientation of $\pi_E$ (and $\R K$)
to subtract a half of the positive twist from
the bundle $\pi_A^\perp$ and arrive to an 
orientable line bundle in $\pi_E$.
Denote with $s(A)\in\Ss$ the homology class
of the positive projectivization of that
subbundle (enhanced with any orientation).

We see that $s(A)\in \Ss$ is defined once
an orientation of one the three bundles
$\pi_A$, $\pi_{E/A}$ or $\pi_E$ is fixed.
}
the class $s(A)$ is well-defined 
once we fix the same orientation data
 as in Corollary \ref{coroor},
i.e. an orientation of one of 
$\pi_A$, $\pi_{E/A}$ or $\pi_E$ depending 
on parity of $d$ and $a$.
Furthermore, a different choice of this
orientation results in the different value for $s(A)$.
\begin{lem}\label{orrs}
Let $\pi_{A_0}$ and $\pi_{A_1}$ be
two line subbundles of $\pi_E$ of the same degree
$a\ge 0$.

We have $\orr_{E,A_0}=\orr_{E,A_1}$
if and only if
$
s(A_0)=
s(A_1).$
\end{lem}
\begin{proof}
Consider generic sections $\xi_{A_0}$
and $\xi_{A_1}$
of $\pi_{A_0}$ and $\pi_{A_1}$
with $a$ 
distinct real zeroes each.
These sections can be deformed into each other
in the class of sections of $\pi_E$
with $a$ distinct real zeroes.
This deformation may be assumed generic.
Each intermediate
section defines a line subbundle
$\pi_{A_t}\subset\pi_E$, and can be decomposed
into the summands $\sigma_{b_t}$
and $\sigma_{c_t}$, $t\in[0,1]$.
If zeroes of $\sigma_{b_t}$ and $\sigma_{c_t}$ do not
collide then both $\orr(E,{A_t})$ and $s(A_t)$
stay invariant.

Suppose that two simple zeroes of $\sigma_{b_t}$ and
$\sigma_{c_t}$ collide at $t=t_0\in [0,1]$
in a point $y\in\R K$ so that the order
of these zeroes
for $t_0-\epsilon$ and $t_0+\epsilon$ gets reversed.
Pass to the fiberwise projectivization
of $\pi_E$ (which is a Hirzebruch surface
$F_{b-c}$). The projectivization of $\pi_{A_t}$
corresponds to a curve intersecting a generic fiber
once. At $t=t_0$ this curve degenerates to
the union of the fiber $F_y$ over $y$ and the residue
curve intersecting $F_y$ 
transversally.
The zeroes of a section of $\pi_{E/A}$
correspond to the intersection points
of this curve with 
another curve resulting from a section of $\pi_E$.
This curve intersects $F_y$ once.
For $t=t_0\pm\epsilon$ the fiber $F_y$ gets added
to the residue curve with different orientation,
so that the orientations of $\pi_{E/A_{t-\epsilon}}$
and $\pi_{E/A_{t+\epsilon}}$ given by Lemma
\ref{olb} are opposite.
The classes $s(A_{t-\epsilon})$ and
$s(A_{t+\epsilon})$ are also different
as the corresponding projectivization
are different by the fiber of $F_+$ over $y$.
\ignore{
Each of these sections decomposes into
the sum of the 
To prove the lemma it suffices to consider
the case when $\xi_{A_2}$ is obtained from
$\xi_{A_1}$ by colliding two simple zeroes
of the summands $\sigma_{b}$ and $\sigma_{c}$
(so that $a_2=a_1+1$).
Any other generic deformation decomposes
into a sequence of such degenerations
(or their inverse)
as well as deformations of $\pi_A$  within
the class of line subbundles of the same
degree which have no effect on $\orr_{E,A}$.
\ignore{
Consider a local orientation of $\pi_E$
over an interval $I\subset\R K$ containing
colliding zeroes of $\sigma_{b}$ and $\sigma_{c}$.
Choose the points $x_0,\dots,x_{a_1}$
and $y_0,\dots,y_{d-a_1}$ responsible for
the ..
By our construction
} 
The orientation $\orr_{E,A_1}$ is given according
to Lemma \ref{olb} by $\xi_{A_1}$ as well
as an appropriate section of $\pi_{E/A_1}$.
Zeroes of $\xi_{A_1}$ correspond to common
zeroes of $\sigma_b$ and $\sigma_c$.
}
\end{proof}
\begin{coro}\label{orrsab}
Let $\pi_A$ and $\pi_B$ be line subbundles of $\pi_E$
of degrees $a\ge 0$ and $b\ge 0$. There exists a sign
$\sigma(a,b)=\pm 1$ depending only on the numbers $a$ and $b$
such that 
$\orr_{E,A}=\sigma(a,b)\orr_{E,B}$
if and only if $s(A)=s(B)$.
\end{coro}
\begin{proof}
Define $\sigma(a,b)$ so that $\orr_{E,A}=\sigma(a,b)\orr_{E,B}$
holds for two particular line subbundles $\pi_A$ and $\pi_B$ 
of degree $a$ and $b$.
This sign
depends only on $a$ and $b$ by Lemma \ref{orrs}.
\end{proof}

\begin{rmk}
Note that a $Spin$-structure on
an orientable 2-dimensional vector
bundle $\pi_E$ is a choice
of an orientation of $\pi_E$ together with
a bijective map $\Ss\mapsto\{\pm 1\}$.
In the case if $\pi_E$ is non-orientable,
a choice of bijection $\Ss\mapsto\{\pm 1\}$
is known as a $Pin_-$-structure.
(This interpretation of $Spin$ and $Pin_-$-structures
holds for any real rank 2 vector
bundle over a topological circle.)

Thus Lemma \ref{orrs} simply states that for
a given $a$
the
orientation of $\Gamma(\pi_E)$ is determined
by the $Spin$ or $Pin_-$ structure on $\pi_E$
as long as a subbundle $\pi_A$ of degree $a$
in $\pi_E$ exists.
Namely, we take $\orr_{E,A}$ 
if the value of $Spin/Pin_-$-structure on $s(A)$
is positive and $-\orr_{E,A}$ otherwise.
In the case if
$\R K=\R \tilde C$,
and $\pi_E$
is the pull-back of the tangent bundle
the $Spin/Pin_-$-structure is given by 
the logarithmic trivialization of 
$\ctor\subset\C\Delta$.
This connection
can be traced in the proof below (even
though the proof does not use it explicitly).
\end{rmk}
\begin{proof}[Proof of Proposition \ref{wlem}]
Consider the connected component
${\mathcal C}$.
%
The tangent vector fields
$(z,w)\mapsto z\frac{\dd}{\dd z}$
and $(z,w)\mapsto w\frac{\dd}{\dd w}$
on $\ctor$ are non-vanishing, and
extend to (possibly vanishing on $\dd\C\Delta$) 
holomorphic vector fields on 
$\C \Delta$ away from the
finite number of intersections of toric divisors
in $\C \Delta$.
In particular, for $\R C^\circ\in{\mathcal C}$
these vector fields provide sections
for the pull-back $\pi_{E^C}$ of the
tangent bundles
$T\C\Delta$ to the normalization $\R\tilde C$ of 
$\R \bar C$.
Therefore, the bundles $\pi_{E^C}$ are generated
by global sections.
Furthermore, one of these vector fields,
say $z\frac{\dd}{\dd z}$, defines a line subbundle
$\pi_{A^C}\subset \pi_{E^C}$ well-defined for all $\R C^\circ\in\MM_\Delta$.

The number of positive (adjacent to
the positive quadrant $(\R_{>0})^2$)
and negative points of $\PP_t\subset\dd\R\Delta$
contained in a divisor $\R E\subset\dd\R\Delta$
is determined by
$\lambda$ and is invariant of $t$.
If
$\R E\cap\dd{(\R_{>0})^2}
\cap\PP_t\neq\emptyset$
then we choose a reference point
$$p^+_t\in\R E\cap\PP_t$$
as the first point on $\R E\cap\dd{(\R_{>0})^2}$
in the linear order induced by the clockwise
orientation of $\dd{(\R_{>0})^2}$.
Note that by the definition of $\MM_\Delta$
the $m$ points of $\PP_t$ are always distinct.
Thus $p_+\in\R \bar C$ is a consistent choice
of a reference point for $\R C^\circ\in \MM_\Delta$.

Furthermore, at $p_+$ we have a canonical choice
of a local orientation of $\R\Delta$ (coming
from the standard orientation of $\R_{>0}^2$).
The orientations of $\pi_{A^C}$ and
$\pi_{E^C}/\pi_{A^C}$
in the positive quadrant and thus at $p_+$
are also standard.
Corollary \ref{coroor} provides an orientation of 
the bundle over $\MM_\Delta\ni \R C^\circ$ formed by
$\Gamma(\pi_{E^C})$, the space of 
sections of the 2-dimensional bundle
$\pi_{E^C}$ over the normalization $\R \tilde C$.

The tangent space to $\MM_\Delta$
at $\R C^\circ$ can be identified with
the quotient of the vector space
$\Gamma(\pi_{E^C})$ by its 3-dimensional subspace
$\Gamma_B$
obtained as the image of the tangent vector
fields on $\R\tilde C$ in the tangent bundle
of $\R\Delta$.
In the case when $\R \bar C$ is immersed
we have $\Gamma_B=\Gamma(\pi_{E^B})$,
where $\pi_{E_B}\subset
\pi_{E^C}$ is the tangent bundle to $\R \tilde C$.
The orientation of $\Gamma_B$ is
provided by Lemma \ref{olb}.
The orientation of $\MM_\Delta$
(i.e. the orientation of its tangent bundle)
is consistently determined by this quotient.

If no points of $\R\bar C\cap\dd\R\Delta$
are adjacent to the positive
quadrant then $\R\tilde C$ is mapped to the same 
quadrant by the normalization map
$\R\tilde C\to \R\bar C$. The standard orientation
of this quadrant induces a consistent orientation
of the bundle $\pi_{E^C}$ which has an even degree
in this case.
Thus Corollary \ref{coroor} provides (once again,
through the quotient by $\Gamma_B$)
an orientation of the corresponding
components of $\MM_\Delta$ in this case as well.

Recall that the space $\MM_{\dd\Delta}$ consists
of the $\conj$-invariant $m$-tuples of distinct
points in $\dd\C\Delta$. Thus it
is oriented
by the standard (clockwise with respect to the
positive real quadrant) orientations of
real toric divisors and the order 
of points on these divisors corresponding
to these orientations.
Indeed, this yields an order for all real points
of the $m$-tuple.
Each conjugate pair of imaginary points is contained
in $\dd\C\Delta\setminus\dd\R\Delta$.
Exactly one point in each pair is contained 
in a connected component $E^+
\subset\dd\C\Delta\setminus\dd\R\Delta$
whose complex orientation
agrees with the chosen orientation on $\dd\R\Delta$.
The conjugate pair is determined by this
{\em representative} point.
The orientation of $\MM_\Delta$ comes as the product
of the orientations of $\dd\R\Delta$
(in the induced order of the real points of the
$m$-tuple) and the complex orientations of $\dd\C\Delta$
at the representative points of the conjugate pairs.
Note that the relative order of conjugate pairs
is irrelevant for the product orientation.

Thus the map \eqref{Mmap} is a map between
oriented manifolds.
To prove \eqref{locW} it suffices to show that
the local degree of the map \eqref{Mmap} at
$\R C^\circ\in\MM_\Delta$ agrees with
$\sigma(\R C)$ (up to a global sign on $\mathcal C$)
in the case when $\ev(\R C^\circ)$ is a regular value.

Note that the tangent space
$T_{\ev(\R C^\circ)}\MM_{\dd\Delta}$
is identified with the
tangent space $T_{\R C^\circ}\MM_{\Delta}$
by the differential of the map \eqref{Mmap} 
at its regular point $\R C^\circ$, we get
\begin{equation}\label{3spaces}
T_{\ev(\R C^\circ)}\MM_{\dd\Delta}=
T_{\R C^\circ}\MM_{\Delta}=
\Gamma(\pi_{E^C}/\pi_{B^C}).
\end{equation}
The orientation of $\Gamma(\pi_{E^C}/\pi_{B^C})$
given by Lemma \ref{olb}
differs from
the
orientation of $T_{\ev(\R C^\circ)}\MM_{\dd\Delta}$
according to the cyclic order
of $\R\tilde C\cap\dd\R\Delta$
at $\R\tilde C$.
Thus these orientations agree or disagree
uniformly 
on ${\mathcal C}$.

To compare the orientations of
$T_{\R C^\circ}\MM_{\Delta}$ and 
$T_{\ev(\R C^\circ)}\MM_{\dd\Delta}$
we use the identification \eqref{3spaces}.
Each of these orientations corresponds to an orientation
of the vector space $\Gamma(\pi_{E^C})$
induced 
by the presentation
$\Gamma(\pi_{E^C}/\pi_{B^C})=\Gamma(\pi_{E^C})/
\Gamma(\pi_{B^C})$ (which holds in the case when $\R\bar C$ is immersed in $\R\Delta$).
One orientation
is defined by the subbundle 
$\pi_{E^A}\subset\pi_{E^C}$, the other one
by $\pi_{B^C}\subset\pi_{E^C}$.
By Corollary \ref{orrsab} these orientations agree or not
depending on the values
$s(E^A),s(E^B)\in{\mathcal S}$.

Note that the pullbacks of the tangent bundle of $\R\Delta$
to $\R \tilde C$ are canonically isomorphic as
vector bundles over $\R \tilde C$ (treated as a topological
space) for all $\R C^\circ\in{\mathcal C}$.
Thus we may identify the two-element sets ${\mathcal S}$
for different curves in ${\mathcal C}$.
By the definition of $\pi_{A^C}$ the elements
$s(E^A)$ agree under this identification.
In the same time the value of $s(E^B)$
depends on $\operatorname{Rot}_{\Log}(\R C)$.
Since each half-turn
at the positive projectivization $F_+$ of $\pi_E^C$
corresponds to a full turn in the projectivization of $\pi_{E^C}$,
the value of $s(B^C)\in{\mathcal S}\subset H_1(F_+;\Z_2)$
is determined
by the mod 4 residue of $\operatorname{Rot}_{\Log}(\R C)$
and thus by $\sigma(\R C)$.
\end{proof}

Suppose now that
\begin{equation}\label{Rbig}
\RR_{\Delta}(\PP_{t})=\bigcup_{k'=-\Area(\Delta)}^{\Area(\Delta)}\RR_{\Delta,k'}(\PP_t)
\end{equation}
contains a reducible curve $\R\bar D$ for $t=t_0$.

As the dimension of the space of deformations of each component is equal
to the number of points in its intersection with $\dd\C\Delta$ minus 1,
for the generic path $\gamma$ (in the
non-singular space $\MM_\Delta$)
the curve $\Fr^\Delta(\C\bar D)$ is the union of two
irreducible rational immersed curves $\Fr^\Delta(\C\bar D_j)$, $j=1,2$. These two curves define a subdivision
of the Menelaus
configuration $\PP_t$ (which is no longer generic
even among Menelaus configurations for this particular
value of $t$)
into the disjoint union of two real generic Menelaus
configurations
$\PP^{(j)}_t=\C\bar D_j\cap\dd\Delta$, $j=1,2$.
Thus each $\C\bar D_j$ is defined over $\R$.
Also, the presence of multiplicative translations in $\rtor$
implies that $\Fr^\Delta(\R\bar D_1)$ and $\Fr^\Delta(\R\bar D_2)$
intersect transversally.

We may assume that all points $\Fr^\Delta(\PP_t)$ except for two points
$p_j(t)\in \R\bar D_j$, $j=1,2$,
remain independent of $t\in [t_0-\epsilon,t_0+\epsilon]$
for a small $\epsilon>0$.
For $t\in [t_0-\epsilon,t_0+\epsilon]$ the deformation $p_1(t)\in\dd\R\Delta$
determines the deformation $p_2(t)$ by
the Menelaus condition. The points 
$p_1(t_0)$ and $p_2(t_0)$ must belong to
two different components of $\R\bar D$.
Indeed, otherwise one of the component of $\R\bar D$
contains only points of $\PP_t$ invariant of $t$
and then  
$\RR_{\Delta}(\PP_t)$ contains reducible curves
for all $t\in [t_0-\epsilon,t_0+\epsilon]$
which contradicts to our assumption.
Denote with $\RR^{\R D}_\Delta(\PP_t)$ for
$t_0-\epsilon\le t\le t_0+\epsilon$
the curves whose images under $\Fr^\Delta$ is close to $\Fr^\Delta(\R \bar D)$.

Let us choose some orientations of $\R \bar D_1$ and $\R \bar D_2$.
Then the intersection points of
\begin{equation}\label{I}
I = \Fr^{\Delta}(\R\bar D_1)\cap \Fr^{\Delta}(\R\bar D_2)
\end{equation}
come with the intersection sign in $\R_{>0}^2$.
The set of positive points $I_+\subset I$ has the same cardinality as
the set of negative points $I_-\subset I$, $I=I_+\cap I_-$
as $\R_{>0}^2$ is contractible.

The curves in $\RR_{\Delta}^{\R D}(\PP_{t_0\pm\epsilon})$
are obtained by smoothing a nodal point
$q\in I$
in one of the two ways,
one that agrees with our choice of orientation and one that does not.
Without loss of generality (changing the direction of the path $\gamma$
if needed) we may assume that the orientation-preserving
smoothing in 
a point $q\in I_+$ corresponds to a curve
$\R \bar D_{q,+}\in\RR_{\Delta}(\PP_{t_0+\epsilon})$
and thus the orientation-reversing smoothing at the same point corresponds to a curve
$\R \bar D_{q,-}\in\RR_{\Delta}(\PP_{t_0-\epsilon})$.
The following lemma determines the situation at all the other points of $I$.
\begin{lem}\label{sign}
A curve obtained by the smoothing of a node from $I_+$ in the orientation-preserving
way, or by the smoothing of nodes from $I_-$ in the orientation-reversing way belongs to
$\RR_{\Delta}(\PP_{t_0+\epsilon})$, $\epsilon>0$.
\end{lem}
Accordingly,
a curve obtained by the smoothing of a node from $I_-$ in the orientation-preserving
way, or by the smoothing of nodes from $I_+$ in the orientation-reversing way belongs to
$\RR_{\Delta}(\PP_{t_0-\epsilon})$.
\begin{proof}
Let $\R D_{q',s}\in\RR_{\Delta}(\PP_{t_0+\epsilon})$, $s=\pm 1$,
be the curve obtained by smoothing $\R D$ at a point $q'\in I$ according
to the sign $s$.
Note that $\Fr^{\Delta}(\R D_{q,+})$ and $\Fr^{\Delta}(\R D_{q',s})$
are tangent to each other at $m$ points of $\PP_{t_0+\epsilon}$
and must intersect each other at pairs of points close
to the nodes of $I\setminus\{q,q'\}$.
Tangencies contribute to $2m$ of the intersection,
this number is the integer perimeter
of the polygon $2\Delta$.
The number of nodes in $I\setminus\{q\}$
is equal to 
the number of lattice points inside
the polygon $2\Delta$ by the genus formula
since $\Fr^\Delta(\R D_{q,+})$ is a rational curve
and $2\Delta$ is its Newton polygon.
By Pick's formula the total number of intersection
points resulting from the nodes and tangencies as above
equals
twice the area of $2\Delta$.
By the Kouchnirenko-Bernstein formula, these curves do not have
any other intersection points which implies that $s=+1$ if
$q'\in I_+$, see Figure \ref{figsign}.
\end{proof}
\begin{figure}[h]
\includegraphics[height=45mm]{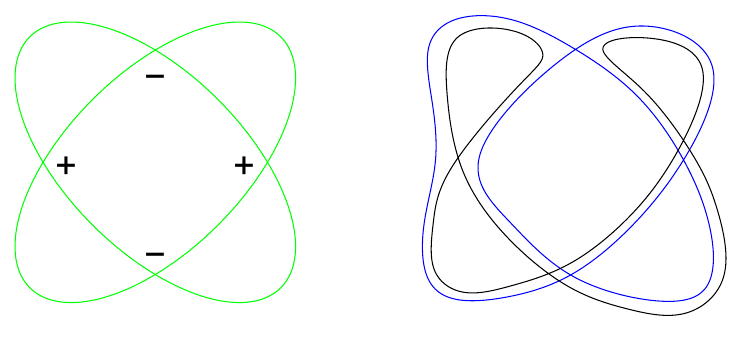}
\caption{
The signs of intersection points of two components of $\R D$ and
the corresponding direction of smoothing.
\label{figsign}}
\end{figure}

Note that any curve in $\RR_\Delta^{\R D}(\PP_{t_0\pm\epsilon})$
is obtained by 
smoothing $\R D$ at a point $q\in I$. 
The quantum index of the result is $\pm k(\R D_1)\pm k(\R D_2)$,
where the signs are determined by the agreement or disagreement
of the orientation of the resulting curve with the chosen orientations
of $\R D_j$.
Since $\#(I_+)=\#(I_-)$, Lemma \ref{sign} implies
that $\RR_\Delta^{\R D}(\PP_{t_0+\epsilon})$
and $\RR_\Delta^{\R D}(\PP_{t_0-\epsilon})$
have the same number of curves of each quantum index.
\end{proof}


\begin{proof}[Proof of Theorem \ref{tildeR}]
By Proposition \ref{wlem} we have
$\tilde R_{\Delta_d,k}(\PP)=\tilde R_{\Delta_d,k}(\PP')$
if there are no reducible curves with the Newton polygon $\Delta_d$
that pass through $\PP_t$.
Also we may assume that if $\R D$ is a reducible curve
with the Newton polygon $\Delta_d$ passing
through $\PP_{t_0}$ then it consists of two components $\R D_1$ and $\R D_2$
that intersect
transversely at a finite set $I$.
Note that the degree of both components, $\R D_1$ and $\R D_2$, must be even,
as a real curve of odd degree must intersect $\dd\rp^2$ in a negative
point as the boundary of the positive quadrant is null-homologous.

We have two smoothings of $\R D$ at $q\in I$
that pass through $\PP_{t_0-\epsilon}$
and $\PP_{t_0+\epsilon}$. One can be oriented in accordance with the orientations
of $\R D_1$ and $\R D_2$ and the other in accordance with the orientation of $\R D_1$,
but opposite to the orientation of $\R D_2$.
The corresponding quantum indices are different by $2k(\R D_2)$.
The index $k(\R D_2)$ is integer since the degree of $\R D_2$ is even.
\end{proof}

\subsection{Indices of real phase-tropical curves}
We start by recalling the basic notions of tropical geometry (cf. \cite{Mi05},
\cite{Mi06}) specializing to the case of plane curves.
Recall that a {\em metric graph} is a topological space
homeomorphic to $\Gamma^\circ=\Gamma\setminus\dd\Gamma$ enhanced with a complete inner
metric. Here $\Gamma$ is a finite graph and $\dd\Gamma$ is the set of its
1-valent vertices. The metric graph is also sometimes called
a tropical curve (while in some other instances the term tropical curve is reserved
for the equivalence class of metric graphs with respect to tropical modifications).
In this paper we require the graph $\Gamma$ to be connected so that $\Gamma^\circ$ is
irreducible as a tropical curve. We assume that $\Gamma$ has a vertex of valence
at least three,
and that $\Gamma$ does not have 2-valent vertices.
The half-open edges of $\Gamma^\circ$ obtained from the closed edges of $\Gamma$
adjacent to $\dd\Gamma$ are called {\em leaves}.

A {\em plane tropical curve} is a proper continuous map
$h:\Gamma^\circ\to\R^2$ such that $h|_E$ is smooth for every edge $E\subset\Gamma$
with $dh(u)\in\Z^2$ for a unit tangent vector $u$ at any point of $E$.
In addition we require the following {\em balancing condition}
at every vertex $v\in\Gamma$
\begin{equation}
\label{bc}
\sum\limits_E dh(u(E))=0,
\end{equation}
where $u(E)$ is the unit tangent vector in the outgoing direction
with respect to $v$ and the sum is taken over all edges $E$ adjacent to $v$.

The collection of vectors $\{dh(u_v)\}_{v\in\dd\Gamma}$ where $u_v$ is a unit
vector tangent to the leaf adjacent to $v$ (and directed towards $v$) is
called {\em the (toric) degree} of $h:\Gamma^\circ\to\R^2$.
The identity \eqref{bc} implies that the sum of all vectors in this collection is zero.
Therefore this collection is dual to a lattice polygon $\Delta\in\Z^2$ which is
well-defined up to translations in $\Z^2$. The polygon $\Delta$ is
determined by $h(\Gamma^\circ)$. 
We call $\Delta$
{\em the Newton polygon} of $h:\Gamma^\circ\to\R^2$.

Tropical curves appear as limits of scaled sequences of complex curves
in the plane. Let $A$ be any set and $\alpha\to t_\alpha\in\R$
be a function unbounded from above (this function is called
{\em the tropical scaling sequence}).
Let $\C C_\alpha\subset\ctor$, $\alpha\in A$,
be a family of complex curves with the Newton polygon $\Delta$.

\begin{defn}
\ignore{
We say that $h:\Gamma^\circ\to\R^2$ is the {\em tropical limit}
of $\C C_\alpha$ if for every $p\in\R^2$ there exists
an open convex neighborhood $p\in U\subset\R^2$ such that for sufficiently large $t_\alpha$
there is a 1-1 correspondence between
the ends of the Riemann surface
$\Log_{t_\alpha}(U)\cap\C C_\alpha$
and
the ends of the open graph
$h^{-1}(U)$
with the following property.

The homology class in $H_1(\ctor;\Z)=\Z^2$ of the loop going clockwise around an end of
$\Log_{t_\alpha}(U)\cap\C C_\alpha$
coincides with $dh(u_v)\in\Z^2$, where $u_v$ is the unit tangent vector
to the corresponding leaf of the open graph $h^{-1}(U)$.
}
We say that a family $\C C_\alpha$ has a {\em phase-tropical limit}
with respect to $t_\alpha$ if
for every $p\in \R^2$ we have
\begin{equation}
\lim\limits_{t_\alpha\to+\infty} t_\alpha^{-p}\C C_\alpha=\Phi_p
\end{equation}
for a (possibly empty) algebraic curve $\Phi(p)\subset\ctor$.
Here $t_\alpha^{-p}\C C_\alpha$ is the multiplicative translation
of the curve $\C C_\alpha$ by $t^{-p}_\alpha\in\ctor$.
The coefficients of the polynomials defining $t_\alpha^{-p}\C C_\alpha$
represent a point in the projective space of dimension $\#(\Delta\cap\Z^2)-1$.
The limit is understood
in the sense of topology of this projective space.
The curve $\Phi_p\subset\ctor$ may be reducible and even non-reduced.

We say that $h:\Gamma^\circ\to\R^2$ is {\em the tropical limit} of
$\C C_\alpha$ with respect to $t_\alpha$
if for a sufficiently small open convex neighborhood $p\in U\subset\R^2$
the irreducible components $\Psi\subset\Phi(p)\subset\ctor$
correspond to the connected components $\psi\subset h^{-1}(U)$ so that
the lattice polygon $\Delta_\psi$ determined by the ends of the open graph $\psi$
coincides with the Newton polygon $\Delta_{\Psi}$ of
the irreducible component $\Psi$ taken with some multiplicity.
The same component $\Psi$ may correspond to several components of $h^{-1}(U)$
so that the some of all resulting multiplicities is equal to the multiplicity of $\Psi$ in $\Phi(p)$.
Each connected component of $h^{-1}(U)$ corresponds to a unique component of $\Phi(p)$.

If $h$ does not contract any edge of $\Gamma^\circ$ to a point then
the open set $\psi\subset\Gamma^\circ$ may contain at most one vertex.
If $v\in\Gamma$ is such a vertex then we call $\Psi$ {\em the phase $\Phi_v$ of the vertex $v$}.
If $\psi$ is contained in an edge $E$ then we call $\Psi$ {\em the phase $\Phi_E$
of the edge $E$}.
The phases $\Phi(E)\subset\ctor$ do not depend on the choice of a point $p\in h(E)$ and are well-defined
up to multiplicative translations by $\rtor$.
The curve $h:\Gamma^\circ\to\R^2$ enhanced with the phases $\Phi_v$ and $\Phi_E$
for its vertices and edges is called {\em the phase-tropical limit of $\C C_\alpha$
with respect to the scaling sequence $t_\alpha\to+\infty$}.

We consider the phases in $\ctor$ that are different by multiplicative translation
by vectors from $(\R_{>0})^2$ equivalent.
\end{defn}

Note that the Newton polygon of the phase $\Phi_E$ of an edge $E$
is an interval. Thus after a suitable change of coordinates in $\ctor$
the (irreducible) curve $\Phi_E$ is given by a linear equation in one variable.
Therefore, $\Phi_E$ is a multiplicative translation of a subtorus
$S^1\approx T_E\subset S^1\times S^1$ in the direction parallel to $h(E)$.

Let us orient $E$. Then $T_E$ as well as the quotient space $B_E=(S^1\times S^1)/T_E$
also acquire an orientation.
The image $\Arg(\Phi_E)$ coincides with $\pi_E^{-1}(\sigma_E)$ for some $\sigma_E\in B_E$,
where $\pi_E : S^1\times S^1\to B_E$ is the projection.
Since $B_E$ is isomorphic to $S^1$ and oriented, we have a canonical
isomorphism $B_E=\Z/2\pi\Z$. Thus, a phase $\Phi_E$ of an oriented edge $E$
of a planar tropical curve is determined by a single argument
$\sigma(E)\in\Z/2\pi\Z$. The change of the orientation of $E$
results in the change of sign of $\sigma(E)$.

Let $v\in\Gamma^\circ$ be a vertex and $E_j$ be the edges adjacent to $v$.
Orient $E_j$ outwards from $v$.
The oriented edges $E_j$ can be associated a {\em momentum} $\mo(E_j)$
with respect to the origin $0\in\R^2$. This is the wedge product of
the vector connecting the origin with a point of $E_j$ and the unit tangent vector
$u(E_j)$ coherent with the orientation. Clearly, it does not depend on the choice
of the point in $E_j$.

Recall that the vertex $v$ is dual to the lattice polygon $\Delta_v$
determined by the integer vectors $dh(u(E_j))$.
The multiplicity is defined as $m(v)=2\Area\Delta_v$, cf. \cite{Mi05}.
\begin{prop}[tropical Menelaus theorem]\label{tropMenelaus}
For any tropical curve $h:\Gamma^\circ\to\R^2$ and
a vertex $v\in\Gamma^\circ$ the momenta
$\mo(E_j)$ of the edges adjacent to $v$ and oriented outwards from $v$
satisfy to the equality
\begin{equation}\label{mo0}
\sum\limits_j \mo(E_j)=0.
\end{equation}

If $\sigma(E_j)\in \Z/2\pi\Z$ are phases of the oriented edges $E_j$
then
\begin{equation}\label{si}
\sum\limits_j w(E_j)\sigma(E_j)=\pi m(v)
\end{equation}
(assuming that $\sigma(E_j)$ appear in the phase-tropical limit
of a family $\C C_\alpha\subset\ctor$ of complex curves).
\end{prop}
This statement can be viewed as a counterpart
of the ancient Menelaus theorem
(before its generalizations by Carnot and Weil)
stating that three points $D,E,F$ on the extensions of three sides of
a planar triangle $ABC$ are collinear if and only if
\begin{equation}\label{classicalMenelaus}
\frac{|AD|}{|DB|}\frac{|BE|}{|EC|}\frac{|CF|}{|FA|}=-1.
\end{equation}
Here the length is taken with the minus sign if the direction of an interval
(e.g. $|CF|$) is opposite to the orientation of the triangle, see
Figure \ref{FigMenelaus}.
\begin{figure}[h]
\includegraphics[height=45mm]{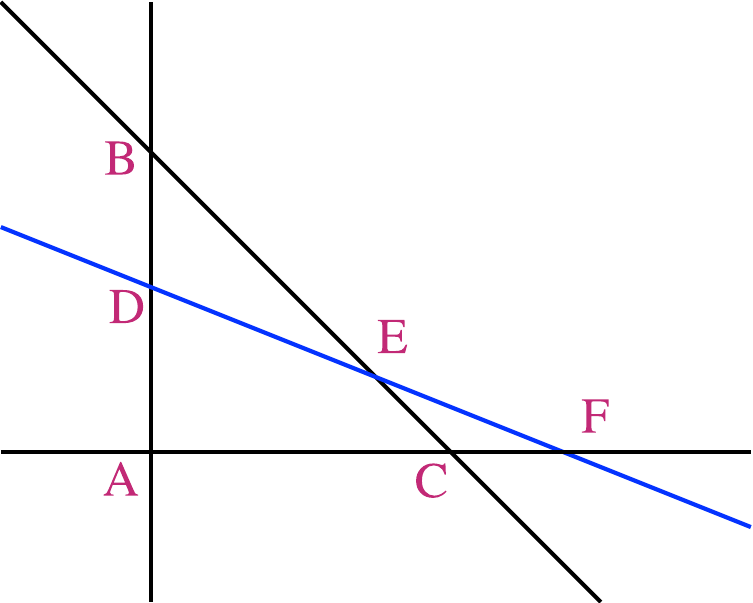}
\caption{
The Menelaus theorem.
\label{FigMenelaus}}
\end{figure}
\begin{proof}
The wedge product of the balancing condition \eqref{bc} with the
vector connecting $0$ and $v$ gives \eqref{mo0}.
To deduce \eqref{si} we consider the polynomial $f_v$ (whose Newton
polygon is $\Delta_v$) defining the phase $\Phi_v\subset\ctor$.
By Vieta's theorem, the product of the roots
cut by $f_v$ on a divisor of $\C\Delta_v$ corresponding to an oriented
side $F\subset\Delta_v$ is $(-1)^{\#(F\cap\Z^2)}$ times the ratio of the coefficients
at the endpoints of $F$. Therefore the sum of the phases of the edges of $\Gamma$
corresponding to $F$ is the argument of this ratio plus $\#(F\cap\Z^2)\pi$.
Since by Pick's formula the parity of $\#(\dd\Delta\cap \Z^2)$ coincides with
that of $m(v)=2\Area(\Delta_v)$ we recover \eqref{si}.
\end{proof}
\begin{coro}\label{coroMenelaus}
We have $\sum\limits_E\mo(E)=0,$
where the sum is taken over all leaves of $h:\Gamma^\circ\to\R^2$
oriented in the outwards direction.
\end{coro}
\begin{proof}
Take the sum of the expression \eqref{mo0} over all vertices of $\Gamma^\circ$.
The momenta of all bounded edges will enter twice with the opposite signs.
\end{proof}

If all curves $\C C_\alpha$ are defined over $\R$ then the phases $\Phi(p)$
must be real for all points $p\in\R^2$. Note, however that in general,
the phase $\Phi_v$ for a vertex $v\in\Gamma^\circ$
does not have to be real as the involution of
complex conjugation may exchange it with $\Phi_{v'}$ for another vertex
$v'\in\Gamma$ with $h(v)=h(v')$.
We say that a vertex $v$ is {\em real} if $\Phi_v$ is defined over $\R$.

Let $\R C_\alpha$ be a scaled sequence of type I curves enhanced
with a complex orientation,
so that a component $S_\alpha\subset\C C_\alpha\setminus\R C_\alpha$ is fixed for
all $\alpha$. Suppose that $\C C_\alpha$ has a phase-tropical limit, and
the orientations of $\R C_\alpha$ agree with some complex
orientations of the real part $\R\Phi(p)$ of the phases $\Phi(p)$.
The quantum index of $\R C_\alpha$ is well-defined if it has real or purely imaginary
coordinate intersection.
Similarly, the phase $\R\Phi_v$ of a real vertex $v$ of the tropical limit
has a well-defined quantum index
if $\sigma(E)\equiv 0\pmod\pi$ for any edge $E$ adjacent to $v$.
\begin{prop}\label{prop-ksum}
For large $t_\alpha$ we have
\begin{equation}
\label{ksum}
k(\R C_\alpha)=\sum\limits_v k(\R\Phi_v),
\end{equation}
where the sum is taken over all real vertices whenever all quantum indices in \eqref{ksum} are well-defined.
\end{prop}
\begin{proof}
Additivity of the quantum index with respect to the phases $\Phi_v$ follows
from Theorem \ref{thm-larea} through additivity of the degree of the map $2\Arg$ restricted
to $S\cap\ctor$. Non-real vertices have zero contribution to $k(\R C_\alpha)$ as the signed
area of the amoeba of the whole complex curve is zero.
\end{proof}

\ignore{

Recall that a {\em ribbon structure} of a graph is a choice of a cyclic
orientation of adjacent edges at all its vertices.
Any ribbon structure on $\Gamma$ defines a compact surface with boundary $S$
and a retraction $\rho:S\to\Gamma$. The surface $S$ is constructed

\begin{defn}
A real phase-tropical curve is a metric graph enhanced with a ribbon structure
\end{defn}
}
\begin{proof}[Proof of Theorem \ref{BG}]
Recall the definition of the (tropical) Block-G\"ottsche invariants, see \cite{ItMi},
which refine tropical enumerative invariants of \cite{Mi05}.
Namely, to any 3-valent (open) tropical immersed curve $h:\Gamma^\circ\to\R^2$
we may associate the Laurent polynomial
\begin{equation}\label{BGdef}
n_q(h(\Gamma^\circ))=\prod\limits_v \frac{q^{\frac{m(v)}2}-q^{-\frac{m(v)}2}}{q^{\frac 12}-q^{-\frac 12}},
\end{equation}
where $v$ runs over all vertices $v\in\Gamma$
and $m(v)$ is the multiplicity of the vertex $v$.
The genus of a (connected) tropical curve $\Gamma^\circ$ is
the first Betti number of $\Gamma^\circ$.
In particular, a rational tropical curve is a tree.

\ignore{
Recall that a line $L\subset\R^2$ enhanced with a vector $u$ parallel to $L$
can be associated a {\em momentum} 
with respect to the origin $0\in\R^2$. This is the wedge product of
the vector connecting the origin with a point of $L$ and $u$.
Accordingly, each unbounded edge $E\subset\Gamma$ has a momentum $\mo(E)$
where we take the line
containing $h(E)$ for $L$ and the image of the unit tangent vector to $E$ directed
towards infinity under $dh$ for $u$.
The absolute value part of the tropical Menelaus
may be stated as the following proposition, cf. \cite{Ueda}.
\begin{prop}[tropical Menelaus theorem, the absolute value part]
We have
\begin{equation}\label{trMenelaus}
\sum\limits_E\mo(E)=0,
\end{equation}
where the sum is taken over all unbounded edges $E$ of $\Gamma$.
\end{prop}
\begin{proof}
The balancing condition, see \cite{Mi05}, ensures that the sum of the images
of all unite vectors adjacent to the same vertex under $dh$ is zero as
long as all curves are oriented outwards from the vertex.
Therefore, \eqref{trMenelaus} holds if $\Gamma$ consists of a single vertex.
In the general case we take the sum of the corresponding expression
over all vertices of $\Gamma$.
All bounded edges have the zero contribution to this sum as they enter twice
with the opposite sign.
\end{proof}
}

Let us fix
a collection $\mu=\{m_j\}_{j=1}^m$, $m=\#(\dd\Delta\cap\Z^2)$,
of generic real numbers
subject to
the condition
$\sum\limits_{j=1}^m\mu_j=0$. This means that $\mu_j$, $j=1,\dots,m-1$ are
chosen generically, and $\mu_m$ is determined from our condition.

If $h:\Gamma^\circ\to\R^2$ is a tropical curve with the Newton polygon $\Delta$ then
we number its leaves so that the first $m_1$ leaves
are dual to the side $E\subset\dd\Delta$, the second $m_2$ to the side
$E_2\subset\dd\Delta$ and so on with the last $m_n$ leaves dual to $E_n$.
We say that $h:\Gamma\to\R^2$
passes through the $\dd\T\Delta$-points determined by $\mu$
if the $j$th unbounded edge of $\Gamma$ has the momentum $\mu_j$.
\ignore{
By Proposition 2.23 of \cite{Mi05} all immersed
rational tropical curves with the Newton
polygon $\Delta$ form a polyhedral space of dimension $m-1$. All facets of this
space are formed by 3-valent curves.
Thus any tropical rational curve with the Newton polygon $\Delta$ passing through
the $\dd\T\Delta$-points
determined by $\mu$ must be 3-valent and immersed (as otherwise
we may find another tropical rational curve with the same image but
with vertices of higher valence).
}
Note that
a leaf $E\subset\Gamma^\circ$ must have the momentum $\mo(E)$
if it passes through a point $p_E$ on the oriented line parallel
to the vector $(dh)u(E)$ with the momentum $\mo(E)$.
Thus a generic choice of the momenta ensures that $h:\Gamma^\circ\to\R^2$
passes through a generic collection of $m-1$ points in $\R^2$.
Thus we have only finitely many rational tropical curves with the Newton polygon $\Delta$
passing through the $\dd\T\Delta$ points determined by $\mu$ by Lemma 4.22 of \cite{Mi05}
(as the number of combinatorial types of tropical curves with the given Newton polygon $\Delta$
is finite). By Proposition 4.11 of \cite{Mi05} all these tropical curves are simple in the sense
of Definition 4.2 of \cite{Mi05}.

{\em The Block-G\"ottsche number associated to $\mu$} is
\begin{equation}\label{BGDeltadef}
\ntropm_\Delta=\ntropm_\Delta(\mu)=\sum\limits_{h:\Gamma^\circ\to\R^2}
n_q(h(\Gamma^\circ)),
\end{equation}
where the sum is taken over all $h:\Gamma^\circ\to\R^2$ passing
through the $\dd\T\Delta$ points determined by $\mu$.
Independence of $\ntropm_\Delta$ from $\mu$ can be proved in the same
way as in \cite{ItMi}. Also it follows from Theorem \ref{Rinv}
once we prove coincidence of $R_{\Delta}$ and $\ntropm_{\Delta}(\mu)$.

A toric divisor $\C E_j\subset\C\Delta$ is the compactification of the torus $\C^\times$
obtained by taking the quotient group of $\ctor$ by the subgroup defined by the side $E_j\subset\Delta$.
Thus a configuration $\PP=\{p_j\}_{j=1}^m\subset\dd\C\Delta$
is given by a collection of $m$ nonzero complex numbers as well as an attribution of the points to
the toric divisors. This collection is real if the corresponding numbers are real
and positive if these number are positive.

We set $\PP^t=\{p_1^t,\dots,p_m^t\}\subset\dd\R\Delta$ be the configuration of points with the
same toric divisor attribution as $\PP$, and given by the positive numbers $\{t^{2\mu_j}\}$, $t>1$.
By Proposition 8.7 of \cite{Mi05} the amoebas of rational complex curves with the Newton polygon $\Delta$
passing through $(\Fr^\Delta)^{-1}(\PP_t)$ converge when $t\to+\infty$
to tropical curves passing through the $\dd\T\Delta$-points determined by $\mu$.
Proposition 8.23 of \cite{Mi05} determines the number of complex curves with amoeba in
a small neighborhood of a rational tropical curve $h:\Gamma^\circ\to\R^2$ passing through
any choice of points $\tilde p_j^t\in (\Fr^\Delta)^{-1}(p_j^t)$, $j=1,\dots,m-1$, for large $t$,
while Remark 8.25 of \cite{Mi05}
determines the number of the corresponding real curves.
E.g. if the weights of all edges of $\Gamma^\circ$ are odd we have a single real curve for
any choice of $\tilde\PP^t=\{\tilde p_j^t\}$. In general, some choices of $\tilde\PP_t$
may correspond to no real solutions, while others may correspond to multiple solutions.
We claim that nevertheless
there are $2^{m-1}$ different real curves whose amoeba is close to $h:\Gamma^\circ\to\R^2$
with the image under $\Fr^\Delta$ passing through $\PP^t$ for large $t$.
Thus we have $2^m$ different oriented curves. We show this by induction on $m$ as follows.

If $\Gamma^\circ$ has a single vertex $v$ (so that $m=3$)
then there are 4 different real rational phase $\Phi_v$
which differ by the deck transformations of the map $\Fr^\Delta$.
Thus we have 8 different oriented real rational phases in this case.
The positive logarithmic rotation number for half of them
is positive, for the other half is negative.
Adding each new 3-valent vertex $v'$ to the tree $\Gamma$ doubles
the number of oriented real phases as there are two ways to
attach the phase for $v'$: so that the logarithmic rotation number
of the resulting real curve will increase by one
and so that it will decrease by one.
\ignore{
The phases of a simple tropical curve with leaves of weight 1
are determined by the phases $\sigma(E)\in\Z/2\pi\Z$ of its leaves.
The symmetry group $\Z_2^2$ acts on the collections of such phases by reflections in
$\rtor\subset\R \Delta$. For each leaf $E$ there is a unique non-trivial reflection
$I_E:\rtor\to\rtor$ preserving $\Phi_E$.
}
\ignore{
Let $E\subset\Gamma^\circ$ be a bounded edge and $\sigma(E)$ be a choice for its real phase
(equal to 0 or $\pi$).
Each connected component of $\Gamma^\circ\setminus E$ determines a tropical curve $h':\Gamma'\to\R^2$
by extending $E$ to a leaf.
The group $\Z^2_2$ acting on $\rtor$ by reflection
of the coordinates contains the unique non-trivial
reflection that preserves the real phase of $E$
(this reflection is determined by the parity
of the slope of $E$).
Oriented real curves close to $h'$ come in pairs that are different by
the reflection $I_E$.
}
Inductively we get 4 real oriented curves
for each of the $2^{m-2}$ possible distribution of signs for
the vertices of $\Gamma^\circ$.

For each vertex $v$
the real phase $\R\Phi_v$ 
is the image of a line by
a multiplicative-linear map of determinant $m(v)$ by Corollary 8.20 of \cite{Mi05}.
Therefore
$k(\R\Phi_v)=\pm\frac{m(v)}2$,
where the sign is determined by the degree
of the logarithmic Gau{\ss} map.
According to our sign convention \eqref{sigma} each oriented real curve comes with the sign
equal to the number of negative vertices.
Thus by Proposition \ref{prop-ksum} the contribution of $h:\Gamma\to\R^2$ to $R_\Delta(\PP^t)$ for large $t$
is $\prod\limits_v(q^{\frac m2}-q^{-\frac m2})$ which coincides with the numerator of
the Block-G\"ottsche
multiplicity \eqref{BGdef}.
\end{proof}

%% file: main-AM.bbl
\def\cprime{$'$}
\begin{thebibliography}{10}

\bibitem{BlGo}
Florian Block and Lothar G\"ottsche.
\newblock Refined curve counting with tropical geometry.
\newblock arXiv:1407.2901.

\bibitem{Br}
Erwan Brugall\'e.
\newblock Pseudoholomorphic simple {H}arnack curves.
\newblock arXiv:1410.2423.

\bibitem{Ca}
Lucia Caporaso.
\newblock Algebraic and tropical curves: comparing their moduli spaces.
\newblock arXiv:1101.4821.

\bibitem{Carnot}
L.~N.~M Carnot.
\newblock {\em G\'eom\'etrie de position}.
\newblock Imprimerie de Crapelet, Paris, 1803.

\bibitem{Mumford}
P.~Deligne and D.~Mumford.
\newblock The irreducibility of the space of curves of given genus.
\newblock {\em Inst. Hautes \'Etudes Sci. Publ. Math.}, (36):75--109, 1969.

\bibitem{FoPaTs}
Mikael Forsberg, Mikael Passare, and August Tsikh.
\newblock Laurent determinants and arrangements of hyperplane amoebas.
\newblock {\em Adv. Math.}, 151(1):45--70, 2000.

\bibitem{GoSh}
Lothar G{\"o}ttsche and Vivek Shende.
\newblock Refined curve counting on complex surfaces.
\newblock {\em Geom. Topol.}, 18(4):2245--2307, 2014.

\bibitem{GrHa}
Phillip Griffiths and Joseph Harris.
\newblock {\em Principles of algebraic geometry}.
\newblock Wiley-Interscience [John Wiley \& Sons], New York, 1978.
\newblock Pure and Applied Mathematics.

\bibitem{IKS1}
Ilia Itenberg, Viatcheslav Kharlamov, and Eugenii Shustin.
\newblock Welschinger invariant and enumeration of real rational curves.
\newblock {\em Int. Math. Res. Not.}, (49):2639--2653, 2003.

\bibitem{ItMi}
Ilia Itenberg and Grigory Mikhalkin.
\newblock On {B}lock-{G}\"ottsche multiplicities for planar tropical curves.
\newblock {\em Int. Math. Res. Not. IMRN}, (23):5289--5320, 2013.

\bibitem{Ka-Gauss}
M.~M. Kapranov.
\newblock A characterization of {$A$}-discriminantal hypersurfaces in terms of
  the logarithmic {G}auss map.
\newblock {\em Math. Ann.}, 290(2):277--285, 1991.

\bibitem{KeOkSh}
Richard Kenyon, Andrei Okounkov, and Scott Sheffield.
\newblock Dimers and amoebae.
\newblock {\em Ann. of Math. (2)}, 163(3):1019--1056, 2006.

\bibitem{KoSo}
Maxim Kontsevich and Yan Soibelman.
\newblock Stability structures, motivic {D}onaldson-{T}homas invariants and
  cluster transformations.
\newblock arXiv:0811.2435.

\bibitem{Kou}
A.~G. Kouchnirenko.
\newblock Newton polyhedra and {B}ezout's theorem.
\newblock {\em Funkcional. Anal. i Prilo\v zen.}, 10(3, 82--83.), 1976.

\bibitem{Mi14}
Grigory Mikhalkin.
\newblock Amoebas of half-dimensional varieties.
\newblock arXiv:1412.4658.

\bibitem{Mi00}
Grigory Mikhalkin.
\newblock Real algebraic curves, the moment map and amoebas.
\newblock {\em Ann. of Math. (2)}, 151(1):309--326, 2000.

\bibitem{Mi05}
Grigory Mikhalkin.
\newblock Enumerative tropical algebraic geometry in {$\Bbb R^2$}.
\newblock {\em J. Amer. Math. Soc.}, 18(2):313--377, 2005.

\bibitem{Mi06}
Grigory Mikhalkin.
\newblock Tropical geometry and its applications.
\newblock In {\em International {C}ongress of {M}athematicians. {V}ol. {II}},
  pages 827--852. Eur. Math. Soc., Z\"urich, 2006.

\bibitem{MiOk}
Grigory Mikhalkin and Andrei Okounkov.
\newblock Geometry of planar log-fronts.
\newblock {\em Mosc. Math. J.}, 7(3):507--531, 575, 2007.

\bibitem{MiRu}
Grigory Mikhalkin and Hans Rullg{\aa}rd.
\newblock Amoebas of maximal area.
\newblock {\em Internat. Math. Res. Notices}, (9):441--451, 2001.

\bibitem{NeOk}
Nikita Nekrasov and Andrei Okounkov.
\newblock Membranes and sheaves.
\newblock arXiv:1404.2323.

\bibitem{Pa-zeta}
Mikael Passare.
\newblock How to compute {$\sum 1/n^2$} by solving triangles.
\newblock {\em Amer. Math. Monthly}, 115(8):745--752, 2008.

\bibitem{PaRu}
Mikael Passare and Hans Rullg{\aa}rd.
\newblock Amoebas, {M}onge-{A}mp\`ere measures, and triangulations of the
  {N}ewton polytope.
\newblock {\em Duke Math. J.}, 121(3):481--507, 2004.

\bibitem{Rokhlin}
V.~A. Rohlin.
\newblock Complex orientation of real algebraic curves.
\newblock {\em Funkcional. Anal. i Prilo\v zen.}, 8(4):71--75, 1974.

\bibitem{Vi-fle}
O.~Ya. Viro.
\newblock Achievements in the topology of real algebraic varieties in the last
  six years.
\newblock {\em Uspekhi Mat. Nauk}, 41(3(249)):45--67, 240, 1986.

\bibitem{We}
Jean-Yves Welschinger.
\newblock Invariants of real symplectic 4-manifolds and lower bounds in real
  enumerative geometry.
\newblock {\em Invent. Math.}, 162(1):195--234, 2005.

\end{thebibliography}
